\documentclass[11pt]{article}

\usepackage{amsmath}
\usepackage{amssymb}
\usepackage{amsthm}
\usepackage{mathtools,nccmath}
\usepackage{tikz}
\usetikzlibrary{arrows.meta}

\usepackage{graphicx}

\usepackage[normalem]{ulem}

\usepackage[colorlinks,citecolor=blue,urlcolor=blue]{hyperref}

\usepackage{fullpage}

\usepackage{natbib}
\usepackage{algorithm}
\usepackage{algpseudocode}
\allowdisplaybreaks
\makeatother
\algblock{Input}{EndInput}
\algnotext{EndInput}
\algblock{Output}{EndOutput}
\algnotext{EndOutput}

\algnewcommand{\Initialize}[1]{%
  \State \textbf{Initialize:}
  \Statex \hspace*{\algorithmicindent}\parbox[t]{.8\linewidth}{\raggedright #1}
}

\usepackage{bm}
\usepackage{listings}
\usepackage{xcolor}

\newtheorem{lemma}{Lemma}
\newtheorem{theorem}{Theorem}

\newtheorem{definition}{Definition}
\newtheorem{proposition}{Proposition}

\newtheorem{corollary}{Corollary}
\newtheorem{assumption}{Assumption}

\usepackage[capitalize,noabbrev]{cleveref}
\title{Minimax Wasserstein Testing under Differential Privacy}
\author{
Bonwoo Lee \qquad Ilmun Kim \qquad Cheolwoo Park\\[0.5em]
Department of Mathematical Sciences\\
KAIST\\[0.7em]
\texttt{righthim@kaist.ac.kr},\ 
\texttt{ilmunk@kaist.ac.kr},\ 
\texttt{parkcw2021@kaist.ac.kr}
}
\date{}
\begin{document}

\maketitle
\begin{abstract}
We study Wasserstein-$p$ goodness-of-fit and two-sample testing under differential privacy. For general compact metric spaces, we establish upper and lower bounds on the minimax separation radii in terms of covering numbers, matching up to logarithmic factors. On the hypercube $[0,1]^d$, these bounds match up to constant factors for two-sample testing in all dimensions and for goodness-of-fit testing whenever $d\neq 4p$, thereby closing logarithmic gaps in existing non-private results. At the critical case $(d,p)=(4,1)$, we further show that a logarithmic factor is unavoidable by establishing matching minimax bounds for non-private Wasserstein-$1$ goodness-of-fit testing under a bounded-density assumption. Finally, for unbounded domains, we establish minimax-optimal separation rates up to logarithmic factors under an exponential tail condition.
\end{abstract}
\section{Introduction}
Testing whether observations are drawn from a prescribed distribution, or whether two samples are drawn from the same distribution, is a fundamental problem across many scientific disciplines. These problems are known, respectively, as goodness-of-fit testing and two-sample testing in the statistical literature. However, as concerns about data privacy grow, classical testing procedures may be unsuitable for applications involving sensitive individual-level data, such as clinical and demographic data. Accordingly, a growing body of work has developed differentially private hypothesis tests \citep[e.g.,][]{pmlr-v70-cai17a,NEURIPS2020_72b32a1f,NEURIPS2018_7de32147,Kim25032026}. This line of work has shown that privacy constraints impose fundamental statistical limits, making testing intrinsically more difficult than in the non-private setting.

Most existing minimax results for differentially private testing, however, concern alternatives separated in $\ell_1$ or $\ell_2$ distance. These distances compare probability masses coordinatewise without accounting for the geometry of the underlying data space. In particular, they do not distinguish between transporting probability mass to a nearby point and transporting it to a distant point. The Wasserstein distance provides a natural alternative by explicitly incorporating the cost of transporting probability mass across the data space, and has consequently been widely used in statistics, machine learning, and computer science  \citep[e.g.,][]{panaretos2019statistical,word_moving_distance,improved_WGANs}. Hypothesis tests based on the Wasserstein distance have been studied in the non-private setting \citep{delBarrio2000Contributions,ba2011sublinear,Ramdas2017,Wasserstein_CI}, but research on Wasserstein testing under privacy constraints remains limited. While \citet{AwanWang2025} proposed differentially private tests based on one-dimensional Wasserstein statistics, their applicability is inherently limited to one-dimensional data. Many applications of the Wasserstein distance involve multivariate, structured, or non-Euclidean data, for which a general metric-space formulation is more natural and broadly applicable. Beyond this broader applicability, such a formulation is also theoretically informative, as it allows us to identify the geometric quantities that govern the difficulty of Wasserstein testing. At the same time, because privacy constraints fundamentally alter the statistical difficulty of hypothesis testing, it is important to quantify the smallest separation that can be reliably detected under a given privacy level. This motivates a minimax analysis of differentially private Wasserstein testing, which characterizes the fundamental statistical price of privacy and its dependence on the geometry of the underlying space.

\subsection{Our contributions} In this work, we study differentially private goodness-of-fit and two-sample testing under Wasserstein separation on general compact metric spaces. To the best of our knowledge, this provides the first minimax analysis of differentially private Wasserstein testing beyond one-dimensional domains. We begin by showing that, over unbounded domains, Wasserstein testing is infeasible without additional structural assumptions on the underlying distributions. For the upper bounds, we develop differentially private Wasserstein tests based on multiresolution partitions of the metric space. At each resolution, the problem is reduced to finite-domain $\epsilon$-differentially private (DP) $\ell_1$ testing, with the resulting tests combined through a Bonferroni procedure. A key distinction from existing multiresolution approaches \citep[e.g.,][]{ba2011sublinear,deng2017wasserstein} is that we allocate the privacy and test error budgets nonuniformly across resolutions according to the difficulty of the corresponding $\ell_1$ testing problems. This more efficient allocation leads to sharper separation rates. For the lower bounds, we embed hard finite-domain testing instances into well-separated subsets of the metric space. Together, these arguments identify the covering numbers of the underlying space as the geometric quantities governing the difficulty of Wasserstein testing, and yield nearly matching upper and lower bounds for the minimax separation radii up to logarithmic factors.

We then specialize our results to $[0,1]^d$ equipped with the $\ell_1$ metric, where the benefit of the nonuniform budget allocation becomes explicit. For goodness-of-fit testing, we remove the logarithmic factors from the minimax separation radius except in the critical case $d=4p$, thereby closing the logarithmic gap in the non-private Wasserstein goodness-of-fit bounds of \citet{deng2017wasserstein}. Additionally, we show that the logarithmic factor is inevitable in the case $(d,p)=(4,1)$. Under the additional assumption that the data-generating distribution has a uniformly bounded density, we establish the minimax separation radius up to a constant factor. In particular, our lower bound captures the critical logarithmic factor that is absent from previous minimax lower bounds for Wasserstein goodness-of-fit testing. The proof combines a multiscale construction with a sharp analytical lower bound on the Wasserstein-1 distance in terms of negative-order Sobolev quantities, allowing discrepancies across multiple resolution levels to accumulate simultaneously. This establishes that the logarithmic factor in the critical case is intrinsic rather than an artifact of the upper-bound analysis and yields a tight minimax characterization up to a constant factor. For two-sample testing, we remove the logarithmic factors entirely, closing the corresponding logarithmic gap in the non-private sample-complexity bounds of \citet{ba2011sublinear}. Thus, in addition to characterizing the cost of differential privacy, our analysis sharpens existing minimax results for non-private Wasserstein testing. Finally, we return to unbounded domains and show that, under an exponential-tail condition on the data distributions, the corresponding separation rates continue to hold up to logarithmic factors.

\subsection{Related work}
\paragraph{Differentially private hypothesis testing.} Early work on differentially private hypothesis testing developed private analogues of classical tests and studied their asymptotic properties \citep{5360513,uhler2013privacy,raj2019differentially}. A subsequent line of work adopted a non-asymptotic perspective, aiming to characterize the trade-off between privacy and statistical power through finite-sample or minimax guarantees. Our work falls into this latter category. \citet{pmlr-v70-cai17a} initiated the study of private goodness-of-fit testing for distributions over a finite domain, providing sample-complexity guarantees under $\ell_1$ separation. \citet{NEURIPS2018_7de32147} subsequently established the optimal sample complexity of private goodness-of-fit testing across all parameter regimes and obtained matching upper and lower bounds for private two-sample testing in the sparse-sample regime. Concurrently, \citet{pmlr-v80-aliakbarpour18a} developed related private identity and two-sample tests, with particular emphasis on the sublinear-sample regime. \citet{zhang2021statistical} subsequently completed the characterization of private two-sample testing over finite domains, establishing upper and lower bounds that match up to constant factors across all parameter regimes. For structured high-dimensional families, \citet{NEURIPS2020_72b32a1f} and \citet{narayanan22a} studied private goodness-of-fit testing for Boolean product distributions and multivariate Gaussian distributions with known covariance. Differentially private two-sample testing for continuous distributions on the real line was studied by \citet{NEURIPS2024_04cc90ec}. In related work, \citet{kim2026sharp} developed differentially private two-sample tests using permutation calibration. Their tests achieve minimax-optimal rates for H\"older-smooth densities and distributions over finite domains. Much of this literature, however, focuses on separation in classical distances such as $\ell_1$ or $\ell_2$, often on specific domains or within restricted distribution classes. Several works have also developed general private testing frameworks. \citet{10.1145/3313276.3316336} characterized the optimal sample complexity of private simple hypothesis testing over general sample spaces. More recently, \citet{Kim25032026} developed differentially private permutation tests and applied them to obtain minimax-optimal private kernel two-sample and independence tests over general domains. Most directly related to our use of the Wasserstein distance, \citet{AwanWang2025} proposed a private goodness-of-fit test based on a one-dimensional Wasserstein statistic. Their work focuses on the validity and empirical power of the proposed test and does not characterize minimax separation rates.

\paragraph{Wasserstein hypothesis testing.}
Given the widespread use of the Wasserstein distance in statistics, machine learning, and computer science, a substantial literature has developed hypothesis tests based on this distance. Early work studied procedures for comparing distributions using trimmed Wasserstein distances and established their asymptotic properties \citep{10.1111/1467-9868.00121,FREITAG2007697,Alvarez-Esteban01062008,91101a87-ac1a-38eb-b389-231b14fc1446}. Wasserstein goodness-of-fit tests and their asymptotic behavior were investigated by \citet{10.1214/aos/1017938923,delBarrio2000Contributions,10.3150/bj/1110228245}. More recently, projection-based variants have been introduced to improve the performance and computational tractability of Wasserstein tests in high dimensions; for example, \citet{10.1093/biomet/asaf001} proposed a two-sample test based on the sparse max-sliced Wasserstein distance.

Non-asymptotic minimax guarantees for Wasserstein-based testing remain comparatively scarce. \citet{ba2011sublinear} derived sample-complexity upper and lower bounds for two-sample testing on subsets of $[0,\Delta]^d$ equipped with the $\ell_1$ metric based on Wasserstein-1 distance. The bounds match up to logarithmic factors for $d\leq 2$ and $d\geq 6$, while leaving a gap for $d=3,4,5$. The gap was closed by \cite{Ferreira2024Distribution} (see Lemma 3.12). \cite{deng2017wasserstein} also studied Wasserstein goodness-of-fit testing using Wasserstein-1 distance for a general compact metric space. They showed matching upper and lower bounds for goodness-of-fit testing on general compact metric spaces, up to logarithmic factors. None of these works considers privacy constraints and all leave logarithmic factors in the bounds for Wasserstein testing on $[0,1]^d$. In contrast, we explore testing under differential privacy constraints using Wasserstein-$p$ distance for general $p\geq1$, and remove the logarithmic gaps in most dimensions through a careful allocation of testing budgets. We further show that a logarithmic factor is unavoidable in at least one critical case, namely $(d,p)=(4,1)$. To the best of our knowledge, this is the first minimax lower bound exhibiting an unavoidable logarithmic factor for Wasserstein goodness-of-fit testing. Additionally, \citet{Wasserstein_CI} characterized the minimax separation rate for conditional independence testing under Wasserstein separation up to logarithmic factors. 

\subsection{Notation} For $a,b \in \mathbb{R}$, $a\wedge b$ and $a\vee b$ denote the minimum and maximum of $\{a,b\}$, respectively. We write $a\lesssim b$ if there exists a universal constant $c>0$ such that $a\leq cb$. If both $a\lesssim b$ and $b\lesssim a$ hold, we write $a\asymp b$. Let $(\mathcal{Y},d)$ be a metric space $\mathcal{Y}$ equipped with a metric $d$. We use $\mathcal{B}(\mathcal{Y})$ to denote the Borel $\sigma$-algebra of $\mathcal{Y}$ and $\mathcal{P}(\mathcal{Y})$ to denote the collection of all probability measures on the measurable space $(\mathcal{Y},\mathcal{B}(\mathcal{Y}))$. For a subset $S\subset \mathcal{Y}$, we write $d(S,x)\coloneqq \inf_{s\in S}d(s,x)$. For a set $A$, we denote its cardinality by $|A|$. The total variation distance between probability distributions $P$ and $Q$ is given by $\lVert P-Q\rVert_{\mathrm{TV}}\coloneqq \sup_{A}|P(A)-Q(A)|$. Finally, we write $[k]\coloneqq \{1,\ldots,k\}$ for the set of positive integers not exceeding $k$.

\subsection{Organization} The remainder of this paper is organized as follows. Section~\ref{sec:preliminaries} introduces the basic definitions and the minimax formulations of private goodness-of-fit and two-sample testing. Section~\ref{sec:main} establishes minimax bounds on compact metric spaces in terms of covering numbers, together with an impossibility result for unrestricted distributions on unbounded spaces. Section~\ref{sec:hypercube} derives explicit rates on $([0,1]^d,\ell_1)$ and establishes sharp nonprivate bounds under a bounded density assumption, demonstrating the necessity of a logarithmic factor when $(d,p)=(4,1)$. Section~\ref{sec:unbounded} extends the analysis to distributions with exponentially decaying tails, and Section~\ref{sec:discussion} discusses open questions. All technical proofs omitted from the main text are provided in the appendix.

\section{Preliminaries} \label{sec:preliminaries}

In this section, we introduce the basic concepts and definitions used throughout the paper, including the metric space, Wasserstein distance, differential privacy, and the minimax separation radius framework for goodness-of-fit and two-sample testing.
\subsection{Metric spaces and covering numbers}
We assume that observations take values in a Polish metric space $(\mathcal{X},d)$, that is, a complete and separable metric space. Our private Wasserstein tests, introduced later, combine multiple tests over partitions of $\mathcal{X}$ at different resolutions. Therefore, their performance depends on the size of the partition at each resolution, which can be quantified by the covering number of $\mathcal{X}$.

\begin{definition}[Covering number]
For $\delta>0$, the $\delta$-covering number of $\mathcal{X}$, denoted by $\mathcal{N}(\mathcal{X},\delta)$, is the smallest number of balls of diameter at most $\delta$ needed to cover $\mathcal{X}$. If no such finite collection exists, we set $\mathcal{N}(\mathcal{X},\delta)=\infty$.
\end{definition}
The covering number quantifies the metric complexity of $\mathcal{X}$ at resolution $\delta$. In our multiresolution testing procedure, it controls the number of cells required at each resolution and, consequently, the difficulty of the induced finite-domain testing problem.
\subsection{Wasserstein distance}
A coupling of two probability measures $P$ and $Q$ on $\mathcal{X}$ is a probability measure on $\mathcal{X}\times\mathcal{X}$ with marginals $P$ and $Q$, respectively.
\begin{definition}[Wasserstein-$p$ distance]\label{def:Wasserstein}
Let $p\geq1$, and let $P$ and $Q$ be probability measures on $(\mathcal{X},\mathcal{B}(\mathcal{X}))$. Let $\Pi(P,Q)$ denote the set of all couplings of $P$ and $Q$. The Wasserstein-$p$ distance between $P$ and $Q$ is defined by
\begin{align*}
W_p(P,Q)
\coloneqq 
\inf_{\pi\in\Pi(P,Q)}
\left(
\int_{\mathcal{X}\times\mathcal{X}} d^p(x,y) d\pi(x,y)
\right)^{1/p}.
\end{align*}
The value $W_p(P,Q)$ is allowed to be infinite.
\end{definition}

For discrete $\mathcal{X}$, the Wasserstein distance has a natural interpretation in terms of mass transport. A coupling $\pi\in\Pi(P,Q)$ specifies how much probability mass is transported from each point $x\in\mathcal{X}$ under $P$ to each point $y\in\mathcal{X}$ under $Q$. The quantity $d^p(x,y)$ is the cost of transporting one unit of mass from $x$ to $y$. Consequently, $W_p^p(P,Q)$ is the minimum total transportation cost over all couplings, and $W_p(P,Q)$ is its $p$th root.

\subsection{Differential privacy}
For $N\geq 1$, two datasets $D=(x_1,\ldots,x_N)$ and $D'=(x'_1,\ldots,x'_N)$ in $\mathcal{X}^N$ are said to be neighboring if they differ in exactly one coordinate; that is, there exists $i\in\{1,\ldots,N\}$ such that $x_j=x'_j$ for every $j\neq i$ and $x_i\neq x'_i$.

\begin{definition}[{$\epsilon$-differential privacy; \citealp{dwork2006calibrating}}]\label{def:DP}
Let $\epsilon\geq0$. For a measurable space $(\mathcal{Z},\Omega)$, let $\mathcal{A}:\mathcal{X}^N\rightarrow\mathcal{P}(\mathcal{Z})$ be a randomized algorithm. The algorithm $\mathcal{A}$ is $\epsilon$-differentially private or $\epsilon$-DP if
\begin{align}
\mathbb{P}(\mathcal{A}(D)\in S)
\leq
e^\epsilon \mathbb{P}(\mathcal{A}(D')\in S)\label{eq:pure_dp}
\end{align}
for every measurable set $S\in\Omega$ and every pair of neighboring datasets $D,D'\in\mathcal{X}^N$.
\end{definition}
An algorithm is differentially private if replacing a single observation changes the distribution of its output by at most a controlled multiplicative factor $e^\epsilon$. Smaller values of $\epsilon$ correspond to stronger privacy guarantees, whereas larger values permit greater dependence of the output distribution on any single observation.

One useful property of differential privacy is its preservation under adaptive composition.
\begin{proposition}[{Composition theorem; \citealp{4389483}}]\label{prop:comp}
    For measurable spaces $(\mathcal{Z}_1,\Omega_1)$ and $(\mathcal{Z}_2,\Omega_2)$, let $\mathcal{A}_1:\mathcal{X}^N\rightarrow\mathcal{P}(\mathcal{Z}_1)$ be an $\epsilon_1$-DP algorithm. Also, let $\mathcal{A}_2:\mathcal{X}^N\times\mathcal{Z}_1\rightarrow\mathcal{P}(\mathcal{Z}_2)$ be a randomized algorithm such that, for every fixed $z_1\in\mathcal{Z}_1$, $\mathcal{A}_2(\cdot,z_1)$ is $\epsilon_2$-DP. Define the composed algorithm $\mathcal{A}_{\mathrm{comp}}:\mathcal{X}^N\rightarrow\mathcal{P}(\mathcal{Z}_2)$ by
    \begin{align*}
        \mathcal{A}_{\mathrm{comp}}(D)\coloneqq \mathcal{A}_2(D,\mathcal{A}_1(D)),\qquad D\in\mathcal{X}^N.
    \end{align*}
    Then $\mathcal{A}_{\mathrm{comp}}$ is $(\epsilon_1+\epsilon_2)$-DP.
\end{proposition}
Repeated application of Proposition~\ref{prop:comp} allows us to combine tests across resolutions while controlling the total privacy budget.

\subsection{Goodness-of-fit testing}

Suppose that a null distribution $P_0$ is given and that $X_1,\ldots,X_n$ are drawn independently from an unknown distribution $P\in\mathcal{P}(\mathcal{X})$. The goodness-of-fit testing problem aims to decide
\begin{align}\label{eq:ht_gof}
H_0:P=P_0
\quad\textrm{vs}\quad
H_1:P\neq P_0.
\end{align}
We use the minimax separation radius framework \citep[e.g.,][]{ingster1993asymptotically} to assess the optimality of goodness-of-fit testing algorithms. For $p\geq 1$, a null distribution $P_0$, and a separation parameter $\rho>0$, we define the class of alternatives separated from $P_0$ according to the $p$th power of the Wasserstein-$p$ distance by
\begin{align*}
\mathcal{P}_1^{\mathrm{gof}}(P_0,\rho,p) \coloneqq \left\{P\in\mathcal{P}(\mathcal{X}):W_p^p(P,P_0)\geq \rho\right\}.
\end{align*}

We assume the null distribution $P_0$ is chosen from $\mathcal{P}_0^{\mathrm{gof}}\coloneqq \mathcal{P}(\mathcal{X})$.

\begin{definition}[DP minimax separation radius for goodness-of-fit testing]\label{def:minimax_sep_gof}
Let $\phi_{P_0}:\mathcal{X}^n\rightarrow\{0,1\}$ be a goodness-of-fit testing procedure for \eqref{eq:ht_gof}. Given Type I and Type II error levels $\alpha,\beta\in(0,1/3]$, $p\geq 1$, and a sample size $n$, the separation radius of $\phi_{P_0}$ is defined as
\begin{align*}
    \rho_{n}^{\mathrm{gof}}(\phi_{P_0},\mathcal{X},p,\alpha,\beta)\coloneqq \inf\left\{\rho>0:\mathbb{E}_{P_0}[\phi_{P_0}]\leq\alpha,\sup_{P\in\mathcal{P}_{1}^{\mathrm{gof}}(P_0,\rho,p)}\mathbb{E}_P[1-\phi_{P_0}]\leq\beta\right\}.
\end{align*}
For the privacy parameter $\epsilon>0$, the minimax separation radius for $\epsilon$-DP goodness-of-fit testing is then defined as
\begin{align*} \rho_{n,\epsilon}^{\mathrm{gof}}(\mathcal{X},p,\alpha,\beta)\coloneqq \sup_{P_0\in\mathcal{P}_0^{\mathrm{gof}}}\inf_{\phi_{P_0}:\epsilon\textrm{-DP}}\rho_{n}^{\mathrm{gof}}(\phi_{P_0},\mathcal{X},p,\alpha,\beta).\end{align*}
\end{definition}

Intuitively, $\rho_{n,\epsilon}^{\mathrm{gof}}$ is the smallest separation, uniformly over all null distributions $P_0$, above which an $\epsilon$-DP goodness-of-fit test can distinguish $P_0$ from alternatives separated in Wasserstein-$p$ distance with Type I error at most $\alpha$ and Type II error at most $\beta$.

\subsection{Two-sample testing}
In two-sample testing, we observe two independent samples $X_1,\ldots,X_n\overset{\textrm{i.i.d.}}{\sim} P$ and $Y_1,\ldots,Y_m\overset{\textrm{i.i.d.}}{\sim} Q$ and aim to determine whether
\begin{align}\label{eq:ht_ts}
    H_0:P=Q\textrm{ vs }H_1:P\not=Q.
\end{align}
Unlike in goodness-of-fit testing, neither distribution is known in advance. This additional uncertainty can make two-sample testing statistically more demanding than goodness-of-fit testing.

For $p\geq1$ and $\rho>0$, define the null and alternative classes by
\begin{align*}
    \mathcal{P}_{0}^{\mathrm{ts}}&\coloneqq \{(P,Q):P=Q\in\mathcal{P}(\mathcal{X})\}\\[6pt]
    \mathcal{P}_{1}^{\mathrm{ts}}(\rho,p)&\coloneqq \{(P,Q):W_p^p(P,Q)\geq \rho,P,Q\in\mathcal{P}(\mathcal{X})\}.
\end{align*}
We write $\mathbb{E}_{P,Q}$ for expectation over the independent samples drawn from $P$ and $Q$ and any internal randomness of the test. We next define the minimax separation radius under differential privacy. 
\begin{definition}[DP minimax separation radius for two-sample testing]
Let $\phi:\mathcal{X}^{n+m}\rightarrow\{0,1\}$ be a two-sample test for \eqref{eq:ht_ts}. For given Type I and Type II error levels $\alpha,\beta\in(0,1/3]$, $p\geq 1$, and sample sizes $n$ and $m$, the separation radius of $\phi$ is defined as
    \begin{align*}
        \rho_{n,m}^{\mathrm{ts}}(\phi,\mathcal{X},p,\alpha,\beta)\coloneqq \inf\left\{\rho>0:\sup_{(P,Q)\in\mathcal{P}_{0}^{\mathrm{ts}}}\mathbb{E}_{P,Q}[\phi]\leq\alpha,\sup_{(P,Q)\in\mathcal{P}_{1}^{\mathrm{ts}}(\rho,p)}\mathbb{E}_{P,Q}[1-\phi]\leq\beta\right\}.
    \end{align*}
For the privacy parameter $\epsilon>0$, the minimax separation radius for $\epsilon$-DP two-sample testing is then defined as
    \begin{align*}
        \rho_{n,m,\epsilon}^{\mathrm{ts}}(\mathcal{X},p,\alpha,\beta)\coloneqq \inf_{\phi:\epsilon\textrm{-DP}}\rho_{n,m}^{\mathrm{ts}}(\phi,\mathcal{X},p,\alpha,\beta).
    \end{align*}
\end{definition}
For simplicity, throughout the main theoretical analysis, we consider the balanced case $n=m$ and simply denote $\rho_{n,\epsilon}^{\mathrm{ts}}(\mathcal{X},p,\alpha,\beta)\coloneqq \rho_{n,n,\epsilon}^{\mathrm{ts}}(\mathcal{X},p,\alpha,\beta)$. This does not change the separation rates up to constant factors, provided that $n\asymp m$. In such a case, one may subsample the larger sample to obtain two samples of size $\min\{n,m\}$.
\section{Minimax bounds on compact metric spaces}\label{sec:main}
In this section, we establish upper and lower bounds on the minimax separation radii for goodness-of-fit and two-sample testing over a given Polish metric space $(\mathcal{X},d)$. For this purpose, we impose the following boundedness assumption.
\begin{assumption}\label{assump:compact}
The metric space $(\mathcal{X},d)$ is compact. Equivalently, since $\mathcal{X}$ is complete, $\mathrm{diam}(\mathcal{X})\coloneqq \sup_{x,y\in \mathcal{X}}d(x,y)<\infty$ and $\mathcal{N}(\mathcal{X},\delta)<\infty$ for every $\delta>0$.
\end{assumption}
To justify this assumption, we show that, without further restrictions on the distributions, differentially private goodness-of-fit testing under Wasserstein separation is infeasible on an unbounded metric space.
\begin{proposition}\label{prop:unbounded_impossibility}
Suppose that $\mathrm{diam}(\mathcal{X})=\infty$. For every $n\in\mathbb{N}$, $p\geq1$, $\epsilon>0$, and $\rho>0$, there exist $P\in\mathcal{P}(\mathcal{X})$ and $Q\in\mathcal{P}_1^{\mathrm{gof}}(P,\rho,p)$ such that no test can have both Type I and Type II errors at most $1/3$.
\end{proposition}
The Type I and Type II error bounds $1/3$ can be replaced by an arbitrary pair of numbers $(\alpha,\beta)$ as long as $\alpha>0,\beta>0$, and $\alpha+\beta<1$. The proof exploits the fact that, unlike metrics such as total variation distance, the Wasserstein distance depends on the geometry of the underlying metric space. A small amount of probability mass can be placed arbitrarily far from the support of the null distribution, producing an arbitrarily large Wasserstein separation while keeping the resulting distributions statistically difficult to distinguish.
\begin{proof}[Proof of Proposition~\ref{prop:unbounded_impossibility}]
Fix any $x_\ast\in\mathcal{X}$ and let $P=\delta_{x_\ast}$. Set $\theta\coloneqq (1-\alpha-\beta)/(2n)$.
Since $\mathrm{diam}(\mathcal{X})=\infty$, there exists $x_0\in\mathcal{X}$ such that $d^p(x_\ast,x_0)\geq\rho/\theta$.
Define
\begin{align*}
Q\coloneqq (1-\theta)\delta_{x_\ast}+\theta\delta_{x_0}.
\end{align*}
Because the first marginal of any coupling between $P$ and $Q$ is concentrated at $x_\ast$,
\begin{align*}
W_p^p(P,Q)=\theta d^p(x_\ast,x_0)\geq\rho.
\end{align*}
Thus, $Q\in\mathcal{P}_1^{\mathrm{gof}}(P,\rho,p)$.

Now consider testing $P^n$ against $Q^n$. Since $P^n$ is concentrated on $(x_\ast,\ldots,x_\ast)$, while $Q^n$ assigns probability $(1-\theta)^n$ to this point,
\begin{align*}
\lVert P^n-Q^n\rVert_{\mathrm{TV}}=1-(1-\theta)^n\leq n\theta=\frac{1-\alpha-\beta}{2}<1-\alpha-\beta.
\end{align*}
On the other hand, if a test $\phi$ had Type I error at most $\alpha$ and Type II error at most $\beta$, then
\begin{align*}
\lVert P^n-Q^n\rVert_{\mathrm{TV}}\geq\mathbb{P}_{Q^n}(\phi=1)-\mathbb{P}_{P^n}(\phi=1)\geq1-\alpha-\beta,
\end{align*}
which is a contradiction.
\end{proof}
The same conclusion holds for two-sample testing. Indeed, suppose there exists a two-sample test that achieves the prescribed Type I and Type II error bounds uniformly over the corresponding null and alternative classes. Then one could use it for goodness-of-fit testing by generating i.i.d. samples from the null distribution. This would yield a goodness-of-fit test with the same error guarantees, contradicting Proposition~\ref{prop:unbounded_impossibility}.
\subsection{Goodness-of-fit testing}
In this subsection, we characterize the minimax separation radius $\rho_{n,\epsilon}^{\mathrm{gof}}$ for goodness-of-fit testing over a general compact metric space. For the upper bound, we construct an $\epsilon$-DP Wasserstein test by combining multiple $\epsilon$-DP $\ell_1$ testers across resolution scales and allocating the privacy and test error budgets non-uniformly according to their testing difficulty. We then derive a minimax lower bound that matches the upper bound up to a logarithmic factor. 

\subsubsection{Upper bound}
To construct the Wasserstein goodness-of-fit test, we use a multiscale partition of the underlying metric space. We first define a dyadic partition of $\mathcal{X}$. A collection of pairwise disjoint Borel sets is called a partition of $\mathcal{X}$ if their union is $\mathcal{X}$.
\begin{definition}\label{def:dyadic_partition}
For a given $\delta\in(0,1)$ and a positive integer $J$, let $\{\mathcal{R}^j\}_{j=1}^J$ be a sequence of partitions of $\mathcal{X}$. We say that $\{\mathcal{R}^j\}_{j=1}^J$ is a dyadic partition of $\mathcal{X}$ with parameter $\delta$ if it satisfies the following properties:
   \begin{itemize}
       \item $R\in \mathcal{R}^j$ implies $\mathrm{diam}(R)\leq\mathrm{diam}(\mathcal{X})\delta^j$
       \item $R^{j+1}\in\mathcal{R}^{j+1}$ and $R^j\in\mathcal{R}^j$ imply $R^{j+1}\subset R^j$ or $R^{j+1}\cap R^j=\emptyset$.
   \end{itemize}
\end{definition}
Thus, each successive partition refines the preceding one, as illustrated in Figure~\ref{fig:dyadic_partition}.
\begin{figure}
    \centering
    \begingroup
    \pgfmathsetlengthmacro{\dyadicunit}{0.24*\linewidth}
    \begin{tikzpicture}[
        x=\dyadicunit, y=\dyadicunit,
        line cap=butt, line join=miter,
        partition fine/.style={draw=black!50, line width=0.3pt},
        partition middle/.style={draw=black!65, line width=0.4pt},
        partition coarse/.style={draw=black!80, line width=0.5pt},
        partition border/.style={draw=black!85, line width=0.65pt},
        refinement arrow/.style={-{Stealth[length=4.5pt,width=3.5pt]}, draw=black!65, line width=0.55pt},
        every node/.style={inner sep=0pt, outer sep=0pt}
    ]
        \foreach \level/\offset/\cellcount in {1/0/4, 2/1.5/16, 3/3/64} {
            \begin{scope}[shift={(\offset,0)}]
                \fill[black!12] (0.5,0.5) rectangle (1,1);
                \ifnum\level>2
                    \foreach \position in {0.125,0.375,0.625,0.875} {
                        \draw[partition fine] (\position,0) -- (\position,1);
                        \draw[partition fine] (0,\position) -- (1,\position);
                    }
                \fi
                \ifnum\level>1
                    \foreach \position in {0.25,0.75} {
                        \draw[partition middle] (\position,0) -- (\position,1);
                        \draw[partition middle] (0,\position) -- (1,\position);
                    }
                \fi
                \draw[partition coarse] (0.5,0) -- (0.5,1);
                \draw[partition coarse] (0,0.5) -- (1,0.5);
                \draw[partition border] (0,0) rectangle (1,1);
                \node[anchor=south, yshift=8pt, font=\normalsize]
                    at (0.5,1) {$\mathcal{R}^{\level}$};
                \node[anchor=north, yshift=-7pt, font=\footnotesize]
                    at (0.5,0) {$\cellcount$ cells};
            \end{scope}
        }
        \foreach \offset in {0,1.5} {
            \draw[refinement arrow]
                (\offset+1.12,0.5) -- (\offset+1.38,0.5);
        }
    \end{tikzpicture}
    \endgroup
    \caption{Dyadic partition of $[0,1]^2$ with parameter $\delta=1/2$. Each partition $\mathcal{R}^j$ refines $\mathcal{R}^{j-1}$, thereby providing a finer resolution. Shading tracks the same cell of $\mathcal{R}^1$ across refinements.}
    \label{fig:dyadic_partition}
\end{figure}
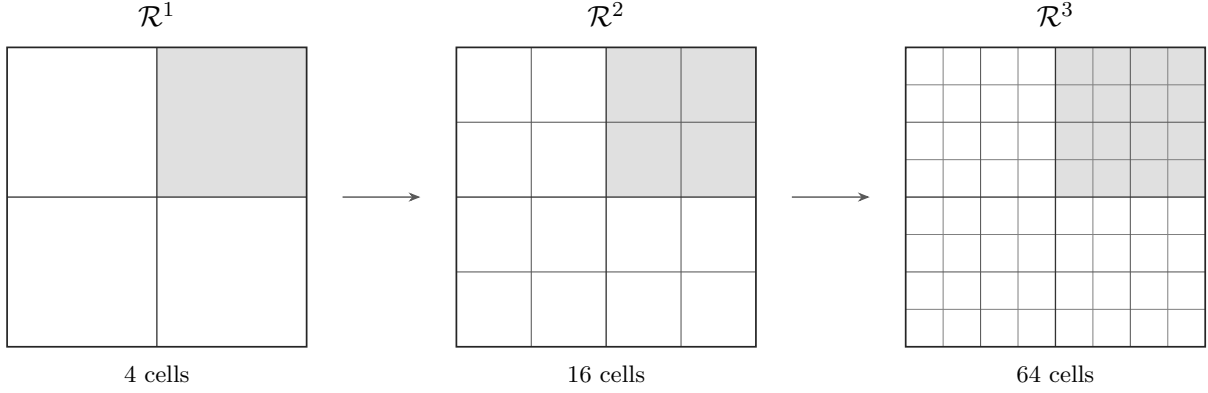
The cardinality of each partition can be controlled by the covering number of $\mathcal{X}$ at a slightly finer scale if $\gamma\in(0,1/3]$. For the remainder of the paper, we assume $\gamma\in(0,1/3]$.
\begin{lemma}\label{lem:dyadic}
For any $\gamma\in(0,1/3]$ and positive integer $J$, there exists a dyadic partition $\{\mathcal{R}^j\}_{j=1}^J$ of $\mathcal{X}$ with parameter $\gamma$ such that $|\mathcal{R}^j|\leq\mathcal{N}(\mathcal{X},\mathrm{diam}(\mathcal{X})\gamma^j/3)$ for every $j=1,\ldots,J$.
\end{lemma}
Lemma~\ref{lem:dyadic} follows from the argument used in the proof of \citet[][Proposition~3]{weed2019sharp}. For completeness, we include a proof in Appendix~\ref{appendix:dyadic}.

We can bound the Wasserstein distance by a weighted sum of the $\ell_1$ distances between the distributions induced by these partitions.
\begin{proposition}[{\citealp[Proposition~1]{weed2019sharp}}]\label{prop:Wp_decomp}
Let $P$ and $Q$ be two Borel probability measures on $\mathcal{X}$, and let $\{\mathcal{R}^j\}_{j=1}^J$ be a dyadic partition of $\mathcal{X}$ with parameter $\gamma$. Denote by $P^j$ and $Q^j$ the probability distributions on $\mathcal{R}^j$ defined by
\begin{align*}
    P^j(R)\coloneqq P(R)\textrm{ and }Q^j(R)=Q(R).
\end{align*}
Then the following inequality holds:
\begin{align}\label{eq:Wp_decomp}
W_p^p(P,Q)\leq\mathrm{diam}^p(\mathcal{X}) \left(\gamma^{Jp}+\sum_{j=1}^{J}\gamma^{(j-1)p}\lVert P^j-Q^j\rVert_1\right).
\end{align}
\end{proposition}
Given $\tau_1,\ldots,\tau_J>0$, if $\lVert P^j-Q^j\rVert_1\leq\tau_j$ for every $j=1,\ldots,J$, then Proposition~\ref{prop:Wp_decomp} implies
\begin{align*}
    W_p^p(P,Q)\leq \mathrm{diam}^p(\mathcal{X})\left(\gamma^{Jp}+\sum_{j=1}^J\gamma^{(j-1)p}\tau_j\right).
\end{align*}
Consequently, if $W_p^p(P,Q)$ is sufficiently large, then at least one resolution level must satisfy $\lVert P^j-Q^j\rVert_1>\tau_j$, provided that the weighted sum of the thresholds $\tau_j$ is appropriately controlled. This observation suggests applying an $\ell_1$ goodness-of-fit test at each resolution level. For the $\ell_1$ tester, we use the minimax-optimal $\epsilon$-DP $\ell_1$ goodness-of-fit test for discrete distributions developed by \cite{NEURIPS2018_7de32147}. The following proposition states the corresponding minimax separation bounds, which we use to analyze our Wasserstein test. Its proof is deferred to Appendix~\ref{appendix:tv_minimax}.
\begin{proposition}[{\citealp{NEURIPS2018_7de32147}}]\label{prop:gof_tv_minimax}
Let $k\geq2$ and $n\geq1$ be integers, and let $\epsilon>0$. Let $u_0\in\mathcal{P}([k])$ be a known distribution, and suppose that we observe $n$ i.i.d. samples from an unknown distribution $u\in\mathcal{P}([k])$. For given Type I and Type II error levels $\alpha,\beta\in(0,1/3]$, there exist a universal constant $C_1>0$ and a constant $C_2\in(0,1/2)$ depending only on $\alpha$ and $\beta$ with the following properties.
\begin{enumerate}
\renewcommand{\labelenumi}{(\roman{enumi})}
\item There exists an $\epsilon$-DP tester for
\begin{align*}
H_0:u=u_0
\quad\textrm{vs}\quad
H_1:\lVert u-u_0\rVert_1\geq \tau
\end{align*}
with Type I error at most $\alpha$ and Type II error at most $\beta$ if
\begin{align*}
\tau\geq C_1\max\left\{\frac{k^{1/4}}{\sqrt{n}},\frac{k^{1/2}}{n\sqrt{\epsilon}},\frac{k^{1/4}}{n^{3/4}\sqrt{\epsilon}},\frac{1}{n\epsilon}\right\}\log\frac{1}{\alpha\wedge\beta}.
\end{align*}
\item Conversely, when $u_0=\mathrm{Unif}([k])$, any $\epsilon$-DP tester for this testing problem with Type I and Type II errors at most $\alpha$ and $\beta$, respectively, satisfies
\begin{align*} 
    \tau\geq C_2\left(\max\left\{\frac{k^{1/4}}{\sqrt{n}},\frac{k^{1/2}}{n\sqrt{\epsilon}},\frac{k^{1/4}}{n^{3/4}\sqrt{\epsilon}},\frac{1}{n\epsilon}\right\}\wedge1\right).
\end{align*}
\end{enumerate}
\end{proposition}
By combining the private $\ell_1$ testers across the resolution levels, we construct the $\epsilon$-DP Wasserstein test described in Algorithm~\ref{alg:gof_tester_Wp}. The use of $\ell_1$ goodness-of-fit testers as subroutines is standard in non-private Wasserstein testing; see, e.g., \cite{ba2011sublinear,deng2017wasserstein}. Such a multiscale construction requires an allocation of the test error budget, since the overall Type I error is controlled by the sum of the Type I errors of the individual testers. Similarly, by the basic composition property of differential privacy, the overall privacy parameter is controlled by the sum of the privacy parameters assigned to the individual testers, so the privacy budget must also be allocated across resolution levels. Existing multiscale Wasserstein tests allocate the test error budget uniformly across resolution levels. In contrast, our construction introduces allocation parameters $\nu_1,\ldots,\nu_J$ to distribute the Type I error and privacy budgets non-uniformly across scales. We calibrate $\nu_1,\ldots,\nu_J$ according to the difficulty of the corresponding $\ell_1$ testing problems, assigning relatively larger budgets to harder-to-test resolution levels. This calibrated allocation leads to a smaller upper bound on the separation radius of the resulting Wasserstein test.
\begin{algorithm}
\caption{$W_p$-based goodness-of-fit tester}\label{alg:gof_tester_Wp}
\begin{algorithmic}
\Input:\textrm{ }Null distribution $P_0$, data $X_1,\ldots,X_n$, Wasserstein order $p$, partition depth $J$, separation parameters $\tau_1,\ldots,\tau_J$, Type I and Type II error levels $\alpha$ and $\beta$, privacy parameter $\epsilon$, and allocation parameters $\nu_1,\ldots,\nu_J>0$ satisfying $\sum_{j=1}^J\nu_j\leq1$.
\EndInput
\begin{enumerate}
\item Generate a dyadic partition $\{\mathcal{R}^j\}_{j=1}^J$ with parameter $\gamma=1/3$ such that $|\mathcal{R}^j|\leq\mathcal{N}(\mathcal{X},\mathrm{diam}(\mathcal{X})\gamma^j/3)$.
\item \textbf{For }$j=1,\ldots,J$ \textbf{ do}
\begin{enumerate}
\item Define $Z_i^j\coloneqq R$ if $X_i\in R$ for $R\in\mathcal{R}^j$, and define the induced null distribution $P_0^j$.
\item Apply the minimax-optimal $\nu_j\epsilon$-DP $\ell_1$-based tester $\phi_j$ from Proposition~\ref{prop:gof_tv_minimax} to a sample $Z_1^j,\ldots,Z_n^j$, using $P_0^j$ as the null distribution, $\tau_j$ as the separation parameter, and $\nu_j\alpha$ and $\nu_j\beta$ as the Type I and Type II error levels, respectively.
\item \Return \textbf{Reject} if $\phi_j$ rejects; otherwise, proceed to the next level.
\end{enumerate}
\end{enumerate}
\Return \textbf{Accept}.
\end{algorithmic}
\end{algorithm}

To state the resulting upper bound concisely, we introduce the following quantities. For each $j\in\mathbb{N}$, let $N_j\coloneqq \mathcal{N}(\mathcal{X},\mathrm{diam}(\mathcal{X})\gamma^{j}/3)$. For $p>0$, $q\geq0$, $r\geq1$, and a positive integer $J$, define 
\begin{align*}
\Sigma(p,q,r,J)&\coloneqq \left(\sum_{j=1}^J(N_j^{q}\gamma^{jp})^{\frac{1}{r}}\right)^r, \qquad w_j \coloneqq \frac{(N_j^q\gamma^{jp})^{1/r}}{\sum_{k=1}^{J}(N_k^q\gamma^{kp})^{1/r}},\quad j=1,\ldots,J.
\end{align*}
Let $H(p,q,r,J)$ denote the entropy of the weights $w_1,\ldots,w_J$ and define $\Lambda(p,q,r,J)$ by
\begin{align*}
    H(p,q,r,J)&\coloneqq -\sum_{j=1}^J w_j\log w_j, \qquad \Lambda(p,q,r,J)\coloneqq \Sigma(p,q,r,J)\bigl(1+H(p,q,r,J)\bigr).
\end{align*}
As shown in Lemma~\ref{lem:dyadic}, $N_j$ provides an upper bound on the cardinality of $\mathcal{R}^j$. The quantity $\Sigma(p,q,r,J)$ captures the main aggregate contribution of the separation radii of the private $\ell_1$ testers across resolution levels, while $H(p,q,r,J)$ quantifies the additional contribution arising from the calibrated allocation of the error and privacy budgets. Here, $p$ denotes the Wasserstein distance parameter, and $J$ is the number of resolution levels used in the test. The parameter $q$ represents the dependence of the underlying $\ell_1$ testing difficulty on the domain size, while $r$ determines its dependence on the allocated privacy budget. The logarithmic effect of the allocated error budget is captured by $H(p,q,r,J)$. For comparison, under uniform allocation, the contribution corresponding to $H(p,q,r,J)$ becomes $J^{r-1}\log J$, whereas the calibrated allocation yields $H(p,q,r,J)$. Since $H(p,q,r,J)\leq\log J$ and $r\geq1$, the contribution under the calibrated allocation is always no larger. Combining these two quantities yields $\Lambda(p,q,r,J)$, which will be used to express the resulting upper bound on the separation radius of the Wasserstein test.

In the following theorem, $\alpha$ and $\beta$ are treated as fixed constants, and their dependence is absorbed into the constants.
\begin{theorem}\label{thm:goodness_of_fit_minimax_upper}
The minimax separation radius of $\epsilon$-DP goodness-of-fit testing satisfies
\begin{align*}
\rho_{n,\epsilon}^{\mathrm{gof}}\leq C_{\mathrm{upper}}^{\mathrm{gof}} \mathrm{diam}^p(\mathcal{X})R_{\mathrm{upper}}^{\mathrm{gof}},
\end{align*}
where $C_{\mathrm{upper}}^{\mathrm{gof}}>\gamma^{-p}/2$ is a constant depending only on $p$, $\alpha$, and $\beta$, and
\begin{align*}
    R_{\mathrm{upper}}^{\mathrm{gof}}&\coloneqq \inf\Bigg\{R\in(0,2\gamma^{p}]:R\geq \Psi^{\mathrm{gof}}_{\mathrm{upper}}(R)\Bigg\},\\[6pt]
    \Psi^{\mathrm{gof}}_{\mathrm{upper}}(R)&\coloneqq \frac{\Lambda(p,1/4,1,J_{\mathrm{upper}}(R))}{\sqrt{n}}+\frac{\Lambda(p,1/4,3/2,J_{\mathrm{upper}}(R))}{n^{3/4}\sqrt{\epsilon}}+\frac{\Lambda(p,1/2,3/2,J_{\mathrm{upper}}(R))}{n\sqrt{\epsilon}}\\
    &~+\frac{\Lambda(p,0,2,J_{\mathrm{upper}}(R))}{n\epsilon},\\[6pt]
    J_{\mathrm{upper}}(R)&\coloneqq \left\lceil\frac{1}{p}\frac{\log(2/R)}{\log(1/\gamma)}\right\rceil.
\end{align*}
If the set in the above infimum is empty, we set $R_{\mathrm{upper}}^{\mathrm{gof}}=2\gamma^p$.
\end{theorem}
We note that the infimum is attained whenever the set is nonempty. Indeed, $\Psi_{\mathrm{upper}}^{\mathrm{gof}}(R)$ is nonincreasing while $R$ is increasing. Also, $\Psi_{\mathrm{upper}}^{\mathrm{gof}}(R)-R$ is right-continuous. Thus, the infimum is attained. We establish the monotonicity of $\Psi_{\mathrm{upper}}^{\mathrm{gof}}(R)$ in Appendix~\ref{appendix:monotonicity}.
\begin{proof}[Proof of Theorem~\ref{thm:goodness_of_fit_minimax_upper}]
We divide the proof into two cases. First, suppose that the set in the definition of $R_{\mathrm{upper}}^{\mathrm{gof}}$ is nonempty. Given $p,\alpha,$ and $\beta$, we choose the remaining parameters $\tau_1,\ldots,\tau_J$ and $\nu_1,\ldots,\nu_J$ of Algorithm~\ref{alg:gof_tester_Wp} as follows. Set $J\coloneqq J_{\mathrm{upper}}(R_{\mathrm{upper}}^{\mathrm{gof}})$. Then $J\geq 1$ since $\frac{1}{p}\frac{\log(2/R_{\mathrm{upper}}^{\mathrm{gof}})}{\log(1/\gamma)}\geq1$. For each $1\leq j\leq J$, set
\begin{align*}
    \nu_j&\coloneqq \frac{N_j^{1/4}\gamma^{jp}}{4\sum_{j=1}^JN_j^{1/4}\gamma^{jp}}\vee\frac{(N_j^{1/2}\gamma^{jp})^{\frac{2}{3}}}{4\sum_{j=1}^J(N_j^{1/2}\gamma^{jp})^{\frac{2}{3}}}\vee\frac{(N_j^{1/4}\gamma^{jp})^{\frac{2}{3}}}{4\sum_{j=1}^J(N_j^{1/4}\gamma^{jp})^{\frac{2}{3}}}\vee\frac{\gamma^{\frac{jp}{2}}}{4\sum_{j=1}^J\gamma^{\frac{jp}{2}}}.
\end{align*}
Also, let $\tau_j$ be the separation radius of the tester in Step 2(b) of Algorithm~\ref{alg:gof_tester_Wp}.

With these parameter choices, we show that, for any $P_0\in\mathcal{P}(\mathcal{X})$, Algorithm~\ref{alg:gof_tester_Wp} is an $\epsilon$-DP test with Type I and Type II errors at most $\alpha$ and $\beta$, respectively, for the following Wasserstein goodness-of-fit testing problem.
\begin{align*}
    H_0:P=P_0\textrm{ vs }H_1:W_p^p(P,P_0)\geq C_{\mathrm{upper}}^{\mathrm{gof}}\mathrm{diam}^p(\mathcal{X})R_{\mathrm{upper}}^{\mathrm{gof}},
\end{align*}
where $C_{\mathrm{upper}}^{\mathrm{gof}}>0$ is a sufficiently large constant. The privacy guarantee and the Type I error bound follow directly. Note that $\sum_{j=1}^J\nu_j\leq 1$ since each of the four sequences appearing in this maximum sums to $1/4$. The algorithm is a composition of $\nu_j\epsilon$-DP algorithms for $j=1,\ldots,J$. By Proposition~\ref{prop:comp}, Algorithm~\ref{alg:gof_tester_Wp} is $\sum_{j=1}^J\nu_j\epsilon$-DP. Since $\sum_{j=1}^J\nu_j\epsilon\leq\epsilon$, the algorithm is $\epsilon$-DP. Similarly, since the algorithm rejects whenever at least one of the $J$ finite-domain testers rejects, a union bound gives that the Type I error of Algorithm~\ref{alg:gof_tester_Wp} is at most $\sum_{j=1}^J\nu_j\alpha\leq \alpha$.

We next show that the Type II error of Algorithm~\ref{alg:gof_tester_Wp} is at most $\beta$. By the multiresolution upper bound \eqref{eq:Wp_decomp}, it suffices to establish
\begin{align*}
\mathrm{diam}^p(\mathcal{X})\left(\gamma^{Jp}+\sum_{j=1}^J\gamma^{(j-1)p}\tau_j\right)\leq C_{\mathrm{upper}}^{\mathrm{gof}}\mathrm{diam}^p(\mathcal{X})R_{\mathrm{upper}}^{\mathrm{gof}},
\end{align*}
or, equivalently,
\begin{align}\label{eq:proof_gof_minimax_upper_goal}
    \gamma^{Jp}+\sum_{j=1}^J\gamma^{(j-1)p}\tau_j\leq C_{\mathrm{upper}}^{\mathrm{gof}}R_{\mathrm{upper}}^{\mathrm{gof}}.
\end{align}
Suppose inequality \eqref{eq:proof_gof_minimax_upper_goal} holds. Under the alternative hypothesis, upper bound \eqref{eq:Wp_decomp} then implies
\begin{align*}
\mathrm{diam}^p(\mathcal{X})\left(\gamma^{Jp}+\sum_{j=1}^J\gamma^{(j-1)p}\lVert P^j-P_0^j\rVert_1\right)&\geq  W_p^p(P,P_0)\\
&\geq  C_{\mathrm{upper}}^{\mathrm{gof}}\mathrm{diam}^p(\mathcal{X})R_{\mathrm{upper}}^{\mathrm{gof}}\\
&\geq  \mathrm{diam}^p(\mathcal{X})\left(\gamma^{Jp}+\sum_{j=1}^J\gamma^{(j-1)p}\tau_j\right).
\end{align*}
Thus, there exists $1\leq j_0\leq J$ such that $\lVert P^{j_0}-P^{j_0}_0\rVert_1\geq \tau_{j_0}$. Since Algorithm~\ref{alg:gof_tester_Wp} accepts the null hypothesis if every private $\ell_1$ tester accepts, its Type II error is at most the Type II error of the private $\ell_1$ tester $\phi_{j_0}$. By the choice of $\tau_{j_0}$, this error is at most $\nu_{j_0}\beta\leq\beta$.

It remains to verify inequality \eqref{eq:proof_gof_minimax_upper_goal}. By the definition of $J$, we have $\gamma^{Jp}\leq\frac{1}{2}R_{\mathrm{upper}}^{\mathrm{gof}}$. For $1\leq j\leq J$, since $|\mathcal{R}^j|\leq \mathcal{N}(\mathcal{X},\mathrm{diam}(\mathcal{X})\gamma^j/3)$, Proposition~\ref{prop:gof_tv_minimax} implies that $\tau_j$ with $\epsilon_j = \nu_j \epsilon$ satisfies
\begin{align*}
    \tau_j&\lesssim \Bigg(\frac{N_j^{1/4}}{\sqrt{n}}+\frac{N_j^{1/4}}{n^{3/4}\sqrt{\epsilon_j}}+\frac{N_j^{1/2}}{n\sqrt{\epsilon_j}}+\frac{1}{n\epsilon_j}\Bigg) \log\frac{1}{(\nu_j\alpha)\wedge(\nu_j\beta)}\notag\\
    &\lesssim \Bigg(\frac{N_j^{1/4}}{\sqrt{n}}+\frac{N_j^{1/4}\nu_j^{-1/2}}{n^{3/4}\sqrt{\epsilon}}+\frac{N_j^{1/2}\nu_j^{-1/2}}{n\sqrt{\epsilon}}+\frac{\nu_j^{-1}}{n\epsilon}\Bigg)\log\frac{1}{\nu_j},
\end{align*}
up to a constant factor depending only on $\alpha$ and $\beta$. By the definition of $\nu_j$,
\begin{align*}
    \sum_{j=1}^JN_j^{1/4}\gamma^{jp}\log\frac{1}{\nu_j}&\leq
    \sum_{j=1}^JN_j^{1/4}\gamma^{jp}\log\left(\frac{4\sum_{j=1}^JN_j^{1/4}\gamma^{jp}}{N_j^{1/4}\gamma^{jp}}\right)\\
    &= \Sigma(p,1/4,1,J)\left(\log 4+H(p,1/4,1,J)\right)\\
    &\lesssim \Sigma(p,1/4,1,J)\bigl(1+H(p,1/4,1,J)\bigr)\\
    &= \Lambda(p,1/4,1,J).
\end{align*}
Applying the same argument to the remaining three allocation sequences yields
\begin{align*}
\sum_{j=1}^JN_j^{1/4}\gamma^{jp}\nu_j^{-1/2}\log\frac{1}{\nu_j}&\lesssim \Lambda(p,1/4,3/2,J),\\
\sum_{j=1}^JN_j^{1/2}\gamma^{jp}\nu_j^{-1/2}\log\frac{1}{\nu_j}&\lesssim \Lambda(p,1/2,3/2,J),\\
\sum_{j=1}^J\gamma^{jp}\nu_j^{-1}\log\frac{1}{\nu_j}&\lesssim \Lambda(p,0,2,J).
\end{align*}
Therefore,
\begin{align*}
    \sum_{j=1}^J\gamma^{(j-1)p}\tau_j&\lesssim \frac{\Lambda(p,1/4,1,J)}{\sqrt{n}}+\frac{\Lambda(p,1/4,3/2,J)}{n^{3/4}\sqrt{\epsilon}}+\frac{\Lambda(p,1/2,3/2,J)}{n\sqrt{\epsilon}}+\frac{\Lambda(p,0,2,J)}{n\epsilon}\\
    &\leq R_{\mathrm{upper}}^{\mathrm{gof}}
\end{align*}
up to a constant factor depending only on $\alpha$ and $\beta$, where the last inequality follows because $R_{\mathrm{upper}}^{\mathrm{gof}}$ belongs to the set in its definition. Combining this bound with $\gamma^{Jp}\leq R_{\mathrm{upper}}^{\mathrm{gof}}/2$ shows that \eqref{eq:proof_gof_minimax_upper_goal} holds for a sufficiently large constant $C_{\mathrm{upper}}^{\mathrm{gof}}$ depending only on $p,\alpha,\beta$.

It remains to consider the case in which the set defining $R_{\mathrm{upper}}^{\mathrm{gof}}$ is empty. By definition, $R_{\mathrm{upper}}^{\mathrm{gof}}=2\gamma^p$. Our choice $C_{\mathrm{upper}}^{\mathrm{gof}}>\gamma^{-p}/2$ ensures that $C_{\mathrm{upper}}^{\mathrm{gof}}\mathrm{diam}^p(\mathcal{X})R_{\mathrm{upper}}^{\mathrm{gof}}>\mathrm{diam}^p(\mathcal{X})$. However, no distribution $P$ satisfies $W_p^p(P,P_0)> \mathrm{diam}^p(\mathcal{X})$ since
\begin{align*}
    W_p^p(P,P_0)&= \inf_{\pi\in\Pi(P,P_0)}\int_{\mathcal{X}\times \mathcal{X}}d^p(x,y)d\pi(x,y)\leq\inf_{\pi\in\Pi(P,P_0)}\int_{\mathcal{X}\times \mathcal{X}} \mathrm{diam}^p(\mathcal{X})d\pi(x,y)\leq \mathrm{diam}^p(\mathcal{X}).
\end{align*}
Thus, the alternative hypothesis is empty. The test that always accepts the null hypothesis therefore has zero Type I error, while the Type II error bound holds vacuously.
\end{proof}
\subsubsection{Lower bound}
As in the previous subsection, we derive a lower bound on the minimax separation radius by reducing Wasserstein goodness-of-fit testing to $\ell_1$-based testing for distributions on a finite domain. The following lemma shows that a collection of well-separated points induces a mapping from discrete distributions to $\mathcal{P}(\mathcal{X})$ that converts total variation separation into Wasserstein separation.
\begin{lemma}\label{lem:grid}
Let $k$ be a positive integer, and let $x_1,\ldots,x_k\in\mathcal{X}$ satisfy $d(x_i,x_{i^\prime})\geq h$, for $i\not=i^\prime$. For $u\in\mathcal{P}([k])$, define
\begin{align*}
    T_hu\coloneqq \sum_{i=1}^ku_i\delta_{x_i}.
\end{align*}
Then for all $u,v\in\mathcal{P}([k])$,
\begin{align*}
    W_p^p(T_hu,T_hv)\geq \frac{1}{2}h^p\lVert u-v\rVert_1.
\end{align*}
\end{lemma}
\begin{proof}[Proof of Lemma~\ref{lem:grid}]
We use the fact that if $\pi\in\Pi(u,v)$ then $\sum_{i\not=j}\pi(i,j)\geq \frac{1}{2}\lVert u-v\rVert_1$. This can be shown as follows:
\begin{align*}
    \sum_{i\not=j}\pi(i,j)=1-\sum_{i=1}^k\pi(i,i)\geq1-\sum_{i=1}^k\min\{u_i,v_i\}=\sum_{i=1}^k\max\{u_i-v_i,0\},
\end{align*}
where the inequality holds since $\pi(i,i)\leq \sum_{j=1}^k\pi(i,j)=u_i$ and $\pi(i,i)\leq \sum_{j=1}^k\pi(j,i)=v_i$. Similarly, we obtain $\sum_{i\not=j}\pi(i,j)\geq \sum_{i=1}^k\max\{v_i-u_i,0\}$. Thus,
\begin{align*}
    2\sum_{i\not=j}\pi(i,j)\geq \sum_{i=1}^k\max\{u_i-v_i,0\}+\sum_{i=1}^k\max\{v_i-u_i,0\}=\sum_{i=1}^k|u_i-v_i|=\lVert u-v\rVert_1.
\end{align*}
Using this bound and the separation condition $d(x_i,x_j)\geq h$ for $i\neq j$, we obtain
\begin{align*}
W_p^p(T_hu,T_hv)&= \inf_{\pi\in\Pi(T_hu,T_hv)}\int_{\mathcal{X}\times\mathcal{X}}d^p(x,y)d\pi(x,y)\\
        &= \inf_{\pi\in\Pi(u,v)}\sum_{i,j}d^p(x_i,x_j)\pi(i,j)\\
        &\geq \inf_{\pi\in\Pi(u,v)}h^p\sum_{i\not=j}\pi(i,j)\\
        &\geq \frac{1}{2}h^p\lVert u-v\rVert_1.
\end{align*}
This proves the claim.
\end{proof}
We also use the following standard relationship between covering and packing numbers.
\begin{lemma}[{\citealp[Lemma~5.5]{wainwright2019high}}]\label{lem:sep_cover}
For a given $h>0$, there exist $\mathcal{N}(\mathcal{X},h)$ points $\{x_i\}_{i=1}^{\mathcal{N}(\mathcal{X},h)}$ such that
    \begin{align*}
        d(x_i,x_{i^\prime})\geq h/2,
    \end{align*}
    for every $i\not=i^\prime$.
\end{lemma}
For $p>0$, $q\geq0$, and a positive integer $J$, define
\begin{align*}
    M(p,q,J)&\coloneqq \max_{1\leq j\leq J}N_j^q\gamma^{jp}.
\end{align*}
We now state the lower bound for
$\rho_{n,\epsilon}^{\mathrm{gof}}$.
\begin{theorem}\label{thm:goodness_of_fit_minimax_lower}
The minimax separation radius of $\epsilon$-DP goodness-of-fit testing satisfies
\begin{align*}
\rho_{n,\epsilon}^{\mathrm{gof}}\geq C_{\mathrm{lower}}^{\mathrm{gof}}\mathrm{diam}^p(\mathcal{X})R_{\mathrm{lower}}^{\mathrm{gof}},
\end{align*}
where $C_{\mathrm{lower}}^{\mathrm{gof}}>0$ is a constant depending only on $p,\alpha$, and $\beta$, and
\begin{align*}
   R_{\mathrm{lower}}^{\mathrm{gof}}&\coloneqq \sup\left\{R\in(0,2\gamma^p]:R\leq \Psi^{\mathrm{gof}}_{\mathrm{lower}}(R)\right\},\\[6pt]
    \Psi^{\mathrm{gof}}_{\mathrm{lower}}(R)&\coloneqq \frac{M(p,1/4,J_{\mathrm{lower}}(R))}{\sqrt{n}}\vee\frac{M(p,1/2,J_{\mathrm{lower}}(R))}{n\sqrt{\epsilon}}\vee \frac{M(p,1/4,J_{\mathrm{lower}}(R))}{n^{3/4}\sqrt{\epsilon}}\\
    & ~ \vee\frac{M(p,0,J_{\mathrm{lower}}(R))}{n\epsilon},\\[6pt]
    J_{\mathrm{lower}}(R)&\coloneqq \left\lfloor\frac{1}{p}\frac{\log(2/R)}{\log(1/\gamma)}\right\rfloor.
\end{align*}
\end{theorem}
We note that the set in the definition of $R_{\mathrm{lower}}^{\mathrm{gof}}$ is nonempty and that its supremum is attained. Indeed, since $M(p,0,J)=\gamma^p$ for every positive integer $J$, $\Psi^{\mathrm{gof}}_{\mathrm{lower}}(R)\geq\frac{\gamma^p}{n\epsilon}$. Hence, every sufficiently small positive $R$ belongs to the set. Moreover, $R\mapsto J_{\mathrm{lower}}(R)$ is left-continuous, and therefore so is $R\mapsto\Psi^{\mathrm{gof}}_{\mathrm{lower}}(R)$. It follows that the supremum also belongs to the set.

Also, the stated bounds match up to a logarithmic factor. Since $C_{\mathrm{upper}}^{\mathrm{gof}}\mathrm{diam}^p(\mathcal{X})R_{\mathrm{upper}}^{\mathrm{gof}}\geq C_{\mathrm{lower}}^{\mathrm{gof}}\mathrm{diam}^p(\mathcal{X})R_{\mathrm{lower}}^{\mathrm{gof}}\gtrsim \frac{1}{\sqrt{n}}$, the following holds:
\begin{align*}
    \Sigma(p,q,r,J)&\leq J^rM(p,q,J)\lesssim \log^r nM(p,q,J),\\[6pt]
    H(p,q,r,J)&\leq\log J\lesssim \log\log n.
\end{align*}
Therefore, $R_{\mathrm{upper}}^{\mathrm{gof}}\lesssim \log^{3/2} n\log\log nR_{\mathrm{lower}}^{\mathrm{gof}}$ up to a constant factor depending only on $p$ and $\mathcal{X}$. A detailed derivation is given in Appendix~\ref{appendix:logfactor}.
\begin{proof}[Proof of Theorem~\ref{thm:goodness_of_fit_minimax_lower}]
For a positive integer $j$, denote $h_j\coloneqq \mathrm{diam}(\mathcal{X})\gamma^j/6$. We choose an $h_j$-separated set $\{x_i^{(j)}\}_{i=1}^{N_j}$ of size $N_j$. Since $2h_j=\mathrm{diam}(\mathcal{X})\gamma^j/3$, such a set exists by Lemma~\ref{lem:sep_cover}. That is,
\begin{align*}
    d(x_{i}^{(j)},x_{i^\prime}^{(j)})\geq h_j
\end{align*}
holds for every $i\not=i^\prime$. Then, by Lemma~\ref{lem:grid}, the map $T_{h_j}$ embeds discrete distributions on $[N_j]$ to $\mathcal{P}(\mathcal{X})$ and satisfies
\begin{align*}
    W_p^p(T_{h_j}u,T_{h_j}v)\geq\frac{1}{2}h_j^p\lVert u-v\rVert_1.
\end{align*}
Now, let $u_j$ denote the uniform distribution on $[N_j]$. For $R>0$, suppose there exists an $\epsilon$-DP goodness-of-fit test $\phi$ with Type I and Type II errors at most $\alpha$ and $\beta$, respectively, for the following Wasserstein goodness-of-fit testing problem:
\begin{align*}
    H_0:P=T_{h_j}u_0\textrm{ vs }H_1:W_p^p(T_{h_j}u,T_{h_j}u_j)\geq \mathrm{diam}^p(\mathcal{X})R.
\end{align*}
Next, let $A_1,\ldots,A_n\overset{\textrm{i.i.d.}}{\sim}u$ for some $u\in\mathcal{P}([N_j])$ and consider the following $\epsilon$-DP goodness-of-fit testing problem on $[N_j]$.
\begin{align}\label{eq:reduced_l1_gof_test}
    H_0:u=u_j\textrm{ vs }H_1:\lVert u-u_j\rVert_1\geq 2h_j^{-p}\mathrm{diam}^p(\mathcal{X})R.
\end{align}
Consider a test $\phi(x_{A_1}^{(j)},\ldots,x_{A_n}^{(j)})$ induced from $\phi$ for \eqref{eq:reduced_l1_gof_test}. The test is $\epsilon$-DP since the mapping $A_i\mapsto x_{A_i}^{(j)}$ is deterministic and preserves adjacency. Moreover, Lemma~\ref{lem:grid} implies that its Type I and Type II errors are at most $\alpha$ and $\beta$, respectively. Using the fact that $2h_j^{-p}\mathrm{diam}^p(\mathcal{X})R$ cannot be smaller than the minimax separation radius for testing on $[N_j]$, we can show a lower bound for $R$.

We now derive a necessary lower bound on $R$. Set $J=J_{\mathrm{lower}}(R_{\mathrm{lower}}^{\mathrm{gof}})\geq1$. Then $R_{\mathrm{lower}}^{\mathrm{gof}}\leq 2\gamma^{Jp}$. Since the induced test solves \eqref{eq:reduced_l1_gof_test} at separation $2h_j^{-p}\mathrm{diam}^p(\mathcal X)R$, this separation must be at least the minimax separation radius of the corresponding finite-domain problem.
\begin{align*}
   2h_j^{-p}\mathrm{diam}^p(\mathcal{X})R&\geq C_2\left(\left(\frac{N_j^{1/4}}{\sqrt{n}}\vee\frac{N_j^{1/4}}{n^{3/4}\sqrt{\epsilon}}\vee \frac{N_j^{1/2}}{n\sqrt{\epsilon}}\vee\frac{1}{n\epsilon} \right)\wedge1\right),
\end{align*}
or, equivalently,
\begin{align*}
    R&\geq  2^{-1}6^{-p}\gamma^{jp}C_2\left(\left(\frac{N_j^{1/4}}{\sqrt{n}}\vee\frac{N_j^{1/4}}{n^{3/4}\sqrt{\epsilon}}\vee \frac{N_j^{1/2}}{n\sqrt{\epsilon}}\vee\frac{1}{n\epsilon} \right)\wedge 1\right)\\
    &= 2^{-1}6^{-p}\gamma^{jp}C_2\left(\frac{N_j^{1/4}}{\sqrt{n}}\vee\frac{N_j^{1/4}}{n^{3/4}\sqrt{\epsilon}}\vee \frac{N_j^{1/2}}{n\sqrt{\epsilon}}\vee\frac{1}{n\epsilon} \right)\wedge 2^{-1}6^{-p}C_2\gamma^{jp}\\
    &\geq2^{-1}6^{-p}\gamma^{jp}C_2\left(\frac{N_j^{1/4}}{\sqrt{n}}\vee\frac{N_j^{1/4}}{n^{3/4}\sqrt{\epsilon}}\vee \frac{N_j^{1/2}}{n\sqrt{\epsilon}}\vee\frac{1}{n\epsilon} \right)\wedge 2^{-1}6^{-p}C_2\gamma^{Jp}\\
    &\geq2^{-1}6^{-p}\gamma^{jp}C_2\left(\frac{N_j^{1/4}}{\sqrt{n}}\vee\frac{N_j^{1/4}}{n^{3/4}\sqrt{\epsilon}}\vee \frac{N_j^{1/2}}{n\sqrt{\epsilon}}\vee\frac{1}{n\epsilon} \right)\wedge 2^{-2}6^{-p}C_2R_{\mathrm{lower}}^{\mathrm{gof}}.
\end{align*}
Taking the maximum over $j=1,\ldots,J$, we get
\begin{align*}
    R&\geq  2^{-1}6^{-p}C_2\left(\frac{M(p,1/4,J)}{\sqrt{n}}\vee\frac{M(p,1/2,J)}{n\sqrt{\epsilon}}\vee \frac{M(p,1/4,J)}{n^{3/4}\sqrt{\epsilon}}\vee\frac{M(p,0,J)}{n\epsilon}\right)\wedge (2^{-2}6^{-p}C_2R_{\mathrm{lower}}^{\mathrm{gof}})\\
    &\geq 2^{-2}6^{-p}C_2R_{\mathrm{lower}}^{\mathrm{gof}},
\end{align*}
where the last inequality holds by the definition of $R_{\mathrm{lower}}^{\mathrm{gof}}$. Therefore, for any $R>0$, there exists a tester for $W_p^p(P,T_hu_0)\geq \mathrm{diam}^p(\mathcal{X})R$ only if $R\geq  2^{-2}6^{-p}C_2R_{\mathrm{lower}}^{\mathrm{gof}}$.

\end{proof}
\subsection{Two-sample testing}
The arguments developed in the previous section can also be applied to the two-sample testing problem. In this section, we characterize the minimax separation radius $\rho_{n,\epsilon}^{\mathrm{ts}}$ up to a logarithmic factor. As in goodness-of-fit testing, we construct a Wasserstein tester by combining minimax-optimal differentially private finite-domain $\ell_1$ two-sample testers with a calibrated privacy and test error budget allocation. The resulting procedure is given in Algorithm~\ref{alg:ts_tester_Wp}, and we subsequently derive upper and lower bounds on $\rho_{n,\epsilon}^{\mathrm{ts}}$.

\begin{proposition}\label{prop:ts_tv_minimax}
Let $k\geq2$ and $n\geq1$ be integers, and let $\epsilon>0$. Suppose that we observe $n$ i.i.d. samples from each of two unknown distributions $u,v\in\mathcal{P}([k])$, with the two samples independent. For given Type I and Type II error levels $\alpha,\beta\in(0,1/3]$, there exist a universal constant $C_3>0$ and a constant $C_4\in(0,1/2)$ depending only on $\alpha$ and $\beta$ with the following properties.
\begin{enumerate}
\renewcommand{\labelenumi}{(\roman{enumi})}
\item There exists an $\epsilon$-DP tester for
\begin{align*}
H_0:u=v
\quad\textrm{vs}\quad
H_1:\lVert u-v\rVert_1\geq \tau
\end{align*}
with Type I error at most $\alpha$ and Type II error at most $\beta$ if
\begin{align*}
\tau\geq C_3\max\left\{\frac{k^{1/4}}{\sqrt{n}},\frac{k^{1/2}}{n^{3/4}},\frac{k^{1/2}}{n\sqrt{\epsilon}},\frac{k^{1/4}}{n^{3/4}\sqrt{\epsilon}},\frac{1}{n\epsilon}\right\}\log\frac{1}{\alpha\wedge\beta}.
\end{align*}
\item Conversely, any $\epsilon$-DP tester with Type I and Type II errors at most $\alpha$ and $\beta$, respectively, satisfies
\begin{align*}
\tau\geq C_4\left(\max\left\{\frac{k^{1/4}}{\sqrt{n}},\frac{k^{1/2}}{n^{3/4}},\frac{k^{1/2}}{n\sqrt{\epsilon}},\frac{k^{1/4}}{n^{3/4}\sqrt{\epsilon}},\frac{1}{n\epsilon}\right\}\wedge1\right).
\end{align*}
\end{enumerate}
\end{proposition}

\begin{algorithm}
    \caption{$W_p$-based two-sample tester}\label{alg:ts_tester_Wp}
\begin{algorithmic}
  \Input:\textrm{  }Given data $X_1,\ldots,X_n$ and $Y_1,\ldots,Y_n$, Wasserstein order $p$, partition depth $J$, separation parameters $\tau_1,\ldots,\tau_J$, Type I and Type II error levels $\alpha$ and $\beta$, privacy parameter $\epsilon$, and allocation parameters $\nu_1,\ldots,\nu_J>0$ satisfying $\sum_{j=1}^J\nu_j\leq1$.
  \EndInput
\begin{enumerate}
    \item Generate a dyadic partition $\{\mathcal{R}^j\}_{j=1}^J$ with parameter $\gamma = 1/3$ such that $|\mathcal{R}^j|\leq\mathcal{N}(\mathcal{X},\mathrm{diam}(\mathcal{X})\gamma^j/3)$.
    \item \textbf{For }$j=1,\ldots,J$ \textbf{do}
    \begin{enumerate}
        \item Define $Z_i^j\coloneqq R$ if $X_i\in R$ and $R\in\mathcal{R}^j$.
        \item Define $W_i^j\coloneqq R$ if $Y_i\in R$ and $R\in\mathcal{R}^j$.
\item Apply the minimax-optimal $\nu_j\epsilon$-DP $\ell_1$-based tester $\phi_j$ from Proposition~\ref{prop:ts_tv_minimax} to two samples $Z_1^j,\ldots,Z^j_n$ and $W_1^j,\ldots,W^j_n$, using $\tau_j$ as the separation parameter, and $\nu_j\alpha$ and $\nu_j\beta$ as the Type I and Type II error levels, respectively.
\item \Return \textbf{Reject} if $\phi_j$ rejects; otherwise, proceed to the next level.
    \end{enumerate}
\end{enumerate}
 \Return \textbf{Accept}.
\end{algorithmic}
\end{algorithm}

We provide the upper and lower bounds on $\rho_{n,\epsilon}^{\mathrm{ts}}$ below.
\begin{theorem}\label{thm:two_sample_minimax_upper}
The minimax separation radius of $\epsilon$-DP two-sample testing satisfies
\begin{align*}
\rho_{n,\epsilon}^{\mathrm{ts}}\leq C_{\mathrm{upper}}^{\mathrm{ts}} \mathrm{diam}^p(\mathcal{X})R_{\mathrm{upper}}^{\mathrm{ts}},
\end{align*}
where $C_{\mathrm{upper}}^{\mathrm{ts}}>\gamma^{-p}/2$ is a constant depending only on $\alpha,\beta$, and $p$, and
\begin{align*}
    R_{\mathrm{upper}}^{\mathrm{ts}}&\coloneqq \inf\Bigg\{R\in(0,2\gamma^{p}]:R\geq \Psi^{\mathrm{ts}}_{\mathrm{upper}}(R)\Bigg\}\\
    \Psi^{\mathrm{ts}}_{\mathrm{upper}}(R)&\coloneqq \frac{\Lambda(p,1/4,1,J_{\mathrm{upper}}(R))}{\sqrt{n}}+\frac{\Lambda(p,1/2,1,J_{\mathrm{upper}}(R))}{n^{3/4}}\\
    &+\frac{\Lambda(p,1/4,3/2,J_{\mathrm{upper}}(R))}{n^{3/4}\sqrt{\epsilon}}+\frac{\Lambda(p,1/2,3/2,J_{\mathrm{upper}}(R))}{n\sqrt{\epsilon}}+\frac{\Lambda(p,0,2,J_{\mathrm{upper}}(R))}{n\epsilon}.
\end{align*}
If $R_{\mathrm{upper}}^{\mathrm{ts}}$ does not exist, we set $R_{\mathrm{upper}}^{\mathrm{ts}}=2\gamma^{p}$.
\end{theorem}

\begin{theorem}\label{thm:two_sample_minimax_lower}
    The minimax separation radius of $\epsilon$-DP two-sample testing satisfies
\begin{align*}
\rho_{n,\epsilon}^{\mathrm{ts}}\geq C_{\mathrm{lower}}^{\mathrm{ts}} \mathrm{diam}^p(\mathcal{X})R_{\mathrm{lower}}^{\mathrm{ts}},
\end{align*}
where $C_{\mathrm{lower}}^{\mathrm{ts}}>0$ is a constant depending only on $p,\alpha$, and $\beta$, and
\begin{align*}
   R_{\mathrm{lower}}^{\mathrm{ts}}&\coloneqq \sup\left\{R\in(0,2\gamma^p]:R\leq \Psi^{\mathrm{ts}}_{\mathrm{lower}}(R)\right\},\\
    \Psi^{\mathrm{ts}}_{\mathrm{lower}}(R)&\coloneqq \frac{M(p,1/4,J_{\mathrm{lower}}(R))}{\sqrt{n}}\vee\frac{M(p,1/2,J_{\mathrm{lower}}(R))}{n^{3/4}}\\
    &\vee\frac{M(p,1/2,J_{\mathrm{lower}}(R))}{n\sqrt{\epsilon}}\vee \frac{M(p,1/4,J_{\mathrm{lower}}(R))}{n^{3/4}\sqrt{\epsilon}}\vee\frac{M(p,0,J_{\mathrm{lower}}(R))}{n\epsilon}.
\end{align*}
\end{theorem}
As in the case of goodness-of-fit testing, the established lower and upper bounds match up to a logarithmic factor. In contrast to goodness-of-fit testing, the minimax separation radius of two-sample testing contains an additional term. This indicates that two-sample testing is harder than goodness-of-fit testing.

Since the proofs are analogous to those for goodness-of-fit testing, we defer the proofs to Appendices~\ref{appendix:two_sample_minimax_upper} and~\ref{appendix:two_sample_minimax_lower}.
\section{$W_p$-testing on $([0,1]^d,\ell_1)$} \label{sec:hypercube}
In this section, we specialize the preceding results to the case where $\mathcal{X}=[0,1]^d$ is equipped with the $\ell_1$ metric. In the previous sections, we established upper and lower bounds that match up to logarithmic factors in $n$. In the setting, these logarithmic gaps can be eliminated except for goodness-of-fit testing in the critical case $d=4p$.

\subsection{Explicit minimax separation rates}
Since $\mathcal{N}([0,1]^d,\ell_1,\delta)\asymp\delta^{-d}$ up to a constant factor depending only on $d$, the quantities appearing in the minimax separation bounds have the following orders of magnitude:
\begingroup
\setlength{\arraycolsep}{12pt}
\renewcommand{\arraystretch}{1.3}
\[
\begin{array}{c|ccc}
 & qd>p & qd=p & qd<p \\ \hline
\Sigma(p,q,r,J)
 & \gamma^{-(qd-p)J} & J^r & 1 \\
H(p,q,r,J)
 & 1 & \log J & 1 \\
\Lambda(p,q,r,J)
 & \gamma^{-(qd-p)J} & J^r(1+\log J) & 1 \\
M(p,q,J)
 & \gamma^{-(qd-p)J} & 1 & 1
\end{array}
\]
\endgroup
Each entry gives the order of the corresponding quantity in the sense of $\asymp$ for $J\geq2$, with comparison constants depending only on $d,p,q,r$ and $\gamma=1/3$ fixed.
Substituting these estimates into Theorems~\ref{thm:goodness_of_fit_minimax_upper}--\ref{thm:two_sample_minimax_lower} yields the following result. A detailed proof is deferred to Appendix~\ref{appendix:special_case}.
\begin{corollary}\label{cor:special_case} For $n\geq e^e$, the minimax separation radius of $\epsilon$-DP goodness-of-fit testing $\rho_{n,\epsilon}^{\mathrm{gof}}([0,1]^d,p)$ is characterized as
\begin{align*}
\rho_{n,\epsilon}^{\mathrm{gof}}([0,1]^d,p,\alpha,\beta)\lesssim\begin{cases}
\left(n^{-2p/d}\vee(n\sqrt{\epsilon})^{-\left(1\wedge \frac{2p}{d}\right)}\vee(n^{3/4}\sqrt{\epsilon})^{-\left(1\wedge \frac{4p}{d}\right)}\vee(n\epsilon)^{-1}\right)\wedge1, &d>4p,\\[6pt]
\left(n^{-1/2}\log n\log\log n \vee(n\sqrt{\epsilon})^{-1/2}\vee(n\epsilon)^{-1}\right)\wedge1,&d=4p,\\[6pt]
\left(n^{-1/2}\vee(n\sqrt{\epsilon})^{-\left(1\wedge \frac{2p}{d}\right)}\vee(n^{3/4}\sqrt{\epsilon})^{-\left(1\wedge \frac{4p}{d}\right)}\vee(n\epsilon)^{-1}\right)\wedge1,&d<4p,
\end{cases}
\end{align*}
and
\begin{align*}
\rho_{n,\epsilon}^{\mathrm{gof}}([0,1]^d,p,\alpha,\beta)\gtrsim\left(
n^{-\left(\frac{1}{2}\wedge\frac{2p}{d}\right)}\vee (n\sqrt{\epsilon})^{-\left(1\wedge \frac{2p}{d}\right)}\vee(n^{3/4}\sqrt{\epsilon})^{-\left(1\wedge \frac{4p}{d}\right)}\vee(n\epsilon)^{-1}\right)\wedge1,
\end{align*}
up to constant factors depending only on $d,p,\alpha$, and $\beta$. In contrast, the minimax separation radius of $\epsilon$-DP two-sample testing satisfies
\begin{align*}
\rho_{n,\epsilon}^{\mathrm{ts}}([0,1]^d,p)\asymp\left(
n^{-\left(\frac{1}{2}\wedge\frac{3p}{2d}\right)}\vee (n\sqrt{\epsilon})^{-\left(1\wedge \frac{2p}{d}\right)}\vee(n^{3/4}\sqrt{\epsilon})^{-\left(1\wedge \frac{4p}{d}\right)}\vee(n\epsilon)^{-1}\right)\wedge1,
\end{align*}
up to constant factors depending only on $d,p,\alpha$, and $\beta$.
\end{corollary}
We make several remarks on Corollary~\ref{cor:special_case}. First, our upper and lower bounds are formulated in terms of the covering numbers of the underlying data space. Therefore, the same conclusions extend beyond $([0,1]^d,\ell_1)$ to metric spaces whose covering numbers exhibit polynomial growth of order $d$, namely, $\mathcal{N}(\mathcal{X},\delta)\asymp\delta^{-d}$. Ahlfors regularity provides a useful sufficient condition for such polynomial growth. A Borel measure $\mu$ on $\mathcal{X}$ is said to be Ahlfors $d$-regular if there exists a constant $C>0$ such that
\begin{align*}
C^{-1}r^d\leq \mu(B(x,r))\leq Cr^d
\end{align*}
for all $x\in \mathcal{X}$ and $0<r\leq\mathrm{diam}(\mathcal{X})$  \citep{heinonen1998quasiconformal}. If $\mathcal{X}$ admits such a measure, then $\mathcal{N}(\mathcal{X},\delta)\asymp\delta^{-d}$. This condition is satisfied by many standard data spaces, including Euclidean cubes and balls, as well as compact Riemannian manifolds.

Next, we compare our results with existing work on Wasserstein testing over $[0,1]^d$ \citep{ba2011sublinear,Ferreira2024Distribution,deng2017wasserstein}. The settings considered in these works fall within the framework of Corollary~\ref{cor:special_case}, but are limited to the $p=1$ setting. Although the underlying metric may be different from $\ell_1$ (e.g., $\ell_s$ with $s>1$), the theorem continues to apply because these metrics induce covering numbers of the same order, up to constant factors, on $[0,1]^d$. Moreover, the bounds in Corollary~\ref{cor:special_case} match up to constant factors except in the critical case $d=4p$, whereas the corresponding bounds in the existing literature contain additional logarithmic factors across the different values of $d$. This improvement comes from the carefully calibrated allocation parameters $\nu_1,\ldots,\nu_J$. In the previous work, the parameters were chosen uniformly, i.e., $\nu_1=\cdots=\nu_J=1/J$, which yields the additional logarithmic factors. However, the allocation parameters used in the proof of Theorem~\ref{thm:goodness_of_fit_minimax_upper} are designed to assign larger budgets to resolution scales at which the testing problem is more difficult. Indeed, the quantities of the form $(N_j^q\gamma^{jp})^{1/r}$ appearing in the allocation reflect the minimax separation radius of the testing problem at resolution $j$. Consequently, the proposed allocation assigns relatively larger values of $\nu_j$ to harder-to-test resolution scales, rather than distributing the budget uniformly across all scales. As a result, logarithmic factors remain only in the critical regimes $d=2p$ and $d=4p$. However, except for goodness-of-fit testing when $d=4p$, every term involving a logarithmic factor is dominated by another term in the corresponding bound and therefore does not affect the final rate.

\subsection{Necessity of a logarithmic factor}
For goodness-of-fit testing at the critical dimension $d=4p$, a logarithmic factor remains in our upper bound. We show that a logarithmic factor is unavoidable, at least in the case $(d,p)=(4,1)$. In fact, we show a stronger result. For $B>1$, denote
\begin{align*}
    \mathcal{P}_B([0,1]^d)&\coloneqq\left\{P\in\mathcal{P}([0,1]^d):\frac{dP}{d\lambda}\leq B\right\},\\
    \mathcal{P}_{1}^{\mathrm{gof}}(P_0,[0,1]^d,B,\rho)&\coloneqq\mathcal{P}_B([0,1]^d)\cap \{P\in\mathcal{P}([0,1]^d):W_1(P,P_0)\geq \rho\},
\end{align*}
where $\lambda$ is the Lebesgue measure on $\mathbb{R}^d$. Define the minimax separation radius of Wasserstein goodness-of-fit testing for distributions with bounded density on $[0,1]^d$ as follows.
\begin{definition}
    Let $\phi_{P_0}:([0,1]^d)^n\rightarrow\{0,1\}$ be a goodness-of-fit testing procedure for \eqref{eq:ht_gof}. Given Type I and Type II error levels $\alpha,\beta\in(0,1/3]$, density upper bound $B$, and sample size $n$, the separation radius of $\phi_{P_0}$ is defined as
\begin{align*}
    \rho_{n}^{\mathrm{gof}}(\phi_{P_0},[0,1]^d,B,\alpha,\beta)\coloneqq \inf\left\{\rho>0:\mathbb{E}_{P_0}[\phi_{P_0}]\leq\alpha,\sup_{P\in\mathcal{P}_{1}^{\mathrm{gof}}(P_0,[0,1]^d,B,\rho)}\mathbb{E}_P[1-\phi_{P_0}]\leq\beta\right\},
\end{align*}
The minimax separation radius for goodness-of-fit testing is then defined as
\begin{align*} \rho_{n}^{\mathrm{gof}}([0,1]^d,B,\alpha,\beta)\coloneqq \sup_{P_0\in\mathcal{P}_B([0,1]^d)}\inf_{\phi_{P_0}}\rho_{n}^{\mathrm{gof}}(\phi_{P_0},[0,1]^d,B,\alpha,\beta).\end{align*}
\end{definition}
We establish the minimax separation radius $\rho_{n}^{\mathrm{gof}}([0,1]^d,B,\alpha,\beta)$ below. The proof can be found in Appendix~\ref{appendix:unavoidable_log}.
\begin{theorem}\label{thm:unavoidable_log}
For $n\geq e$, the minimax separation radius of Wasserstein-1 goodness-of-fit testing for distributions with bounded densities on $[0,1]^d$ is characterized as
    \begin{align*}
\rho_{n}^{\mathrm{gof}}([0,1]^d,B,\alpha,\beta)\asymp \begin{cases}
        n^{-1/2}, & d<4,\\
        n^{-1/2}\log^{1/4}n , & d=4,\\
        n^{-2/d}, & d>4,
    \end{cases}
    \end{align*}
up to constant factors depending only on $d,B,\alpha$, and $\beta$.
\end{theorem}
Since $\mathcal{P}_B([0,1]^4)\subset\mathcal{P}([0,1]^4)$ and private testing is harder than non-private testing, Theorem~\ref{thm:unavoidable_log} yields the following corollary.
\begin{corollary}\label{cor:unavoidable_log}
For $n\geq e$, the minimax separation radius of $\epsilon$-DP Wasserstein-1 goodness-of-fit testing on $[0,1]^4$ satisfies
    \begin{align*}
        \rho_{n,\epsilon}^{\mathrm{gof}}([0,1]^4,1,\alpha,\beta)\gtrsim n^{-1/2}\log^{1/4}n,
    \end{align*}
up to a constant factor depending only on $\alpha$ and $\beta$.
\end{corollary}
In comparison, \citet{deng2017wasserstein} study the corresponding non-private problem and obtain a lower bound without a logarithmic factor. Their lower-bound argument reduces the Wasserstein testing problem at a single resolution to an $\ell_1$ identity testing problem, and therefore cannot capture the accumulation of discrepancies across resolution levels. Our lower bound is genuinely multiscale: a sharp analytical lower bound on the Wasserstein-1 distance in terms of negative-order Sobolev quantities allows perturbations at multiple resolutions to contribute simultaneously. At the critical dimension, these contributions accumulate to produce the logarithmic factor. Combining Corollary~\ref{cor:special_case} with Corollary~\ref{cor:unavoidable_log} therefore gives a sharper picture of the logarithmic factors in our bounds: they can be removed when $d\neq 4p$, whereas at the critical dimension they can be unavoidable, as demonstrated by the case $d=4p=4$. To the best of our knowledge, this is the first instance in Wasserstein goodness-of-fit testing where such a logarithmic factor is shown to be necessary by a minimax lower bound. Similar logarithmic gaps between upper and lower minimax bounds arise elsewhere in Wasserstein estimation \citep[e.g.,][]{singh2018minimax,niles2022minimax}. In these problems as well, existing lower bounds are based on perturbations at a single resolution and do not capture possible accumulation across scales. Our result suggests that sharp multiscale lower-bound arguments using negative-order Sobolev norms to aggregate contributions across resolutions may help determine whether these logarithmic gaps are intrinsic.

\section{$W_p$-testing on unbounded domains} \label{sec:unbounded}
As shown at the beginning of Section~\ref{sec:main}, Wasserstein testing is infeasible on unbounded domains. However, when the distributions are known to have sufficiently light tails, such as exponentially decaying tails, it is possible to design tests at the cost of logarithmic factors. We now consider $\mathcal{X}=\mathbb{R}^d$ equipped with the $\ell_1$ metric. For $c_1\geq1$ and $c_2>0$, define a class of light-tailed distributions over $\mathcal{X}$ as
\begin{align*}
    \mathcal{P}(\mathcal{X},c_1,c_2)\coloneqq \{P\in\mathcal{P}(\mathcal{X}):P(B(r)^c)\leq c_1e^{-c_2r}\textrm{ for all }r\geq0\},
\end{align*}
where $B(r)$ denotes the closed $\ell_1$ ball of radius $r$ centered at the origin.

We next define the minimax separation radii of $\epsilon$-DP goodness-of-fit and two-sample testing under this exponential tail condition. Denote
\begin{align*}
    \mathcal{P}_0^{\mathrm{gof}}(c_1,c_2)&\coloneqq \mathcal{P}(\mathcal{X},c_1,c_2)\\[6pt]
    \mathcal{P}_1^{\mathrm{gof}}(P_0,\rho,p,c_1,c_2)&\coloneqq \{P\in\mathcal{P}(\mathcal{X},c_1,c_2):W_p^p(P,P_0)\geq\rho\}\\[6pt]
    \mathcal{P}_0^{\mathrm{ts}}(c_1,c_2)&\coloneqq \{(P,Q):P=Q\in\mathcal{P}(\mathcal{X},c_1,c_2)\}\\[6pt]
    \mathcal{P}_1^{\mathrm{ts}}(\rho,p,c_1,c_2)&\coloneqq \{(P,Q):W_p^p(P,Q)\geq\rho,P,Q\in\mathcal{P}(\mathcal{X},c_1,c_2)\}.
\end{align*}
\begin{definition}
For given Type I and Type II error levels $\alpha,\beta\in(0,1/3]$,
$p\geq 1$, sample size $n$, privacy parameter $\epsilon>0$,
and tail decay parameters $c_1\geq 1$ and $c_2>0$, the minimax separation radius of $\epsilon$-DP goodness-of-fit testing for exponentially decaying distributions is defined as
\begin{align*}
\rho_{n,\epsilon}^{\mathrm{gof}}
(\mathcal{X},p,c_1,c_2,\alpha,\beta)
\coloneqq
\inf_{\phi_{P_0}:\,\epsilon\text{-DP}}
\inf\Bigg\{\rho>0 &:\
\sup_{P_0\in\mathcal{P}_0^{\mathrm{gof}}(c_1,c_2)}
\mathbb{E}_{P_0}[\phi_{P_0}]\leq\alpha,\\
&\
\sup_{P_0\in\mathcal{P}_0^{\mathrm{gof}}(c_1,c_2)}
\sup_{P\in\mathcal{P}_1^{\mathrm{gof}}(P_0,\rho,p,c_1,c_2)}
\mathbb{E}_{P}[1-\phi_{P_0}]\leq\beta
\Bigg\}.
\end{align*}
Similarly, the minimax separation radius of $\epsilon$-DP two-sample testing for exponentially decaying distributions is defined as
\begin{align*}
\rho_{n,\epsilon}^{\mathrm{ts}}
(\mathcal{X},p,c_1,c_2,\alpha,\beta)
\coloneqq
\inf_{\phi:\,\epsilon\text{-DP}}
\inf\Bigg\{\rho>0 &:\
\sup_{(P,Q)\in\mathcal{P}_0^{\mathrm{ts}}(c_1,c_2)}
\mathbb{E}_{P,Q}[\phi]\leq\alpha,\\
&\
\sup_{(P,Q)\in\mathcal{P}_1^{\mathrm{ts}}(\rho,p,c_1,c_2)}
\mathbb{E}_{P,Q}[1-\phi]\leq\beta
\Bigg\}.
\end{align*}
\end{definition}
The minimax separation radii can be bounded using Theorems~\ref{thm:goodness_of_fit_minimax_upper}-\ref{thm:two_sample_minimax_lower} by reducing testing on the unbounded domain to testing on a bounded domain. For $L>0$, define the clipping function by
\begin{align*}
    \mathrm{clip}(x,L)\coloneqq 
    \begin{cases}
        x, & \lVert x\rVert_1\leq L,\\[6pt]
        Lx/\lVert x\rVert_1, & \lVert x\rVert_1>L.
    \end{cases}
\end{align*}
For a distribution $P$ satisfying $P(B(L))>0$, let $P_L\coloneqq P(\,\cdot\mid B(L))$ denote its truncation to $B(L)$. Applying the $\epsilon$-DP testers developed above to the clipped observations yields $\epsilon$-DP goodness-of-fit and two-sample tests for exponentially decaying distributions. To control their errors, we use the following lemma, which shows that truncation does not substantially reduce the Wasserstein distance when $L$ is sufficiently large.
\begin{lemma}\label{lem:trunc_Wp}
Let $P,Q\in\mathcal{P}(\mathcal{X},c_1,c_2)$ and $L>0$, and denote by $P_L$ and $Q_L$ their truncated distributions on $B(L)$. If $P(B(L)),Q(B(L))>0$ and $3^pe^{-c_2L}<1$, then
\begin{align*}
  W_p^p(P,Q)\leq P(B(L))Q(B(L))W_p^p(P_L,Q_L)+\frac{4c_1(6L)^pe^{-c_2L}}{1-3^pe^{-c_2L}}.
\end{align*}
\end{lemma}
The proof of Lemma~\ref{lem:trunc_Wp} is deferred to Appendix~\ref{appendix:trunc_Wp}. Let $\mathcal{E}_L$ denote the event that all observations lie in $B(L)$. Conditional on $\mathcal{E}_L$, the observations are independent samples from the corresponding truncated distributions $P_L$ and $Q_L$, and clipping leaves them unchanged. Moreover, the exponential tail condition and a union bound give
\begin{align*}
    \mathbb{P}(\mathcal{E}_L^c)\leq 2nc_1e^{-c_2L}
\end{align*}
for the two-sample problem, with the analogous bound $nc_1e^{-c_2L}$ for goodness-of-fit testing. Thus, choosing $L=O(\log n)$, the additional testing error due to $\mathcal{E}_L^c$ becomes negligible. Lemma~\ref{lem:trunc_Wp} then ensures that the Wasserstein separation is preserved up to a negligible truncation error. Applying Theorems~\ref{thm:goodness_of_fit_minimax_upper} and~\ref{thm:two_sample_minimax_upper} on $B(L)$ yields the upper bounds of $\rho_{n,\epsilon}^{\mathrm{gof}}(\mathbb{R}^d,p,c_1,c_2,\alpha,\beta)$ and $\rho_{n,\epsilon}^{\mathrm{ts}}(\mathbb{R}^d,p,c_1,c_2,\alpha,\beta)$; the factor $(\log n)^p$ arises from the diameter of $B(L)$. The lower bounds follow from the same reduction of Wasserstein testing to $\ell_1$-distance testing used in the bounded case. The proof details are deferred to Appendix~\ref{appendix:minimax_unbounded}.

\begin{theorem}\label{thm:minimax_unbounded}
For $n\geq e$, the minimax separation radius of $\epsilon$-DP goodness-of-fit testing for exponentially decaying distributions $\rho_{n,\epsilon}^{\mathrm{gof}}(\mathbb{R}^d,p,c_1,c_2,\alpha,\beta)$ satisfies
\begin{align*}
\rho_{n,\epsilon}^{\mathrm{gof}}(\mathbb{R}^d,p,c_1,c_2,\alpha,\beta)&\lesssim (\log n)^p \rho_{n,\epsilon}^{\mathrm{gof}}([0,1]^d,p,\alpha/2,\beta/2),\\[6pt]
\rho_{n,\epsilon}^{\mathrm{gof}}(\mathbb{R}^d,p,c_1,c_2,\alpha,\beta)&\gtrsim \left(n^{-\left(\frac{1}{2}\wedge\frac{2p}{d}\right)}\vee (n\sqrt{\epsilon})^{-\left(1\wedge \frac{2p}{d}\right)}\vee(n^{3/4}\sqrt{\epsilon})^{-\left(1\wedge \frac{4p}{d}\right)}\vee(n\epsilon)^{-1}\right)\wedge1,
\end{align*}
    up to constant factors depending only on $d,p,c_1,c_2,\alpha$, and $\beta$. The minimax separation radius of $\epsilon$-DP two-sample testing for exponentially decaying distributions $\rho_{n,\epsilon}^{\mathrm{ts}}(\mathbb{R}^d,p,c_1,c_2,\alpha,\beta)$ satisfies
    \begin{align*}
        \rho_{n,\epsilon}^{\mathrm{ts}}(\mathbb{R}^d,p,c_1,c_2,\alpha,\beta)&\lesssim (\log n)^p \rho_{n,\epsilon}^{\mathrm{ts}}([0,1]^d,p,\alpha/2,\beta/2),\\[6pt]
        \rho_{n,\epsilon}^{\mathrm{ts}}(\mathbb{R}^d,p,c_1,c_2,\alpha,\beta)&\gtrsim  \left(n^{-\left(\frac{1}{2}\wedge\frac{3p}{2d}\right)}\vee (n\sqrt{\epsilon})^{-\left(1\wedge \frac{2p}{d}\right)}\vee(n^{3/4}\sqrt{\epsilon})^{-\left(1\wedge \frac{4p}{d}\right)}\vee(n\epsilon)^{-1}\right)\wedge1,
    \end{align*}
    up to constant factors depending only on $d,p,c_1,c_2,\alpha$, and $\beta$.
\end{theorem}

\section{Discussion} \label{sec:discussion}
Our work leaves several open questions. First, our results show that minimax-optimal differentially private Wasserstein tests can be constructed by combining our multiscale reduction with existing differentially private $\ell_1$ tests for goodness-of-fit and two-sample testing. Can a similar approach be developed for other distribution testing problems, such as independence testing or two-sample testing with unequal sample sizes, or under other privacy models such as local differential privacy? Second, our results provide a general characterization of the difficulty of private Wasserstein testing through the multiscale geometry of the underlying metric space. In particular, the upper and lower bounds are governed by quantities built from the covering numbers of the underlying space $\mathcal{X}$, such as $\Lambda$ and $M$. These quantities characterize the minimax separation radii up to logarithmic factors. An important question is whether the upper- and lower-bound quantities can be further unified to characterize the minimax separation radii up to constant factors. Third, we demonstrate that most of the logarithmic gaps in previous works are proof artifacts, while showing that a logarithmic factor in at least one critical case, $(d,p)=(4,1)$, is unavoidable. Are logarithmic factors in the other critical cases also unavoidable? Even for $(d,p)=(4,1)$, there remains a gap between the logarithmic factors in the lower and upper bounds: the lower bound scales as $n^{-1/2}\log^{1/4} n$, whereas the upper bound scales as $n^{-1/2}\log n\log\log n$. Can this gap be closed? Finally, under the exponential tail condition, our results for unbounded domains are minimax-optimal only up to logarithmic factors. Can these logarithmic factors be removed? Resolving these questions would provide a more complete understanding of the fundamental limits of Wasserstein testing and the inherent trade-offs under different privacy constraints.

\section*{Acknowledgments}

Cheolwoo Park’s work was supported in part by Basic Science Research Program of the National Research Foundation of Korea (NRF) funded by the Ministry of Education (RS-2022-NR068758, RS-2026-25517529, RS-2026-25616677).

\bibliography{reference}

@article{kim2026sharp,
  title={Sharp Minimax Rates for Smooth Two-Sample Testing under Central Differential Privacy},
  author={Kim, Ilmun},
  journal={arXiv preprint arXiv:2607.23974},
  year={2026}
}

@article{heinonen1998quasiconformal,
  author  = {Heinonen, Juha and Koskela, Pekka},
  title   = {Quasiconformal Maps in Metric Spaces with Controlled Geometry},
  journal = {Acta Mathematica},
  volume  = {181},
  number  = {1},
  pages   = {1--61},
  year    = {1998},
  doi     = {10.1007/BF02392747},
  url     = {https://doi.org/10.1007/BF02392747}
}

@article{Bobkov2021simple,
author = {Sergey G. Bobkov and Michel Ledoux},
title = {{A simple Fourier analytic proof of the AKT optimal matching theorem}},
volume = {31},
journal = {The Annals of Applied Probability},
number = {6},
publisher = {Institute of Mathematical Statistics},
pages = {2567 -- 2584},
year = {2021},
doi = {10.1214/20-AAP1656},
URL = {https://doi.org/10.1214/20-AAP1656}
}

@article{Iacobelli2024stability,
author = {Iacobelli, Mikaela and Junné, Jonathan},
title = {Stability estimates for the {Vlasov}--{Poisson} system in $p$-kinetic {Wasserstein} distances},
journal = {Bulletin of the London Mathematical Society},
volume = {56},
number = {7},
pages = {2250-2267},
doi = {https://doi.org/10.1112/blms.13053},
url = {https://londmathsoc.onlinelibrary.wiley.com/doi/abs/10.1112/blms.13053},
eprint = {https://londmathsoc.onlinelibrary.wiley.com/doi/pdf/10.1112/blms.13053},
year = {2024}
}

@article{ambrosio2019pde,
  author  = {Ambrosio, Luigi and Stra, Federico and Trevisan, Dario},
  title   = {A {PDE} approach to a 2-dimensional matching problem},
  journal = {Probability Theory and Related Fields},
  volume  = {173},
  pages   = {433--477},
  year    = {2019},
  doi     = {10.1007/s00440-018-0837-x},
  url     = {https://doi.org/10.1007/s00440-018-0837-x}
}

@article{Daubechies1988Orthonormal,
author = {Daubechies, Ingrid},
title = {Orthonormal bases of compactly supported wavelets},
journal = {Communications on Pure and Applied Mathematics},
volume = {41},
number = {7},
pages = {909-996},
doi = {https://doi.org/10.1002/cpa.3160410705},
url = {https://onlinelibrary.wiley.com/doi/abs/10.1002/cpa.3160410705},
eprint = {https://onlinelibrary.wiley.com/doi/pdf/10.1002/cpa.3160410705},
year = {1988}
}

@book{Triebel1983Theory,
 author = {Triebel, Hans},
 title = {Theory of function spaces},
 fseries = {Monographs in Mathematics},
 series = {Monogr. Math., Basel},
 issn = {1017-0480},
 volume = {78},
 year = {1983},
 publisher = {Birkh{\"a}user, Cham},
 language = {English},
 zbMATH = {3870001},
 Zbl = {0546.46027}
}

@book{evans2015measure,
  title     = {Measure Theory and Fine Properties of Functions},
  author    = {Evans, Lawrence C. and Gariepy, Ronald F.},
  year      = {2015},
  edition   = {Revised},
  publisher = {Chapman and Hall/CRC},
  doi       = {10.1201/b18333}
}

@article{niles2022minimax,
  title={Minimax estimation of smooth densities in {Wasserstein} distance},
  author={Niles-Weed, Jonathan and Berthet, Quentin},
  journal={The Annals of Statistics},
  volume={50},
  number={3},
  pages={1519--1540},
  year={2022},
  publisher={JSTOR}
}

@book{evans2010partial,
  title={Partial differential equations},
  author={Evans, Lawrence C},
  volume={},
  year={2010},
  publisher={American Mathematical Society}
}

@article{singh2018minimax,
  title={Minimax distribution estimation in {Wasserstein} distance},
  author={Singh, Shashank and P{\'o}czos, Barnab{\'a}s},
  journal={arXiv preprint arXiv:1802.08855},
  year={2018}
}

@book{meyer1992wavelets,
  title     = {Wavelets and Operators: Volume 1},
  author    = {Meyer, Yves},
  year      = {1992},
  series    = {Cambridge Studies in Advanced Mathematics},
  publisher = {Cambridge University Press},
  isbn      = {9780521458696}
}

@book{zhang2021statistical,
  title={Statistical inference in the differential privacy model},
  author={Zhang, Huanyu},
  year={2021},
  publisher={Cornell University}
}

@inproceedings{improved_WGANs,
 author = {Gulrajani, Ishaan and Ahmed, Faruk and Arjovsky, Martin and Dumoulin, Vincent and Courville, Aaron},
 booktitle = {Advances in Neural Information Processing Systems},
 pages = {},
 publisher = {Curran Associates, Inc.},
title = {Improved Training of {Wasserstein} {GANs}},
 url = {https://proceedings.neurips.cc/paper_files/paper/2017/file/892c3b1c6dccd52936e27cbd0ff683d6-Paper.pdf},
 volume = {30},
 year = {2017}
}

@InProceedings{narayanan22a,
  title = 	 {Private High-Dimensional Hypothesis Testing},
  author =       {Narayanan, Shyam},
  booktitle = 	 {Proceedings of Thirty Fifth Conference on Learning Theory},
  pages = 	 {3979--4027},
  year = 	 {2022},
  editor = 	 {Loh, Po-Ling and Raginsky, Maxim},
  volume = 	 {178},
  series = 	 {Proceedings of Machine Learning Research},
  month = 	 {02--05 Jul},
  publisher =    {PMLR},
  url = 	 {https://proceedings.mlr.press/v178/narayanan22a.html}
}

@article{ba2011sublinear,
  author  = {Ba, Khanh Do and Nguyen, Huy L. and Nguyen, Huy N. and Rubinfeld, Ronitt},
  title   = {Sublinear Time Algorithms for Earth Mover's Distance},
  journal = {Theory of Computing Systems},
  volume  = {48},
  pages   = {428--442},
  year    = {2011},
  doi     = {10.1007/s00224-010-9265-8},
  url     = {https://doi.org/10.1007/s00224-010-9265-8}
}

@article{10.1214/aos/1017938923,
author = {Juan A. Cuesta-Albertos and Carlos Matr{\'a}n and Jes{\'u}s M. Rodr{\'i}guez-Rodr{\'i}guez and Eustasio del Barrio},
title = {Tests of goodness of fit based on the {$L_2$}-{Wasserstein} distance},
volume = {27},
journal = {The Annals of Statistics},
number = {4},
publisher = {Institute of Mathematical Statistics},
pages = {1230 -- 1239},
year = {1999},
doi = {10.1214/aos/1017938923},
URL = {https://doi.org/10.1214/aos/1017938923}
}

@article{10.1093/biomet/asaf001,
    author = {Hu, Xiaoyu and Lin, Zhenhua},
title = {Two-sample distribution tests in high dimensions via max-sliced {Wasserstein} distance and bootstrapping},
    journal = {Biometrika},
    volume = {112},
    number = {2},
    pages = {asaf001},
    year = {2025},
    month = {04},
    issn = {1464-3510},
    doi = {10.1093/biomet/asaf001},
    url = {https://doi.org/10.1093/biomet/asaf001},
    eprint = {https://academic.oup.com/biomet/article-pdf/112/2/asaf001/61530914/asaf001.pdf},
}

@article{10.3150/bj/1110228245,
author = {del Barrio, Eustasio and Gin{\'e}, Evarist and Utzet, Frederic},
title = {{Asymptotics for $L_2$ functionals of the empirical quantile process, with applications to tests of fit based on weighted Wasserstein distances}},
volume = {11},
journal = {Bernoulli},
number = {1},
publisher = {Bernoulli Society for Mathematical Statistics and Probability},
pages = {131 -- 189},
year = {2005},
doi = {10.3150/bj/1110228245},
URL = {https://doi.org/10.3150/bj/1110228245}
}

@article{91101a87-ac1a-38eb-b389-231b14fc1446,
 ISSN = {13507265},
 URL = {http://www.jstor.org/stable/23238612},
 author={Alvarez-Esteban, Pedro C and Del Barrio, Eustasio and Cuesta-Albertos, Juan A and Matr{\'a}n, Carlos},
 journal = {Bernoulli},
 number = {2},
 pages = {606--634},
 publisher = {International Statistical Institute (ISI) and Bernoulli Society for Mathematical Statistics and Probability},
 title = {Similarity of samples and trimming},
 urldate = {2026-08-18},
 volume = {18},
 year = {2012}
}

@article{Alvarez-Esteban01062008,
author={Alvarez-Esteban, Pedro C{\'e}sar and Del Barrio, Eustasio and Cuesta-Albertos, Juan Antonio and Matr{\'a}n, Carlos},
title = {Trimmed Comparison of Distributions},
journal = {Journal of the American Statistical Association},
volume = {103},
number = {482},
pages = {697--704},
year = {2008},
publisher = {Taylor \& Francis},
doi = {10.1198/016214508000000274},


URL = { 
    
        https://doi.org/10.1198/016214508000000274
    
    

},
eprint = { 
    
        https://doi.org/10.1198/016214508000000274
    
    

}

}

@article{10.1111/1467-9868.00121,
    author = {Munk, Axel and Czado, Claudia},
    title = {Nonparametric Validation of Similar Distributions and Assessment of Goodness of Fit},
    journal = {Journal of the Royal Statistical Society Series B: Statistical Methodology},
    volume = {60},
    number = {1},
    pages = {223-241},
    year = {1998},
    month = {01},
    issn = {1369-7412},
    doi = {10.1111/1467-9868.00121},
    url = {https://doi.org/10.1111/1467-9868.00121},
    eprint = {https://academic.oup.com/jrsssb/article-pdf/60/1/223/49589051/jrsssb_60_1_223.pdf},
}

@article{FREITAG2007697,
title = {A nonparametric test for similarity of marginals—With applications to the assessment of population bioequivalence},
journal = {Journal of Statistical Planning and Inference},
volume = {137},
number = {3},
pages = {697-711},
year = {2007},
note = {Special Issue on Nonparametric Statistics and Related Topics: In honor of M.L. Puri},
issn = {0378-3758},
doi = {https://doi.org/10.1016/j.jspi.2006.06.003},
url = {https://www.sciencedirect.com/science/article/pii/S0378375806001352},
author = {Gudrun Freitag and Claudia Czado and Axel Munk}
}

@inproceedings{NEURIPS2024_04cc90ec,
 author = {Omer, Daniel and Sheffet, Or},
 booktitle = {Advances in Neural Information Processing Systems},
 doi = {10.52202/079017-0082},
 editor = {A. Globerson and L. Mackey and D. Belgrave and A. Fan and U. Paquet and J. Tomczak and C. Zhang},
 pages = {2507--2527},
 publisher = {Curran Associates, Inc.},
 title = {Differentially Private Equivalence Testing for Continuous Distributions and Applications},
 url = {https://proceedings.neurips.cc/paper_files/paper/2024/file/04cc90ec6868b97b7423dc38ced1e35c-Paper-Conference.pdf},
 volume = {37},
 year = {2024}
}

@article{delBarrio2000Contributions,
  author  = {del Barrio, Eustasio and Cuesta-Albertos, Juan A. and Matr{\'a}n, Carlos and Cs{\"o}rg{\"o}, S{\'a}ndor and Cuadras, Carles M. and de Wet, Tertius and Gin{\'e}, Evarist and Lockhart, Richard and Munk, Axel and Stute, Winfried},
  title   = {Contributions of Empirical and Quantile Processes to the Asymptotic Theory of Goodness-of-Fit Tests},
  journal = {Test},
  volume  = {9},
  pages   = {1--96},
  year    = {2000},
  doi     = {10.1007/BF02595852},
  url     = {https://doi.org/10.1007/BF02595852}
}

@Article{Ramdas2017,
AUTHOR = {Ramdas, Aaditya and Garc{\'i}a Trillos, Nicol{\'a}s and Cuturi, Marco},
TITLE = {On {Wasserstein} Two-Sample Testing and Related Families of Nonparametric Tests},
PAGES = {47},
JOURNAL = {Entropy},
VOLUME = {19},
YEAR = {2017},
NUMBER = {2},
ARTICLE-NUMBER = {47},
URL = {https://www.mdpi.com/1099-4300/19/2/47},
ISSN = {1099-4300},
DOI = {10.3390/e19020047}
}

@article{Wasserstein_CI,
    author = {Neykov, Matey and Wasserman, Larry and Kim, Ilmun and Balakrishnan, Sivaraman},
    title = {Nearly minimax optimal {Wasserstein} conditional independence testing},
    journal = {Information and Inference: A Journal of the IMA},
    volume = {13},
    number = {4},
    pages = {iaae033},
    year = {2024},
    month = {12},
    issn = {2049-8772},
    doi = {10.1093/imaiai/iaae033},
    url = {https://doi.org/10.1093/imaiai/iaae033},
    eprint = {https://academic.oup.com/imaiai/article-pdf/13/4/iaae033/60983992/iaae033.pdf},
}

@article{panaretos2019statistical,
  title={Statistical aspects of {Wasserstein} distances},
  author={Panaretos, Victor M and Zemel, Yoav},
  journal={Annual Review of Statistics and Its Application},
  volume={6},
  number={1},
  pages={405--431},
  year={2019},
  publisher={Annual Reviews}
}

@InProceedings{word_moving_distance,
  title = 	 {From Word Embeddings To Document Distances},
  author = 	 {Kusner, Matt and Sun, Yu and Kolkin, Nicholas and Weinberger, Kilian},
  booktitle = 	 {Proceedings of the 32nd International Conference on Machine Learning},
  pages = 	 {957--966},
  year = 	 {2015},
  editor = 	 {Bach, Francis and Blei, David},
  volume = 	 {37},
  series = 	 {Proceedings of Machine Learning Research},
  address = 	 {Lille, France},
  month = 	 {07--09 Jul},
  publisher =    {PMLR},
  url = 	 {https://proceedings.mlr.press/v37/kusnerb15.html}
}

@article{weed2019sharp,
  title={Sharp asymptotic and finite-sample rates of convergence of empirical measures in {Wasserstein} distance},
  author={Weed, Jonathan and Bach, Francis},
  journal={Bernoulli},
  volume={25},
  number={4A},
  pages={2620--2648},
  year={2019},
  publisher={JSTOR}
}

@inproceedings{dwork2006calibrating,
  title={Calibrating noise to sensitivity in private data analysis},
  author={Dwork, Cynthia and McSherry, Frank and Nissim, Kobbi and Smith, Adam},
  booktitle={Theory of Cryptography: Third Theory of Cryptography Conference, TCC 2006, New York, NY, USA, March 4-7, 2006. Proceedings 3},
  pages={265--284},
  year={2006},
  organization={Springer}
}

@inproceedings{raj2019differentially,
  title={A differentially private kernel two-sample test},
  author={Raj, Anant and Law, Ho Chung Leon and Sejdinovic, Dino and Park, Mijung},
  booktitle={Joint European Conference on Machine Learning and Knowledge Discovery in Databases},
  pages={697--724},
  year={2019},
  organization={Springer}
}

@inproceedings{NEURIPS2020_72b32a1f,
 author = {Canonne, Cl\'{e}ment L and Kamath, Gautam and McMillan, Audra and Ullman, Jonathan and Zakynthinou, Lydia},
 booktitle = {Advances in Neural Information Processing Systems},
 editor = {H. Larochelle and M. Ranzato and R. Hadsell and M.F. Balcan and H. Lin},
 pages = {10099--10111},
 publisher = {Curran Associates, Inc.},
 title = {Private Identity Testing for High-Dimensional Distributions},
 url = {https://proceedings.neurips.cc/paper_files/paper/2020/file/72b32a1f754ba1c09b3695e0cb6cde7f-Paper.pdf},
 volume = {33},
 year = {2020}
}

@article{uhler2013privacy,
  author  = {Uhler, Caroline and Slavkovi{\'c}, Aleksandra and Fienberg, Stephen E.},
  title   = {Privacy-Preserving Data Sharing for Genome-Wide Association Studies},
  journal = {Journal of Privacy and Confidentiality},
  volume  = {5},
  number  = {1},
  pages   = {137--166},
  year    = {2013},
  doi     = {10.29012/jpc.v5i1.629},
  url     = {https://doi.org/10.29012/jpc.v5i1.629}
}

@INPROCEEDINGS{5360513,
  author={Vu, Duy and Slavkovic, Aleksandra},
  booktitle={2009 IEEE International Conference on Data Mining Workshops}, 
  title={Differential Privacy for Clinical Trial Data: Preliminary Evaluations}, 
  year={2009},
  volume={},
  number={},
  pages={138-143},
  doi={10.1109/ICDMW.2009.52}}

@inproceedings{NEURIPS2018_7de32147,
 author = {Acharya, Jayadev and Sun, Ziteng and Zhang, Huanyu},
 booktitle = {Advances in Neural Information Processing Systems},
 editor = {S. Bengio and H. Wallach and H. Larochelle and K. Grauman and N. Cesa-Bianchi and R. Garnett},
 pages = {},
 publisher = {Curran Associates, Inc.},
 title = {Differentially Private Testing of Identity and Closeness of Discrete Distributions},
 volume = {31},
 year = {2018}
}

@article{Kim25032026,
author = {Ilmun Kim and Antonin Schrab},
title = {Differentially Private Permutation Tests},
journal = {Journal of the American Statistical Association},
volume = {0},
number = {0},
pages = {1--13},
year = {2026},
publisher = {Taylor \& Francis},
doi = {10.1080/01621459.2025.2610033},


URL = { 
    
        https://doi.org/10.1080/01621459.2025.2610033
    
    

},
eprint = { 
    
        https://doi.org/10.1080/01621459.2025.2610033
    
    

}

}

@inproceedings{10.1145/3313276.3316336,
author = {Canonne, Cl{\'e}ment L. and Kamath, Gautam and McMillan, Audra and Smith, Adam and Ullman, Jonathan},
title = {The structure of optimal private tests for simple hypotheses},
year = {2019},
isbn = {9781450367059},
publisher = {Association for Computing Machinery},
address = {New York, NY, USA},
url = {https://doi.org/10.1145/3313276.3316336},
doi = {10.1145/3313276.3316336},
booktitle = {Proceedings of the 51st Annual ACM SIGACT Symposium on Theory of Computing},
pages = {310–321},
numpages = {12},
location = {Phoenix, AZ, USA},
series = {STOC 2019}
}

@InProceedings{pmlr-v80-aliakbarpour18a,
  title = 	 {Differentially Private Identity and Equivalence Testing of Discrete Distributions},
  author =       {Aliakbarpour, Maryam and Diakonikolas, Ilias and Rubinfeld, Ronitt},
  booktitle = 	 {Proceedings of the 35th International Conference on Machine Learning},
  pages = 	 {169--178},
  year = 	 {2018},
  editor = 	 {Dy, Jennifer and Krause, Andreas},
  volume = 	 {80},
  series = 	 {Proceedings of Machine Learning Research},
  month = 	 {10--15 Jul},
  publisher =    {PMLR},
  url = 	 {https://proceedings.mlr.press/v80/aliakbarpour18a.html}
}

@InProceedings{pmlr-v70-cai17a,
  title = 	 {{P}riv'{IT}: Private and Sample Efficient Identity Testing},
  author =       {Bryan Cai and Constantinos Daskalakis and Gautam Kamath},
  booktitle = 	 {Proceedings of the 34th International Conference on Machine Learning},
  pages = 	 {635--644},
  year = 	 {2017},
  editor = 	 {Precup, Doina and Teh, Yee Whye},
  volume = 	 {70},
  series = 	 {Proceedings of Machine Learning Research},
  month = 	 {06--11 Aug},
  publisher =    {PMLR},
  url = 	 {https://proceedings.mlr.press/v70/cai17a.html}
}

@article{ingster1993asymptotically,
  title={Asymptotically minimax hypothesis testing for nonparametric alternatives. I, II, III},
  author={Ingster, Yuri I},
  journal={Math. Methods Statist},
  volume={2},
  number={2},
  pages={85--114},
  year={1993}
}

@book{wainwright2019high,
  title={High-dimensional statistics: A non-asymptotic viewpoint},
  author={Wainwright, Martin J},
  volume={48},
  year={2019},
  publisher={Cambridge University Press}
}

@article{deng2017wasserstein,
  title={Wasserstein identity testing},
  author={Deng, Shichuan and Li, Wenzheng and Wu, Xuan},
  journal={arXiv preprint arXiv:1710.10457},
  year={2017}
}

@article{Ferreira2024Distribution,
  title={Distribution testing with a confused collector},
  author={Pinto Jr, Renato Ferreira and Harms, Nathaniel},
  journal={arXiv preprint arXiv:2311.14247},
  year={2023}
}

@inproceedings{mcsherry2009privacy,
author = {McSherry, Frank D.},
title = {Privacy integrated queries: an extensible platform for privacy-preserving data analysis},
year = {2009},
isbn = {9781605585512},
publisher = {Association for Computing Machinery},
address = {New York, NY, USA},
url = {https://doi.org/10.1145/1559845.1559850},
doi = {10.1145/1559845.1559850},
booktitle = {Proceedings of the 2009 ACM SIGMOD International Conference on Management of Data},
pages = {19–30},
numpages = {12},
location = {Providence, Rhode Island, USA},
series = {SIGMOD '09}
}

@article{AwanWang2025,
author = {Jordan Awan and Yue Wang},
title = {{Differentially private Kolmogorov-Smirnov-type tests}},
volume = {19},
journal = {Electronic Journal of Statistics},
number = {1},
publisher = {Institute of Mathematical Statistics and Bernoulli Society},
pages = {718 -- 744},
year = {2025},
doi = {10.1214/25-EJS2349},
URL = {https://doi.org/10.1214/25-EJS2349}
}

@INPROCEEDINGS{4389483,
  author={McSherry, Frank and Talwar, Kunal},
  booktitle={48th Annual IEEE Symposium on Foundations of Computer Science (FOCS'07)}, 
  title={Mechanism Design via Differential Privacy}, 
  year={2007},
  volume={},
  number={},
  pages={94-103},
  doi={10.1109/FOCS.2007.66}}

\appendix
\section{Proofs of the main results}
\subsection{Monotonicity of $\Psi_{\mathrm{upper}}^{\mathrm{gof}}(R)$}\label{appendix:monotonicity}
\begin{proof}[Proof of the monotonicity of $\Psi_{\mathrm{upper}}^{\mathrm{gof}}(R)$]
Since $J_{\mathrm{upper}}(R)$ is nonincreasing, it suffices to show that $\Lambda$ is nondecreasing in $J$. For brevity, define
\begin{align*}
    a_j
    \coloneqq \left(
    \mathcal{N}\left(
    \mathcal{X},
    \mathrm{diam}(\mathcal{X})\gamma^j/3
    \right)^q
    \gamma^{jp}
    \right)^{1/r},
    \qquad
    S_J\coloneqq \sum_{j=1}^J a_j.
\end{align*}
Then
\begin{align*}
    \Sigma(p,q,r,J)=S_J^r,
    \qquad
    H(p,q,r,J)
    =-\sum_{j=1}^J\frac{a_j}{S_J}
    \log\left(\frac{a_j}{S_J}\right).
\end{align*}
Let
\begin{align*}
    u\coloneqq \frac{S_J}{S_{J+1}}\in(0,1).
\end{align*}
The first $J$ normalized weights at level $J+1$ are obtained by
multiplying the normalized weights at level $J$ by $u$, while the remaining weight is $1-u$. Therefore,
\begin{align*}
    H(p,q,r,J+1)
    &= \,
    uH(p,q,r,J)
    -u\log u-(1-u)\log(1-u)\\
    &\geq \,uH(p,q,r,J).
\end{align*}
Consequently,
\begin{align*}
    \Lambda(p,q,r,J+1)
    &= \,
    S_{J+1}^r\bigl(1+H(p,q,r,J+1)\bigr)\\
    &\geq \,
    S_{J+1}^r\bigl(1+uH(p,q,r,J)\bigr)\\
    &= \,
    S_{J+1}^r
    +S_JS_{J+1}^{r-1}H(p,q,r,J)\\
    &\geq \,
    S_J^r+S_J^rH(p,q,r,J)\\
    &= \,
    \Lambda(p,q,r,J),
\end{align*}
where the second inequality follows from $S_{J+1}\geq S_J$ and
$r\geq1$. Hence, $\Lambda(p,q,r,J)$ is nondecreasing in $J$.
\end{proof}
\subsection{Proof of Lemma~\ref{lem:dyadic}}\label{appendix:dyadic}
For $j\geq 1$, there exists a set $C^j$ of $\mathcal{N}(\mathcal{X},\mathrm{diam}(\mathcal{X})\gamma^j/3)$ balls with diameter $\mathrm{diam}(\mathcal{X})\gamma^j/3$ that covers $\mathcal{X}$. Denote the $l$th element of $C^j$ as $C^j_l$. We first construct $\mathcal{R}^J$ as follows. For $1<l\leq \mathcal{N}(\mathcal{X},\mathrm{diam}(\mathcal{X})\gamma^J/3)$, define
\begin{align*}
    R^J_l\coloneqq C^J_l\backslash\left(\bigcup_{s<l}C_s^J\right),
\end{align*}
and set $\mathcal{R}^J\coloneqq \{R^J_1,\ldots,R^J_{\mathcal{N}(\mathcal{X},\mathrm{diam}(\mathcal{X})\gamma^J/3)}\}$. We remove empty cells from $\mathcal{R}^J$. Then $\mathcal{R}^J$ is a partition of $\mathcal{X}$, and the diameter of $R^J_l$ is bounded by $\mathrm{diam}(\mathcal{X})\gamma^J/3\leq \mathrm{diam}(\mathcal{X})\gamma^J$. Next, suppose $\mathcal{R}^j$ is a partition of $\mathcal{X}$ such that if $R\in \mathcal{R}^j$ then $\mathrm{diam}(R)\leq \mathrm{diam}(\mathcal{X})\gamma^j$. We construct a partition $\mathcal{R}^{j-1}$ from $\mathcal{R}^j$ such that
\begin{itemize}
    \item if $R\in \mathcal{R}^{j-1}$ then $\mathrm{diam}(R)\leq \mathrm{diam}(\mathcal{X})\gamma^{j-1}$
    \item if $R^j\in \mathcal{R}^j$ and $R^{j-1}\in\mathcal{R}^{j-1}$ then $R^j\subset R^{j-1}$ or $R^j\cap R^{j-1}=\emptyset$.
\end{itemize}
This leads to the construction of $\mathcal{R}^{J-1},\ldots,\mathcal{R}^1$ inductively. Let
\begin{align*}
    R^{j-1}_1\coloneqq \bigcup_{R\in\mathcal{R}^{j},R\cap C_1^{j-1}\not=\emptyset}R.
\end{align*}
For $1<l\leq \mathcal{N}(\mathcal{X},\mathrm{diam}(\mathcal{X})\gamma^{j-1}/3)$, let
\begin{align*}
    R^{j-1}_l\coloneqq \left(\bigcup_{R\in\mathcal{R}^{j},R\cap C_l^{j-1}\not=\emptyset}R\right)\Big\backslash\left(\bigcup_{s<l}R_s^{j-1}\right).
\end{align*}
Now set $\mathcal{R}^{j-1}\coloneqq \{R^{j-1}_1,\ldots,R^{j-1}_{\mathcal{N}(\mathcal{X},\mathrm{diam}(\mathcal{X})\gamma^{j-1}/3)}\}$. We remove empty cells from $\mathcal{R}^{j-1}$. Then $\mathcal{R}^{j-1}$ is a partition of $\mathcal{X}$. Let $x,x^\prime\in R_l^{j-1}$. Then there exist $c,c^\prime\in C_l^{j-1}$ such that $d(x,c),d(x^\prime,c^\prime)\leq \mathrm{diam}(\mathcal{X})\gamma^{j-1}/3$. Then
\begin{align*}
    d(x,x^\prime)&\leq d(x,c)+d(c,c^\prime)+d(c^\prime,x^\prime)\\
    &\leq \mathrm{diam}(\mathcal{X})\gamma^{j}+\mathrm{diam}(\mathcal{X})\gamma^{j-1}/3+\mathrm{diam}(\mathcal{X})\gamma^{j}\\
    &\leq \mathrm{diam}(\mathcal{X})\gamma^{j-1}.
\end{align*}
Therefore, the diameter of $R_l^{j-1}$ is at most $\mathrm{diam}(\mathcal{X})\gamma^{j-1}$. Finally, $R^{j-1}_l$ is a union of some sets in $\mathcal{R}^j$. If $R\in \mathcal{R}^j$ and $R\cap R^{j-1}_l \neq \emptyset$ then $R\subset R^{j-1}_l$ since sets in $\mathcal{R}^j$ are disjoint. Therefore, we are done.
\subsection{Proof of Propositions~\ref{prop:gof_tv_minimax} and~\ref{prop:ts_tv_minimax}}\label{appendix:tv_minimax}
Denote $L=\left\lceil(0.32)^{-1}\log\frac{1}{\alpha\wedge\beta}\right\rceil$. For the first statement, if $n<2L$ then $$\max\left\{\frac{k^{1/4}}{\sqrt{n}},\frac{k^{1/2}}{n\sqrt{\epsilon}},\frac{k^{1/4}}{n^{3/4}\sqrt{\epsilon}},\frac{1}{n\epsilon}\right\}\log\frac{1}{\alpha\wedge\beta}\gtrsim 1$$ up to a universal factor so we are done. In the case where $n\geq 2L$, we construct a test achieving Type I and Type II errors at most $\alpha$ and $\beta$, respectively. \cite{NEURIPS2018_7de32147} showed that, for a given number of samples $m$, there exists an $\epsilon$-DP test $\phi_m$ with Type I and Type II errors at most $0.1$ and separation radius
\begin{align*}
\tau_m\asymp\frac{k^{1/4}}{\sqrt{m}}\vee\frac{k^{1/2}}{m\sqrt{\epsilon}}\vee\frac{k^{1/4}}{m^{3/4}\sqrt{\epsilon}}\vee\frac{1}{m\epsilon}.
\end{align*}
We amplify the error guarantee by a majority-vote procedure. Denote $L=\left\lceil(0.32)^{-1}\log\frac{1}{\alpha\wedge\beta}\right\rceil$ and partition the data into $L$ disjoint subsets $D_1,\ldots,D_L$, each of size $m\coloneqq \lfloor n/L\rfloor$, discarding any remaining observations. Define
\begin{align*}
\varphi(D)\coloneqq \mathbf{1}\left\{\frac{1}{L}\sum_{l=1}^L\phi_m(D_l)\geq\frac{1}{2}\right\}.
\end{align*}
The privacy guarantee of the test $\varphi$ can be shown by the parallel composition theorem.
\begin{proposition}[{\citealp[Theorem~4]{mcsherry2009privacy}}]\label{prop:parallel_composition}
Let $D$ be a dataset whose observations are partitioned into disjoint subsets $D_1,\ldots,D_L$ according to a rule independent of the data values. For each $l=1,\ldots,L$, let $\mathcal{M}_l$ be an $\epsilon$-differentially private mechanism applied only to $D_l$. Then the joint mechanism
\begin{align*}
\mathcal{M}(D)\coloneqq \left(\mathcal{M}_1(D_1),\ldots,\mathcal{M}_L(D_L)\right)
\end{align*}
is $\epsilon$-differentially private. In particular, if every $\mathcal{M}_l$ is $\epsilon$-differentially private, then $\mathcal{M}$ is $\epsilon$-differentially private.
\end{proposition}
Since the tests $\phi_m$ are applied to disjoint subsets of the data, $\varphi$ is $\epsilon$-DP.

We next bound its Type I and Type II errors. Under the null hypothesis, $\phi_m(D_1),\ldots,\phi_m(D_L)$ are independent Bernoulli random variables satisfying $\mathbb{P}(\phi_m(D_l)=1\mid H_0)\leq0.1$. We use Hoeffding's inequality to bound the Type I error.
\begin{proposition}[{\citealp[Proposition~2.5]{wainwright2019high}}]\label{prop:hoeffding}
Let $X_1,\ldots,X_L$ be independent random variables such that $a_l\leq X_l\leq b_l$ almost surely for each $l=1,\ldots,L$. Then, for every $t>0$,
\begin{align*}
\mathbb{P}\left(\sum_{l=1}^L(X_l-\mathbb{E}[X_l])\geq t\right)\leq\exp\left(-\frac{2t^2}{\sum_{l=1}^L(b_l-a_l)^2}\right)
\end{align*}
and
\begin{align*}
\mathbb{P}\left(\sum_{l=1}^L(X_l-\mathbb{E}[X_l])\leq-t\right)\leq\exp\left(-\frac{2t^2}{\sum_{l=1}^L(b_l-a_l)^2}\right).
\end{align*}
\end{proposition}

\begin{align*}
\mathbb{P}\left(\frac{1}{L}\sum_{l=1}^L\phi_m(D_l)\geq\frac{1}{2}\middle|H_0\right)\leq\mathbb{P}\left(\frac{1}{L}\sum_{l=1}^L\left(\phi_m(D_l)-\mathbb{E}[\phi_m(D_l)\mid H_0]\right)\geq0.4\middle|H_0\right)\leq e^{-0.32L}\leq\alpha\wedge\beta.
\end{align*}
Similarly, under the alternative hypothesis with $\lVert u-u_0\rVert_1\geq\tau_m$, we have $\mathbb{P}(\phi_m(D_l)=1\mid H_1)\geq0.9$, and Hoeffding's inequality yields
\begin{align*}
\mathbb{P}\left(\frac{1}{L}\sum_{l=1}^L\phi_m(D_l)<\frac{1}{2}\middle|H_1\right)\leq e^{-0.32L}\leq\alpha\wedge\beta.
\end{align*}
Thus, $\varphi$ has Type I and Type II errors at most $\alpha$ and $\beta$, respectively. Since $n\geq2L$, we have $m=\lfloor n/L\rfloor\geq n/(2L)$, and hence
\begin{align*}
\tau_m&\lesssim\frac{k^{1/4}L^{1/2}}{\sqrt{n}}\vee\frac{k^{1/2}L}{n\sqrt{\epsilon}}\vee\frac{k^{1/4}L^{3/4}}{n^{3/4}\sqrt{\epsilon}}\vee\frac{L}{n\epsilon}\\
&\lesssim\frac{k^{1/4}L}{\sqrt{n}}\vee\frac{k^{1/2}L}{n\sqrt{\epsilon}}\vee\frac{k^{1/4}L^{3/4}}{n^{3/4}\sqrt{\epsilon}}\vee\frac{L}{n\epsilon}.
\end{align*}
Since $L\lesssim\log\frac{1}{\alpha\wedge\beta}$ and $L\geq1$, it follows that
\begin{align*}
\tau_m\lesssim\max\left\{\frac{k^{1/4}}{\sqrt{n}},\frac{k^{1/2}}{n\sqrt{\epsilon}},\frac{k^{1/4}}{n^{3/4}\sqrt{\epsilon}},\frac{1}{n\epsilon}\right\}\log\frac{1}{\alpha\wedge\beta}.
\end{align*}

For the second statement, suppose that there exists an $\epsilon$-DP test $\varphi_n$ with Type I and Type II errors at most $\alpha$ and $\beta$, respectively, at separation $\tau$. Let
\begin{align*}
L_0\coloneqq \left\lceil2(1/2-(\alpha\vee\beta))^{-2}\log 10\right\rceil.
\end{align*}
Given $m\coloneqq nL_0$ samples, partition them into $L_0$ disjoint subsets of size $n$, apply $\varphi_n$ independently to each subset, and take a majority vote. By parallel composition, the resulting test is $\epsilon$-DP. Moreover, Hoeffding's inequality implies that both its Type I and Type II errors are at most $0.1$.

\cite{NEURIPS2018_7de32147} showed that, when the null distribution is uniform on $[k]$, the separation radius of any $\epsilon$-DP tester based on $m$ samples with Type I and Type II errors at most $0.1$ satisfies
\begin{align*}
\tau\gtrsim\left(\frac{k^{1/4}}{\sqrt{m}}\vee\frac{k^{1/2}}{m\sqrt{\epsilon}}\vee\frac{k^{1/4}}{m^{3/4}\sqrt{\epsilon}}\vee\frac{1}{m\epsilon}\right)\wedge1.
\end{align*}
Since $m=nL_0$ and $L_0$ depends only on $\alpha$ and $\beta$, this yields
\begin{align*}
\tau\geq C_2\left(\max\left\{\frac{k^{1/4}}{\sqrt{n}},\frac{k^{1/2}}{n\sqrt{\epsilon}},\frac{k^{1/4}}{n^{3/4}\sqrt{\epsilon}},\frac{1}{n\epsilon}\right\}\wedge1\right)
\end{align*}
for some constant $C_2>0$ depending only on $\alpha$ and $\beta$. Decreasing $C_2$ if necessary, we may take $C_2\in(0,1/2)$.

For Proposition~\ref{prop:ts_tv_minimax}, the same argument can be applied by invoking the minimax separation result of $\ell_1$ two-sample testing \citep[Theorem 7]{zhang2021statistical}.
\subsection{Proof of matching up to a logarithmic factor}\label{appendix:logfactor}
We show $\log^{3/2}n\log\log nR_{\mathrm{lower}}^{\mathrm{gof}}\gtrsim R_{\mathrm{upper}}^{\mathrm{gof}}$. Denote
\begin{align*}
    A(J)&= \frac{\Lambda(p,1/4,1,J)}{\sqrt{n}}\\
    &+\frac{\Lambda(p,1/4,3/2,J)}{n^{3/4}\sqrt{\epsilon}}+\frac{\Lambda(p,1/2,3/2,J)}{n\sqrt{\epsilon}}\\
    &+\frac{\Lambda(p,0,2,J)}{n\epsilon}\\
    B(J)&= \frac{M(p,1/4,J)}{\sqrt{n}}\vee\frac{M(p,1/2,J)}{n\sqrt{\epsilon}}\vee \frac{M(p,1/4,J)}{n^{3/4}\sqrt{\epsilon}}\vee\frac{M(p,0,J)}{n\epsilon}.
\end{align*}
Recall that
\begin{align*}
    J_{\mathrm{upper}}(R)&= \left\lceil\frac{1}{p}\frac{\log\frac{2}{R}}{\log(1/\gamma)}\right\rceil\\
    J_{\mathrm{lower}}(R)&= \left\lfloor\frac{1}{p}\frac{\log(2/R)}{\log(1/\gamma)}\right\rfloor.
\end{align*}
First, assume $\gamma^{-p}R_{\mathrm{lower}}^{\mathrm{gof}}\leq \gamma^{2p}R_{\mathrm{upper}}^{\mathrm{gof}}$. By definitions of $R_{\mathrm{lower}}^{\mathrm{gof}}$ and $R_{\mathrm{upper}}^{\mathrm{gof}}$,
\begin{align*}
    R\leq A\left(J_{\mathrm{upper}}(R)\right),
\end{align*}
if $R< R_{\mathrm{upper}}^{\mathrm{gof}}$ and
\begin{align*}
    R\geq B\left(J_{\mathrm{lower}}(R)\right)
\end{align*}
if $R> R_{\mathrm{lower}}^{\mathrm{gof}}$. Also, $J^{3/2}\log (eJ) B(J)\gtrsim A(J)$ since $J^{r}\log (eJ) M(p,q,J)\gtrsim \Lambda(p,q,r,J)$ and $\Lambda(p,0,2,J)\lesssim M(p,0,J)$. Note that  $J_{\mathrm{lower}}(\gamma^{p}R)\geq J_{\mathrm{upper}}(R)$ since $\lceil x\rceil\leq \lfloor x+1\rfloor $. Then
\begin{align*}
   &J_{\mathrm{lower}}^{3/2}(\gamma^{-p}R_{\mathrm{lower}}^{\mathrm{gof}})\log( eJ_{\mathrm{lower}}(\gamma^{-p}R_{\mathrm{lower}}^{\mathrm{gof}})) \gamma^{-p}R_{\mathrm{lower}}^{\mathrm{gof}}\\
   &\geq  J_{\mathrm{lower}}^{3/2}(\gamma^{-p}R_{\mathrm{lower}}^{\mathrm{gof}})\log( eJ_{\mathrm{lower}}(\gamma^{-p}R_{\mathrm{lower}}^{\mathrm{gof}}))B\left(J_{\mathrm{lower}}(\gamma^{-p}R_{\mathrm{lower}}^{\mathrm{gof}})\right)\\
   &\gtrsim A(J_{\mathrm{lower}}(\gamma^{-p}R_{\mathrm{lower}}^{\mathrm{gof}}))\\
   &\geq A(J_{\mathrm{lower}}(\gamma^{2p}R_{\mathrm{upper}}^{\mathrm{gof}}))\\
   &\geq A(J_{\mathrm{upper}}(\gamma^p R_{\mathrm{upper}}^{\mathrm{gof}}))\\
   &\geq \gamma^pR_{\mathrm{upper}}^{\mathrm{gof}}
\end{align*}
since $A(J)$ increases with respect to $J$. Also, since $C_{\mathrm{lower}}^{\mathrm{gof}}\mathrm{diam}^p(\mathcal{X})R_{\mathrm{lower}}^{\mathrm{gof}}\gtrsim \frac{1}{\sqrt{n}}$ by Theorem~\ref{thm:goodness_of_fit_minimax_lower}, we have
\begin{align*}
    J_{\mathrm{lower}}(\gamma^{-p}R_{\mathrm{lower}}^{\mathrm{gof}})\lesssim \log n.
\end{align*}
Therefore,
\begin{align*}
    R_{\mathrm{upper}}^{\mathrm{gof}}\lesssim \log^{3/2}n\log\log nR_{\mathrm{lower}}^{\mathrm{gof}}.
\end{align*}
Next, if $\gamma^{-p}R_{\mathrm{lower}}^{\mathrm{gof}}\geq \gamma^{2p}R_{\mathrm{upper}}^{\mathrm{gof}}$ then $\log^{3/2}n\log\log nR_{\mathrm{lower}}^{\mathrm{gof}}\gtrsim R_{\mathrm{upper}}^{\mathrm{gof}}$ already holds.\qed
\subsection{Proof of Theorem~\ref{thm:two_sample_minimax_upper}}\label{appendix:two_sample_minimax_upper}
We divide the proof into two cases. First, suppose that the set in the definition of $R_{\mathrm{upper}}^{\mathrm{ts}}$ is nonempty. Given $p,\alpha,\beta$, we choose the remaining parameters $\tau_1,\ldots,\tau_J$ and $\nu_1,\ldots,\nu_J$ of Algorithm~\ref{alg:ts_tester_Wp} as follows. Set $J\coloneqq J_{\mathrm{upper}}(R_{\mathrm{upper}}^{\mathrm{ts}})$. Then $J\geq 1$ since $\frac{1}{p}\frac{\log(2/R_{\mathrm{upper}}^{\mathrm{ts}})}{\log(1/\gamma)}\geq1$. For each $1\leq j\leq J$, set
\begin{align*}
    \nu_j&\coloneqq \frac{N_j^{1/4}\gamma^{jp}}{5\sum_{j=1}^JN_j^{1/4}\gamma^{jp}}\vee\frac{N_j^{1/2}\gamma^{jp}}{5\sum_{j=1}^JN_j^{1/2}\gamma^{jp}}\vee\frac{(N_j^{1/2}\gamma^{jp})^{\frac{2}{3}}}{5\sum_{j=1}^J(N_j^{1/2}\gamma^{jp})^{\frac{2}{3}}}\vee\frac{(N_j^{1/4}\gamma^{jp})^{\frac{2}{3}}}{5\sum_{j=1}^J(N_j^{1/4}\gamma^{jp})^{\frac{2}{3}}}\vee\frac{\gamma^{\frac{jp}{2}}}{5\sum_{j=1}^J\gamma^{\frac{jp}{2}}}.
\end{align*}
Also, let $\tau_j$ be the separation radius of the tester in Step 2(c) of Algorithm~\ref{alg:ts_tester_Wp}, that is the minimax separation radius of $\nu_j\epsilon$-DP $\ell_1$ two-sample testing on a domain of size $N_j$ with Type I and Type II errors at most $\nu_j\alpha$ and $\nu_j\beta$ respectively.

Under these choices of parameters, we show that Algorithm~\ref{alg:ts_tester_Wp} is an $\epsilon$-DP test with Type I and Type II errors at most $\alpha$ and $\beta$, respectively, for the following Wasserstein two-sample testing problem.
\begin{align*}
    H_0:P=Q\textrm{ vs }H_1:W_p^p(P,Q)\geq C_{\mathrm{upper}}^{\mathrm{ts}}\mathrm{diam}^p(\mathcal{X})R_{\mathrm{upper}}^{\mathrm{ts}},
\end{align*}
where $C_{\mathrm{upper}}^{\mathrm{ts}}>0$ is a sufficiently large constant. The privacy guarantee and the Type I error bound follow directly. Note that $\sum_{j=1}^J\nu_j\leq 1$ since each of the five sequences appearing in this maximum sums to $1/5$. The algorithm is a composition of $\nu_j\epsilon$-DP algorithms for $j=1,\ldots,J$. By Proposition~\ref{prop:comp}, Algorithm~\ref{alg:ts_tester_Wp} is $\sum_{j=1}^J\nu_j\epsilon$-DP. Since $\sum_{j=1}^J\nu_j\epsilon\leq\epsilon$, the algorithm is $\epsilon$-DP. Similarly, since the algorithm rejects whenever at least one of the $J$ finite-domain testers rejects, a union bound gives that the Type I error of Algorithm~\ref{alg:ts_tester_Wp} is at most $\sum_{j=1}^J\nu_j\alpha\leq \alpha$.

We next show that the Type II error of Algorithm~\ref{alg:ts_tester_Wp} is at most $\beta$. By the multiresolution upper bound \eqref{eq:Wp_decomp}, it suffices to establish
\begin{align*}
\mathrm{diam}^p(\mathcal{X})\left(\gamma^{Jp}+\sum_{j=1}^J\gamma^{(j-1)p}\tau_j\right)\leq C_{\mathrm{upper}}^{\mathrm{ts}}\mathrm{diam}^p(\mathcal{X})R_{\mathrm{upper}}^{\mathrm{ts}},
\end{align*}
or, equivalently,
\begin{align}\label{eq:proof_ts_minimax_upper_goal}
    \gamma^{Jp}+\sum_{j=1}^J\gamma^{(j-1)p}\tau_j\leq C_{\mathrm{upper}}^{\mathrm{ts}}R_{\mathrm{upper}}^{\mathrm{ts}}.
\end{align}
Suppose \eqref{eq:proof_ts_minimax_upper_goal} holds. Under the alternative hypothesis, \eqref{eq:Wp_decomp} then implies
\begin{align*}
\mathrm{diam}^p(\mathcal{X})\left(\gamma^{Jp}+\sum_{j=1}^J\gamma^{(j-1)p}\lVert P^j-Q^j\rVert_1\right)&\geq  W_p^p(P,Q)\\
&\geq  C_{\mathrm{upper}}^{\mathrm{ts}}\mathrm{diam}^p(\mathcal{X})R_{\mathrm{upper}}^{\mathrm{ts}}\\
&\geq  \mathrm{diam}^p(\mathcal{X})\left(\gamma^{Jp}+\sum_{j=1}^J\gamma^{(j-1)p}\tau_j\right).
\end{align*}
Thus, there exists $1\leq j_0\leq J$ such that $\lVert P^{j_0}-Q^{j_0}\rVert_1\geq \tau_{j_0}$. Since Algorithm~\ref{alg:ts_tester_Wp} accepts the null hypothesis if every private $\ell_1$ tester accepts, its Type II error is at most the Type II error of the private $\ell_1$ tester $\phi_{j_0}$. By the choice of $\tau_{j_0}$, this error is at most $\nu_{j_0}\beta\leq\beta$.

It remains to verify \eqref{eq:proof_ts_minimax_upper_goal}. By the definition of $J$, we have $\gamma^{Jp}\leq\frac{1}{2}R_{\mathrm{upper}}^{\mathrm{ts}}$. For $1\leq j\leq J$, since $|\mathcal{R}^j|\leq \mathcal{N}(\mathcal{X},\mathrm{diam}(\mathcal{X})\gamma^j/3)$, Proposition~\ref{prop:ts_tv_minimax} implies that $\tau_j$ with $\epsilon_j = \nu_j \epsilon$ satisfies
\begin{align*}
    \tau_j&\lesssim \Bigg(\frac{N_j^{1/4}}{\sqrt{n}}+\frac{N_j^{1/2}}{n^{3/4}}+\frac{N_j^{1/4}}{n^{3/4}\sqrt{\epsilon_j}}+\frac{N_j^{1/2}}{n\sqrt{\epsilon_j}}+\frac{1}{n\epsilon_j}\Bigg) \log\frac{1}{(\nu_j\alpha)\wedge(\nu_j\beta)}\notag\\
    &\lesssim \Bigg(\frac{N_j^{1/4}}{\sqrt{n}}+\frac{N_j^{1/2}}{n^{3/4}}+\frac{N_j^{1/4}\nu_j^{-1/2}}{n^{3/4}\sqrt{\epsilon}}+\frac{N_j^{1/2}\nu_j^{-1/2}}{n\sqrt{\epsilon}}+\frac{\nu_j^{-1}}{n\epsilon}\Bigg)\log\frac{1}{\nu_j},
\end{align*}
up to a constant factor depending only on $\alpha$ and $\beta$. By the definition of $\nu_j$, we have
\begin{align*}
\sum_{j=1}^JN_j^{1/4}\gamma^{jp}\log\frac{1}{\nu_j}&\lesssim \Lambda(p,1/4,1,J),\\
\sum_{j=1}^JN_j^{1/2}\gamma^{jp}\log\frac{1}{\nu_j}&\lesssim \Lambda(p,1/2,1,J),\\
\sum_{j=1}^JN_j^{1/4}\gamma^{jp}\nu_j^{-1/2}\log\frac{1}{\nu_j}&\lesssim \Lambda(p,1/4,3/2,J),\\
\sum_{j=1}^JN_j^{1/2}\gamma^{jp}\nu_j^{-1/2}\log\frac{1}{\nu_j}&\lesssim \Lambda(p,1/2,3/2,J),\\
\sum_{j=1}^J\gamma^{jp}\nu_j^{-1}\log\frac{1}{\nu_j}&\lesssim \Lambda(p,0,2,J).
\end{align*}
Therefore,
\begin{align*}
    \sum_{j=1}^J\gamma^{(j-1)p}\tau_j&\lesssim \frac{\Lambda(p,1/4,1,J)}{\sqrt{n}}+\frac{\Lambda(p,1/2,1,J)}{n^{3/4}}+\frac{\Lambda(p,1/4,3/2,J)}{n^{3/4}\sqrt{\epsilon}}+\frac{\Lambda(p,1/2,3/2,J)}{n\sqrt{\epsilon}}+\frac{\Lambda(p,0,2,J)}{n\epsilon}\\
    &\leq R_{\mathrm{upper}}^{\mathrm{ts}}
\end{align*}
up to a constant factor depending only on $\alpha$ and $\beta$, where the last inequality follows because $R_{\mathrm{upper}}^{\mathrm{ts}}$ belongs to the set in its definition. Combining this bound with $\gamma^{Jp}\leq R_{\mathrm{upper}}^{\mathrm{ts}}/2$ shows that \eqref{eq:proof_ts_minimax_upper_goal} holds for a sufficiently large $C_{\mathrm{upper}}^{\mathrm{ts}}>\gamma^{-p}/2$.

It remains to consider the case in which the set defining $R_{\mathrm{upper}}^{\mathrm{ts}}$ is empty. By definition, $R_{\mathrm{upper}}^{\mathrm{ts}}=2\gamma^p$. Our choice $C_{\mathrm{upper}}^{\mathrm{ts}}>\gamma^{-p}/2$ ensures that $C_{\mathrm{upper}}^{\mathrm{ts}}\mathrm{diam}^p(\mathcal{X})R_{\mathrm{upper}}^{\mathrm{ts}}>\mathrm{diam}^p(\mathcal{X})$. However, no distributions $P,Q\in\mathcal{P}(\mathcal{X})$ satisfy $W_p^p(P,Q)> \mathrm{diam}^p(\mathcal{X})$ as we have shown in the proof of Theorem~\ref{thm:goodness_of_fit_minimax_upper}. Thus, the alternative hypothesis is empty. The test that always accepts the null hypothesis therefore has zero Type I error, while the Type II error bound holds vacuously.\qed
\subsection{Proof of Theorem~\ref{thm:two_sample_minimax_lower}}\label{appendix:two_sample_minimax_lower}
For a positive integer $j$, denote $h_j\coloneqq \mathrm{diam}(\mathcal{X})\gamma^j/6$. We choose an $h_j$-separated set $\{x_i^{(j)}\}_{i=1}^{N_j}$ of size $N_j$, which exists by Lemma~\ref{lem:sep_cover}. That is,
\begin{align*}
    d(x_{i}^{(j)},x_{i^\prime}^{(j)})\geq h_j,
\end{align*}
holds for every $i\not=i^\prime$. Then, by Lemma~\ref{lem:grid}, the map $T_{h_j}$ embeds discrete distributions on $[N_j]$ to $\mathcal{P}(\mathcal{X})$ and satisfies
\begin{align*}
    W_p^p(T_{h_j}u,T_{h_j}v)\geq\frac{1}{2}h_j^p\lVert u-v\rVert_1.
\end{align*}
Now, suppose there exists an $\epsilon$-DP two-sample test $\phi$ with Type I and Type II errors at most $\alpha$ and $\beta$, respectively, for the following Wasserstein two-sample testing problem.
\begin{align*}
    H_0:P=Q\textrm{ vs }H_1:W_p^p(P,Q)\geq \mathrm{diam}^p(\mathcal{X})R.
\end{align*}
Next, let $A_1,\ldots,A_n\overset{\textrm{i.i.d.}}{\sim}u$ and $B_1,\ldots,B_n\overset{\textrm{i.i.d.}}{\sim}v$ for some $u,v\in\mathcal{P}([N_j])$ and consider the following $\epsilon$-DP two-sample testing problem on $[N_j]$.
\begin{align}\label{eq:reduced_l1_ts_test}
    H_0:u=v\textrm{ vs }H_1:\lVert u-v\rVert_1\geq 2h_j^{-p}\mathrm{diam}^p(\mathcal{X})R.
\end{align}
Consider a test $\phi(x_{A_1}^{(j)},\ldots,x_{A_n}^{(j)},x_{B_1}^{(j)},\ldots,x_{B_n}^{(j)})$ induced from $\phi$ for \eqref{eq:reduced_l1_ts_test}. The test is $\epsilon$-DP since the mappings $A_i\mapsto x_{A_i}^{(j)}$ and $B_i\mapsto x_{B_i}^{(j)}$ are deterministic and preserve adjacency. Moreover, Lemma~\ref{lem:grid} implies that its Type I and Type II errors are at most $\alpha$ and $\beta$, respectively. Using the fact that $2h_j^{-p}\mathrm{diam}^p(\mathcal{X})R$ cannot be smaller than the minimax separation radius for testing on $[N_j]$, we can show a lower bound for $R$.

We now derive a necessary lower bound on $R$. Set $J=J_{\mathrm{lower}}(R_{\mathrm{lower}}^{\mathrm{ts}})\geq1$. Then $R_{\mathrm{lower}}^{\mathrm{ts}}\leq 2\gamma^{Jp}$. Since the induced test solves \eqref{eq:reduced_l1_ts_test} at separation $2h_j^{-p}\mathrm{diam}^p(\mathcal X)R$, this separation must be at least the minimax separation radius of the corresponding finite-domain problem.
\begin{align*}
   2h_j^{-p}\mathrm{diam}^p(\mathcal{X})R&\geq C_4\left(\left(\frac{N_j^{1/4}}{\sqrt{n}}\vee\frac{N_j^{1/2}}{n^{3/4}}\vee\frac{N_j^{1/4}}{n^{3/4}\sqrt{\epsilon}}\vee \frac{N_j^{1/2}}{n\sqrt{\epsilon}}\vee\frac{1}{n\epsilon} \right)\wedge1\right),
\end{align*}
or, equivalently,
\begin{align*}
    R&\geq  2^{-1}6^{-p}\gamma^{jp}C_4\left(\left(\frac{N_j^{1/4}}{\sqrt{n}}\vee\frac{N_j^{1/2}}{n^{3/4}}\vee\frac{N_j^{1/4}}{n^{3/4}\sqrt{\epsilon}}\vee \frac{N_j^{1/2}}{n\sqrt{\epsilon}}\vee\frac{1}{n\epsilon} \right)\wedge 1\right)\\
    &\geq 2^{-1}6^{-p}\gamma^{jp}C_4\left(\frac{N_j^{1/4}}{\sqrt{n}}\vee\frac{N_j^{1/2}}{n^{3/4}}\vee\frac{N_j^{1/4}}{n^{3/4}\sqrt{\epsilon}}\vee \frac{N_j^{1/2}}{n\sqrt{\epsilon}}\vee\frac{1}{n\epsilon} \right)\wedge 2^{-1}6^{-p}C_4\gamma^{Jp}\\
    &\geq  2^{-1}6^{-p}\gamma^{jp}C_4\left(\frac{N_j^{1/4}}{\sqrt{n}}\vee\frac{N_j^{1/2}}{n^{3/4}}\vee\frac{N_j^{1/4}}{n^{3/4}\sqrt{\epsilon}}\vee \frac{N_j^{1/2}}{n\sqrt{\epsilon}}\vee\frac{1}{n\epsilon} \right)\wedge 2^{-2}6^{-p}C_4R_{\mathrm{lower}}^{\mathrm{ts}}.
\end{align*}
Taking the maximum over $j=1,\ldots,J$, we get{\small
\begin{align*}
    R&\geq  2^{-1}6^{-p}C_4\left(\left(\frac{M(p,1/4,J)}{\sqrt{n}}\vee\frac{M(p,1/2,J)}{n^{3/4}}\vee\frac{M(p,1/2,J)}{n\sqrt{\epsilon}}\vee \frac{M(p,1/4,J)}{n^{3/4}\sqrt{\epsilon}}\vee\frac{M(p,0,J)}{n\epsilon}\right)\wedge 2^{-1}R_{\mathrm{lower}}^{\mathrm{ts}}\right)\\
    &\geq 2^{-2}6^{-p}C_4R_{\mathrm{lower}}^{\mathrm{ts}},
\end{align*}}
where the last inequality holds by the definition of $R_{\mathrm{lower}}^{\mathrm{ts}}$. Therefore, for any $R>0$, there exists a tester for $W_p^p(P,Q)\geq \mathrm{diam}^p(\mathcal{X})R$ only if $R\geq 2^{-2}6^{-p}C_4R_{\mathrm{lower}}^{\mathrm{ts}}$.

\enlargethispage{\baselineskip}
\subsection{Proof of Corollary~\ref{cor:special_case}}\label{appendix:special_case}
\paragraph{1. Upper bound for goodness-of-fit testing.} We first prove the bounds for goodness-of-fit testing. For $R\in (0,e^{-e})$, define
\begin{align*}
    \Phi_{\mathrm{upper}}^{\mathrm{gof}}(R)&\coloneqq \begin{dcases}
        \frac{R^{-d/4p+1}}{\sqrt{n}}+\frac{R^{-d/4p+1}}{n^{3/4}\sqrt{\epsilon}}+\frac{R^{-d/2p+1}}{n\sqrt{\epsilon}}+\frac{1}{n\epsilon}&d>4p,\\[6pt]
        \frac{\log(1/R)\log\log(1/R)}{\sqrt{n}}+\frac{\log^{3/2} (1/R)\log\log(1/R)}{n^{3/4}\sqrt{\epsilon}}+\frac{R^{-1}}{n\sqrt{\epsilon}}+\frac{1}{n\epsilon}&d=4p,\\[6pt]
        \frac{1}{\sqrt{n}}+\frac{1}{n^{3/4}\sqrt{\epsilon}}+\frac{R^{-d/2p+1}}{n\sqrt{\epsilon}}+\frac{1}{n\epsilon}&2p<d<4p,\\[6pt]
        \frac{1}{\sqrt{n}}+\frac{1}{n^{3/4}\sqrt{\epsilon}}+\frac{\log^{3/2} (1/R)\log\log(1/R)}{n\sqrt{\epsilon}}+\frac{1}{n\epsilon}&d=2p,\\[6pt]
        \frac{1}{\sqrt{n}}+\frac{1}{n^{3/4}\sqrt{\epsilon}}+\frac{1}{n\epsilon}&d<2p.
    \end{dcases}
\end{align*}
Then there exists a constant $C>0$ depending only on $d$ and $p$ such that
\begin{align*}
    C^{-1}\Psi_{\mathrm{upper}}^{\mathrm{gof}}(R)\leq\Phi_{\mathrm{upper}}^{\mathrm{gof}}(R).
\end{align*}
For $n\geq2$, denote $\ell_n\coloneqq \log n\log\log(en)$. For $c_{\mathrm{upper}}^{\mathrm{gof}}>0$, define
\begin{align*}
g_{\mathrm{upper}}^{\mathrm{gof}}(n)\coloneqq \begin{cases}
c_{\mathrm{upper}}^{\mathrm{gof}}\left(n^{-2p/d}\vee(n\sqrt{\epsilon})^{-\left(1\wedge\frac{2p}{d}\right)}\vee(n^{3/4}\sqrt{\epsilon})^{-\left(1\wedge\frac{4p}{d}\right)}\vee(n\epsilon)^{-1}\right), & d>4p,\\[6pt]
c_{\mathrm{upper}}^{\mathrm{gof}}\left(n^{-1/2}\ell_n\vee(n\sqrt{\epsilon})^{-1/2}\vee(n\epsilon)^{-1}\right), & d=4p,\\[6pt]
c_{\mathrm{upper}}^{\mathrm{gof}}\left(n^{-1/2}\vee(n\sqrt{\epsilon})^{-\left(1\wedge\frac{2p}{d}\right)}\vee(n^{3/4}\sqrt{\epsilon})^{-\left(1\wedge\frac{4p}{d}\right)}\vee(n\epsilon)^{-1}\right), & d<4p.
\end{cases}
\end{align*}
We show that
\begin{align*}
(C_{\mathrm{upper}}^{\mathrm{gof}}\vee e^e\vee \gamma^{-p}/2)\mathrm{diam}^p([0,1]^d)(g_{\mathrm{upper}}^{\mathrm{gof}}(n)\wedge e^{-e}\wedge2\gamma^p)\geq \rho_{n,\epsilon}^{\mathrm{gof}}([0,1]^d,p,\alpha,\beta)
\end{align*}
for sufficiently large $c_{\mathrm{upper}}^{\mathrm{gof}}$. Note that, if $g_{\mathrm{upper}}^{\mathrm{gof}}(n)\geq e^{-e}\wedge2\gamma^p$ then
\begin{align*}
(C_{\mathrm{upper}}^{\mathrm{gof}}\vee e^e\vee \gamma^{-p}/2)\mathrm{diam}^p([0,1]^d)(g_{\mathrm{upper}}^{\mathrm{gof}}(n)\wedge e^{-e}\wedge 2\gamma^p)&\geq\mathrm{diam}^p([0,1]^d)\\
&\geq\rho_{n,\epsilon}^{\mathrm{gof}}([0,1]^d,p,\alpha,\beta).
\end{align*}
Thus, we consider the nontrivial regime $g_{\mathrm{upper}}^{\mathrm{gof}}(n)\in(0,e^{-e}\wedge2\gamma^p)$.

We first state an elementary fact that will be repeatedly used.
\begin{lemma}\label{lem:power_domination}
Let $a,g,R>0$ and $r\geq1$.
\begin{enumerate}
    \item If $g\geq ca^{1/r}$ for some $c>0$, then $ag^{-r+1}\leq\frac{g}{c^r}$.
    \item If $R\leq(Ca)^{1/r}$ for some $C>0$, then $CaR^{-r+1}\geq R$.
\end{enumerate}
\end{lemma}

\begin{proof}[Proof of Lemma~\ref{lem:power_domination}]
For the first claim, $g\geq ca^{1/r}$ implies $a\leq g^r/c^r$, and hence
\begin{align*}
ag^{-r+1}\leq\frac{g}{c^r}.
\end{align*}
For the second claim, $R\leq(Ca)^{1/r}$ implies $R^r\leq Ca$, and hence
\begin{align*}
CaR^{-r+1}\geq R.
\end{align*}
\end{proof}

The logarithmic terms at the boundary dimensions can also be directly controlled by $g_{\mathrm{upper}}^{\mathrm{gof}}(n)/c$.

\begin{lemma}\label{lem:log_domination_by_g}
For every $c_{\mathrm{upper}}^{\mathrm{gof}}\geq \max\left\{1,\ell_2^{-1}\right\}$, $n\geq2$, and $\epsilon>0$ satisfying $0<g_{\mathrm{upper}}^{\mathrm{gof}}(n)<e^{-e}$, the following inequalities hold:
\begin{alignat*}{2}
\frac{\log(1/g_{\mathrm{upper}}^{\mathrm{gof}}(n))\log\log(1/g_{\mathrm{upper}}^{\mathrm{gof}}(n))}{\sqrt{n}} &\leq \frac{1}{2c_{\mathrm{upper}}^{\mathrm{gof}}}g_{\mathrm{upper}}^{\mathrm{gof}}(n)&\qquad&\text{if } d=4p,\\[6pt]
\frac{\log^{3/2}(1/g_{\mathrm{upper}}^{\mathrm{gof}}(n))\log\log(1/g_{\mathrm{upper}}^{\mathrm{gof}}(n))}{n^{3/4}\sqrt{\epsilon}} &\leq \left(\frac{15}{e}\right)^{5/2}\frac{g_{\mathrm{upper}}^{\mathrm{gof}}(n)}{c_{\mathrm{upper}}^{\mathrm{gof}}}&\qquad&\text{if } d=4p,\\[6pt]
\frac{\log^{3/2}(1/g_{\mathrm{upper}}^{\mathrm{gof}}(n))\log\log(1/g_{\mathrm{upper}}^{\mathrm{gof}}(n))}{n\sqrt{\epsilon}} &\leq\left(\frac{15}{e}\right)^{5/2}\frac{g_{\mathrm{upper}}^{\mathrm{gof}}(n)}{c_{\mathrm{upper}}^{\mathrm{gof}}}&\qquad&\text{if } d=2p.
\end{alignat*}
\end{lemma}
Consider first any non-logarithmic term in $\Phi_{\mathrm{upper}}^{\mathrm{gof}}(g_{\mathrm{upper}}^{\mathrm{gof}}(n))$. By the definition of $g_{\mathrm{upper}}^{\mathrm{gof}}(n)$, such a term is of the form $ag_{\mathrm{upper}}^{\mathrm{gof}}(n)^{-r+1}$, where $r\geq1$, $a$ is a polynomial order of $n$, and $g_{\mathrm{upper}}^{\mathrm{gof}}(n)\geq c_{\mathrm{upper}}^{\mathrm{gof}}a^{1/r}$. Lemma~\ref{lem:power_domination} therefore shows that each such term is bounded above by $g_{\mathrm{upper}}^{\mathrm{gof}}(n)/c_{\mathrm{upper}}^{\mathrm{gof}}$. Since $\Phi_{\mathrm{upper}}^{\mathrm{gof}}(g_{\mathrm{upper}}^{\mathrm{gof}}(n))$ contains at most four non-logarithmic terms, their sum is bounded above by $4g_{\mathrm{upper}}^{\mathrm{gof}}(n)/c_{\mathrm{upper}}^{\mathrm{gof}}$.

The only remaining terms are the logarithmic terms occurring when $d=4p$ or $d=2p$. By Lemma~\ref{lem:log_domination_by_g}, their total contribution in $\Phi_{\mathrm{upper}}^{\mathrm{gof}}$ is bounded above by
\begin{align*}
\left(\frac{1}{2}+\left(\frac{15}{e}\right)^{5/2}\right)\frac{g_{\mathrm{upper}}^{\mathrm{gof}}(n)}{c_{\mathrm{upper}}^{\mathrm{gof}}}.
\end{align*}
Combining the two bounds yields
\begin{align*}
C^{-1}\Psi_{\mathrm{upper}}^{\mathrm{gof}}(g_{\mathrm{upper}}^{\mathrm{gof}}(n))\leq\left(\frac{9}{2}+\left(\frac{15}{e}\right)^{5/2}\right)\frac{g_{\mathrm{upper}}^{\mathrm{gof}}(n)}{c_{\mathrm{upper}}^{\mathrm{gof}}}.
\end{align*}
Therefore, by choosing $c_{\mathrm{upper}}^{\mathrm{gof}}\geq\max\left\{1,\ell_2^{-1},C\left(\frac{9}{2}+\left(\frac{15}{e}\right)^{5/2}\right)\right\}$ we have
\begin{align*}
\Psi_{\mathrm{upper}}^{\mathrm{gof}}(g_{\mathrm{upper}}^{\mathrm{gof}}(n))\leq g_{\mathrm{upper}}^{\mathrm{gof}}(n).
\end{align*}
Therefore, $C_{\mathrm{upper}}^{\mathrm{gof}}\mathrm{diam}^p([0,1]^d)g_{\mathrm{upper}}^{\mathrm{gof}}(n)\geq \rho_{n,\epsilon}^{\mathrm{gof}}([0,1]^d,p,\alpha,\beta)$ by Theorem~\ref{thm:goodness_of_fit_minimax_upper} since $g_{\mathrm{upper}}^{\mathrm{gof}}(n)\leq 2\gamma^p$.

\paragraph{2. Lower bound for goodness-of-fit testing.} Next, we show the lower bound.

For $c_{\mathrm{lower}}^{\mathrm{gof}}>0$, define
\begin{align*}
g_{\mathrm{lower}}^{\mathrm{gof}}(n)\coloneqq c_{\mathrm{lower}}^{\mathrm{gof}}\left(n^{-\left(\frac{1}{2}\wedge\frac{2p}{d}\right)}\vee(n\sqrt{\epsilon})^{-\left(1\wedge\frac{2p}{d}\right)}\vee(n^{3/4}\sqrt{\epsilon})^{-\left(1\wedge\frac{4p}{d}\right)}\vee(n\epsilon)^{-1}\right)\wedge2\gamma^p
\end{align*}
We show that $C_{\mathrm{lower}}^{\mathrm{gof}}\mathrm{diam}^p([0,1]^d)(g_{\mathrm{lower}}^{\mathrm{gof}}(n)\wedge 2\gamma^p)\leq \rho_{n,\epsilon}^{\mathrm{gof}}([0,1]^d,p,\alpha,\beta)$.

There exists a constant $C^\prime>0$ depending only on $d$ and $p$ such that
\begin{align*}
\Psi_{\mathrm{lower}}^{\mathrm{gof}}(R)\geq C^\prime\left(\frac{R^{-0\vee(d/4p-1)}}{\sqrt{n}}\vee\frac{R^{-0\vee(d/4p-1)}}{n^{3/4}\sqrt{\epsilon}}\vee\frac{R^{-0\vee(d/2p-1)}}{n\sqrt{\epsilon}}\vee\frac{1}{n\epsilon}\right).
\end{align*}
By Lemma~\ref{lem:power_domination}, if
\begin{align*}
x\leq (C^\prime n^{-1/2})^{1\wedge\frac{4p}{d}}\vee\left(\frac{C^\prime}{n^{3/4}\sqrt{\epsilon}}\right)^{1\wedge\frac{4p}{d}}\vee\left(\frac{C^\prime}{n\sqrt{\epsilon}}\right)^{1\wedge\frac{2p}{d}}\vee\frac{C^\prime}{n\epsilon},
\end{align*}
then
\begin{align*}
C^\prime\left(\frac{x^{-0\vee(d/4p-1)}}{\sqrt{n}}\vee\frac{x^{-0\vee(d/4p-1)}}{n^{3/4}\sqrt{\epsilon}}\vee\frac{x^{-0\vee(d/2p-1)}}{n\sqrt{\epsilon}}\vee\frac{1}{n\epsilon}\right)\geq x.
\end{align*}
The upper bound on $x$ can be rewritten as
\begin{align*}
x\leq (C^\prime)^{1\wedge\frac{4p}{d}}n^{-\left(\frac{1}{2}\wedge\frac{2p}{d}\right)}\vee(C^\prime)^{1\wedge\frac{4p}{d}}(n^{3/4}\sqrt{\epsilon})^{-\left(1\wedge\frac{4p}{d}\right)}\vee(C^\prime)^{1\wedge\frac{2p}{d}}(n\sqrt{\epsilon})^{-\left(1\wedge\frac{2p}{d}\right)}\vee C^\prime(n\epsilon)^{-1}.
\end{align*}
Therefore, by choosing $c_{\mathrm{lower}}^{\mathrm{gof}}\coloneqq \min\left\{(C^\prime)^{1\wedge\frac{4p}{d}},(C^\prime)^{1\wedge\frac{2p}{d}},C^\prime\right\}$ we have
\begin{align*}
g_{\mathrm{lower}}^{\mathrm{gof}}(n)\leq (C^\prime n^{-1/2})^{1\wedge\frac{4p}{d}}\vee\left(\frac{C^\prime}{n^{3/4}\sqrt{\epsilon}}\right)^{1\wedge\frac{4p}{d}}\vee\left(\frac{C^\prime}{n\sqrt{\epsilon}}\right)^{1\wedge\frac{2p}{d}}\vee\frac{C^\prime}{n\epsilon}.
\end{align*}
Consequently,
\begin{align*}
\Psi_{\mathrm{lower}}^{\mathrm{gof}}(g_{\mathrm{lower}}^{\mathrm{gof}}(n))\geq g_{\mathrm{lower}}^{\mathrm{gof}}(n).
\end{align*}
By Theorem~\ref{thm:goodness_of_fit_minimax_lower}, we have $C_{\mathrm{lower}}^{\mathrm{gof}}\mathrm{diam}^p([0,1]^d)g_{\mathrm{lower}}^{\mathrm{gof}}(n)\leq \rho_{n,\epsilon}^{\mathrm{gof}}([0,1]^d,p,\alpha,\beta)$ since $g_{\mathrm{lower}}^{\mathrm{gof}}(n)\leq 2\gamma^p$.

\paragraph{3. Upper bound for the two-sample testing.} Next, we show the bounds for two-sample testing. For $R\in(0,e^{-e})$, define
\begin{align*}
&\Phi_{\mathrm{upper}}^{\mathrm{ts}}(R)\\
&\quad\coloneqq \begin{dcases} \frac{R^{-d/4p+1}}{\sqrt{n}}+\frac{R^{-d/2p+1}}{n^{3/4}}+\frac{R^{-d/4p+1}}{n^{3/4}\sqrt{\epsilon}}+\frac{R^{-d/2p+1}}{n\sqrt{\epsilon}}+\frac{1}{n\epsilon}&d>4p,\\[6pt] \frac{\log(1/R)\log\log(1/R)}{\sqrt{n}}+\frac{R^{-d/2p+1}}{n^{3/4}}+\frac{\log^{3/2} (1/R)\log\log(1/R)}{n^{3/4}\sqrt{\epsilon}}+\frac{R^{-1}}{n\sqrt{\epsilon}}+\frac{1}{n\epsilon}&d=4p,\\[6pt] \frac{1}{\sqrt{n}}+\frac{R^{-d/2p+1}}{n^{3/4}}+\frac{1}{n^{3/4}\sqrt{\epsilon}}+\frac{R^{-d/2p+1}}{n\sqrt{\epsilon}}+\frac{1}{n\epsilon}&2p<d<4p,\\[6pt] \frac{1}{\sqrt{n}}+\frac{\log(1/R)\log\log(1/R)}{n^{3/4}}+\frac{1}{n^{3/4}\sqrt{\epsilon}}+\frac{\log^{3/2} (1/R)\log\log(1/R)}{n\sqrt{\epsilon}}+\frac{1}{n\epsilon}&d=2p,\\[6pt] \frac{1}{\sqrt{n}}+\frac{1}{n^{3/4}}+\frac{1}{n^{3/4}\sqrt{\epsilon}}+\frac{1}{n\epsilon}&d<2p. \end{dcases} \end{align*}
Then there exists a constant $C>0$ depending only on $d$ and $p$ such that \begin{align*} C^{-1}\Psi_{\mathrm{upper}}^{\mathrm{ts}}(R)\leq\Phi_{\mathrm{upper}}^{\mathrm{ts}}(R). \end{align*}

For $c_{\mathrm{upper}}^{\mathrm{ts}}>0$, define
\begin{align*}
g_{\mathrm{upper}}^{\mathrm{ts}}(n)\coloneqq c_{\mathrm{upper}}^{\mathrm{ts}}(n^{-\left(\frac{1}{2}\wedge\frac{3p}{2d}\right)}\vee(n\sqrt{\epsilon})^{-\left(1\wedge\frac{2p}{d}\right)}\vee(n^{3/4}\sqrt{\epsilon})^{-\left(1\wedge\frac{4p}{d}\right)}\vee(n\epsilon)^{-1}).
\end{align*}
We show that
\begin{align*}
\left(C_{\mathrm{upper}}^{\mathrm{ts}}\vee e^e\vee \gamma^{-p}/2\right)\mathrm{diam}^p([0,1]^d)\left(g_{\mathrm{upper}}^{\mathrm{ts}}(n)\wedge (e^{-e}\wedge 2\gamma^p)\right)\geq\rho_{n,\epsilon}^{\mathrm{ts}}([0,1]^d,p,\alpha,\beta).
\end{align*}
If $g_{\mathrm{upper}}^{\mathrm{ts}}(n)\geq e^{-e}\wedge 2\cdot 3^{-p}$, the desired result follows trivially since
\begin{align*}
\left(C_{\mathrm{upper}}^{\mathrm{ts}}\vee e^e\vee \gamma^{-p}/2\right)\mathrm{diam}^p([0,1]^d)\left(g_{\mathrm{upper}}^{\mathrm{ts}}(n)\wedge (e^{-e}\wedge 2\gamma^p)\right)&\geq \mathrm{diam}^p([0,1]^d)\\
&\geq\rho_{n,\epsilon}^{\mathrm{ts}}([0,1]^d,p,\alpha,\beta).
\end{align*}
We therefore consider the nontrivial regime $g_{\mathrm{upper}}^{\mathrm{ts}}(n)\in(0,e^{-e}\wedge 2\gamma^p)$.

The following lemma controls the logarithmic terms appearing at the boundary dimensions.
\begin{lemma}\label{lem:ts_log_domination}
Suppose that $c_{\mathrm{upper}}^{\mathrm{ts}}\geq1$, $n\geq2$, $\epsilon>0$, and $0<g_{\mathrm{upper}}^{\mathrm{ts}}(n)<e^{-e}$. Then
\begin{align*}
\frac{\log(1/g_{\mathrm{upper}}^{\mathrm{ts}}(n))\log\log(1/g_{\mathrm{upper}}^{\mathrm{ts}}(n))}{\sqrt{n}} &\leq \left(\frac{15}{e}\right)^{5/2}\frac{g_{\mathrm{upper}}^{\mathrm{ts}}(n)}{c_{\mathrm{upper}}^{\mathrm{ts}}}\qquad\textrm{if }d=4p,\\
\frac{\log^{3/2}(1/g_{\mathrm{upper}}^{\mathrm{ts}}(n))\log\log(1/g_{\mathrm{upper}}^{\mathrm{ts}}(n))}{n^{3/4}\sqrt{\epsilon}} &\leq \left(\frac{15}{e}\right)^{5/2}\frac{g_{\mathrm{upper}}^{\mathrm{ts}}(n)}{c_{\mathrm{upper}}^{\mathrm{ts}}}\qquad\textrm{if }d=4p,\\
\frac{\log(1/g_{\mathrm{upper}}^{\mathrm{ts}}(n))\log\log(1/g_{\mathrm{upper}}^{\mathrm{ts}}(n))}{n^{3/4}} &\leq \left(\frac{15}{e}\right)^{5/2}\frac{g_{\mathrm{upper}}^{\mathrm{ts}}(n)}{c_{\mathrm{upper}}^{\mathrm{ts}}}\qquad\textrm{if }d=2p,\\
\frac{\log^{3/2}(1/g_{\mathrm{upper}}^{\mathrm{ts}}(n))\log\log(1/g_{\mathrm{upper}}^{\mathrm{ts}}(n))}{n\sqrt{\epsilon}} &\leq \left(\frac{15}{e}\right)^{5/2}\frac{g_{\mathrm{upper}}^{\mathrm{ts}}(n)}{c_{\mathrm{upper}}^{\mathrm{ts}}}\qquad\textrm{if }d=2p.
\end{align*}
\end{lemma}

We now bound all non-logarithmic terms simultaneously. For convenience, denote $q_1\coloneqq 1\wedge\frac{4p}{d}$ and $q_2\coloneqq 1\wedge\frac{2p}{d}$. Also, denote
\begin{align*}
a_1=\frac{1}{\sqrt{n}},\qquad a_2=\frac{1}{n^{3/4}},\qquad a_3=\frac{1}{n^{3/4}\sqrt{\epsilon}},\qquad a_4=\frac{1}{n\sqrt{\epsilon}},\qquad a_5=\frac{1}{n\epsilon}.
\end{align*}
Then
\begin{align*}
g_{\mathrm{upper}}^{\mathrm{ts}}(n)\geq c_{\mathrm{upper}}^{\mathrm{ts}}(a_1^{q_1}\vee a_2^{q_2}\vee a_3^{q_1}\vee a_4^{q_2}\vee a_5).
\end{align*}
By Lemma~\ref{lem:power_domination}, every non-logarithmic term appearing in $\Phi_{\mathrm{upper}}^{\mathrm{ts}}(g_{\mathrm{upper}}^{\mathrm{ts}}(n))$ is bounded above by $g_{\mathrm{upper}}^{\mathrm{ts}}(n)/c_{\mathrm{upper}}^{\mathrm{ts}}$.

Denote
\begin{align*}
D_{\mathrm{ts}}=5+\left(\frac{6}{e}\right)^2+2\left(\frac{15}{e}\right)^{5/2}.
\end{align*}
Since $\Phi_{\mathrm{upper}}^{\mathrm{ts}}(R)$ contains at most five non-logarithmic terms and its logarithmic terms are controlled by Lemma~\ref{lem:ts_log_domination}, we obtain
\begin{align*}
\Phi_{\mathrm{upper}}^{\mathrm{ts}}(g_{\mathrm{upper}}^{\mathrm{ts}}(n))\leq\frac{D_{\mathrm{ts}}}{c_{\mathrm{upper}}^{\mathrm{ts}}}g_{\mathrm{upper}}^{\mathrm{ts}}(n).
\end{align*}
Since
\begin{align*}
C^{-1}\Psi_{\mathrm{upper}}^{\mathrm{ts}}(g_{\mathrm{upper}}^{\mathrm{ts}}(n))\leq\Phi_{\mathrm{upper}}^{\mathrm{ts}}(g_{\mathrm{upper}}^{\mathrm{ts}}(n)),
\end{align*}
we have
\begin{align*}
\Psi_{\mathrm{upper}}^{\mathrm{ts}}(g_{\mathrm{upper}}^{\mathrm{ts}}(n))\leq\frac{CD_{\mathrm{ts}}}{c_{\mathrm{upper}}^{\mathrm{ts}}}g_{\mathrm{upper}}^{\mathrm{ts}}(n).
\end{align*}
Therefore, choosing $c_{\mathrm{upper}}^{\mathrm{ts}}\geq\max\left\{1,CD_{\mathrm{ts}}\right\}$ gives
\begin{align*}
\Psi_{\mathrm{upper}}^{\mathrm{ts}}(g_{\mathrm{upper}}^{\mathrm{ts}}(n))\leq g_{\mathrm{upper}}^{\mathrm{ts}}(n).
\end{align*}
Therefore, $C_{\mathrm{upper}}^{\mathrm{ts}}\mathrm{diam}^p([0,1]^d)g_{\mathrm{upper}}^{\mathrm{ts}}(n)\geq \rho_{n,\epsilon}^{\mathrm{ts}}([0,1]^d,p,\alpha,\beta)$, by Theorem~\ref{thm:two_sample_minimax_upper} since $g_{\mathrm{upper}}^{\mathrm{ts}}(n)\leq 2\gamma^p$.

\paragraph{4. Lower bound for the two-sample testing.} Next, we show the lower bound. There exists a constant $C^\prime>0$ depending only on $d$ and $p$ such that
\begin{align*}
\Psi_{\mathrm{lower}}^{\mathrm{ts}}(R)\geq C^\prime\left(\frac{R^{-0\vee(d/4p-1)}}{\sqrt{n}}\vee\frac{R^{-0\vee(d/2p-1)}}{n^{3/4}}\vee\frac{R^{-0\vee(d/4p-1)}}{n^{3/4}\sqrt{\epsilon}}\vee\frac{R^{-0\vee(d/2p-1)}}{n\sqrt{\epsilon}}\vee\frac{1}{n\epsilon}\right).
\end{align*}
Define
\begin{align*}
g_{\mathrm{lower}}^{\mathrm{ts}}(n)\coloneqq c_{\mathrm{lower}}^{\mathrm{ts}}(n^{-\left(\frac{1}{2}\wedge\frac{3p}{2d}\right)}\vee(n\sqrt{\epsilon})^{-\left(1\wedge\frac{2p}{d}\right)}\vee(n^{3/4}\sqrt{\epsilon})^{-\left(1\wedge\frac{4p}{d}\right)}\vee(n\epsilon)^{-1})\wedge 2\gamma^p.
\end{align*}
We show that $C_{\mathrm{lower}}^{\mathrm{ts}}\mathrm{diam}^p([0,1]^d)g_{\mathrm{lower}}^{\mathrm{ts}}(n)\leq \rho_{n,\epsilon}^{\mathrm{ts}}([0,1]^d,p,\alpha,\beta)$. Choose
\begin{align*}
c_{\mathrm{lower}}^{\mathrm{ts}}=\min\left\{(C^\prime)^{q_1},(C^\prime)^{q_2},C^\prime\right\}.
\end{align*}
Then
\begin{align*}
g_{\mathrm{lower}}^{\mathrm{ts}}(n)\leq(C^\prime a_1)^{q_1}\vee(C^\prime a_2)^{q_2}\vee(C^\prime a_3)^{q_1}\vee(C^\prime a_4)^{q_2}\vee C^\prime a_5.
\end{align*}
Applying the second part of Lemma~\ref{lem:power_domination} to the term attaining the maximum on the right-hand side gives
\begin{align*}
C^\prime\left(a_1g_{\mathrm{lower}}^{\mathrm{ts}}(n)^{1-1/q_1}\vee a_2g_{\mathrm{lower}}^{\mathrm{ts}}(n)^{1-1/q_2}\vee a_3g_{\mathrm{lower}}^{\mathrm{ts}}(n)^{1-1/q_1}\vee a_4g_{\mathrm{lower}}^{\mathrm{ts}}(n)^{1-1/q_2}\vee a_5\right)\geq g_{\mathrm{lower}}^{\mathrm{ts}}(n).
\end{align*}
Since $1/q_1-1=0\vee(d/(4p)-1)$ and $1/q_2-1=0\vee(d/(2p)-1)$, it follows that
\begin{align*}
g_{\mathrm{lower}}^{\mathrm{ts}}(n)\leq\Psi_{\mathrm{lower}}^{\mathrm{ts}}(g_{\mathrm{lower}}^{\mathrm{ts}}(n)).
\end{align*}
Therefore, $C_{\mathrm{lower}}^{\mathrm{ts}}\mathrm{diam}^p([0,1]^d)g_{\mathrm{lower}}^{\mathrm{ts}}(n)\leq \rho_{n,\epsilon}^{\mathrm{ts}}([0,1]^d,p,\alpha,\beta)$ by Theorem~\ref{thm:two_sample_minimax_lower} since $g_{\mathrm{lower}}^{\mathrm{ts}}(n)\leq 2\gamma^p$.

\subsection{Proof of Lemma~\ref{lem:trunc_Wp}}\label{appendix:trunc_Wp}
We use the following facts, which hold for any $\pi \in \Pi(P,Q)$.
\begin{align}
    \int_{B(L)\times B(L)^c}d^p(x,y)d\pi(x,y)&= \sum_{l\geq 1}\int_{B(L)}\int_{B((l+1)L)\cap B(lL)^c}d^p(x,y)d\pi(x,y)\notag\\
    &\leq\sum_{l\geq 1}\int_{B(L)}\int_{B((l+1)L)\cap B(lL)^c}(l+2)^pL^pd\pi(x,y)\notag\\
    &\leq\sum_{l\geq 1}(l+2)^pL^pQ(B(lL)^c)\notag\\
    &\leq\sum_{l\geq 1}(l+2)^pL^p\left(c_1e^{-c_2lL}\right)\notag\\
    &\leq\sum_{l\geq 1}3^{pl}L^p\left(c_1e^{-c_2lL}\right)\notag\\
    &= \frac{c_1(3L)^pe^{-c_2L}}{1-3^pe^{-c_2L}}\label{eq:tail_approx1},
\end{align}
since $l+2\leq 3^l$. Also,
\begin{align}
    \int_{B(L)^c\times B(L)^c}d^p(x,y)d\pi(x,y)&= \sum_{k,l\geq 1}\int_{B((k+1)L)\cap B(kL)^c}\int_{B((l+1)L)\cap B(lL)^c}d^p(x,y)d\pi(x,y)\notag\\
    &\leq\frac{2c_1(6L)^pe^{-c_2L}}{1-3^pe^{-c_2L}}\label{eq:tail_approx2},
\end{align}
since
\begin{align*}
    &\sum_{k\geq l\geq 1}\int_{B((k+1)L)\cap B(kL)^c}\int_{B((l+1)L)\cap B(lL)^c}d^p(x,y)d\pi(x,y)\\
    &\leq\sum_{k\geq 1}\int_{B((k+1)L)\cap B(kL)^c}\int_{B((k+1)L)\cap B(L)^c}d^p(x,y)d\pi(x,y)\\
    &\leq\sum_{k\geq 1}\int_{B((k+1)L)\cap B(kL)^c}\int_{B((k+1)L)\cap B(L)^c}(2k+2)^pL^pd\pi(x,y)\\
    &\leq\sum_{k\geq 1}(2k+2)^pL^p P(B(kL)^c)\\
    &\leq\sum_{k\geq 1}(2k+2)^pL^p \left(c_1e^{-c_2kL}\right)\\
    &\leq\sum_{k\geq 1}3^{pk}(2L)^p \left(c_1e^{-c_2kL}\right)\\
    &= \frac{c_1(6L)^pe^{-c_2L}}{1-3^pe^{-c_2L}}.
\end{align*}
We show the bound by applying \eqref{eq:tail_approx1} and~\eqref{eq:tail_approx2}.
\begin{align*}
    W_p^p(P,Q)&= \inf_{\pi\in\Pi(P,Q)}\int_{\mathcal{X}\times\mathcal{X}}d^p(x,y)d\pi(x,y)\\
    &\leq\inf_{\pi\in\Pi(P,Q)}\int_{B(L)\times B(L)}d^p(x,y)d\pi(x,y)+\int_{B(L)\times B(L)^c}d^p(x,y)d\pi(x,y) \\
    &+\int_{B(L)^c\times B(L)}d^p(x,y)d\pi(x,y)
    +\int_{B(L)^c\times B(L)^c}d^p(x,y)d\pi(x,y)\\
    &\leq\inf_{\pi\in\Pi(P,Q)}\int_{B(L)\times B(L)}d^p(x,y)d\pi(x,y)+\frac{2c_1(3L)^pe^{-c_2L}}{1-3^pe^{-c_2L}}+\frac{2c_1(6L)^pe^{-c_2L}}{1-3^pe^{-c_2L}}\\
    &\leq P(B(L))Q(B(L))\inf_{\pi_L\in\Pi(P_L,Q_L)}\int_{B(L)\times B(L)}d^p(x,y)d\pi_L(x,y)+\frac{4c_1(6L)^pe^{-c_2L}}{1-3^pe^{-c_2L}}\\
    &= P(B(L))Q(B(L))W_p^p(P_L,Q_L)+\frac{4c_1(6L)^pe^{-c_2L}}{1-3^pe^{-c_2L}}.
\end{align*}
The last inequality holds since for every $\pi_L\in\Pi(P_L,Q_L)$, the measure $P(B(L))Q(B(L))\pi_L$ can be extended to a coupling $\pi\in\Pi(P,Q)$ without assigning any additional mass to $B(L)\times B(L)$. For example, given $\pi_L$, define
\begin{align*}
\pi(A)&=P(B(L))Q(B(L))\pi_L(A\cap (B(L)\times B(L)))+P\otimes Q(A\cap (B(L)\times B(L)^c))\\
&+P\otimes Q(A\cap (B(L)^c\times B(L)))+P\otimes Q(A\cap (B(L)^c\times B(L)^c)).
\end{align*}
Then for a Borel subset $A_1\subset\mathbb{R}^d$,
\begin{align*}
    \pi(A_1\times\mathbb{R}^d)&=P(B(L))Q(B(L))\pi_L((A_1\cap B(L))\times B(L))+P(A_1\cap B(L))Q(B(L)^c)\\
&+P(A_1\cap B(L)^c) Q(B(L))+P(A_1\cap B(L)^c)Q(B(L)^c).
\end{align*}
Since
\begin{align*}
    P(B(L))Q(B(L))\pi_L((A_1\cap B(L))\times B(L))&=P(B(L))Q(B(L))P_L(A_1\cap B(L))\\
    &=P(A_1\cap B(L))Q(B(L))\\
    P(A_1\cap B(L))Q(B(L)^c)+P(A_1\cap B(L)^c)Q(B(L)^c)&=P(A_1)Q(B(L)^c),
\end{align*}
we have
\begin{align*}
    \pi(A_1\times\mathbb{R}^d)&=P(A_1\cap B(L))Q(B(L))+P(A_1)Q(B(L)^c)+P(A_1\cap B(L)^c) Q(B(L))=P(A_1).
\end{align*}
Similarly, $\pi(\mathbb{R}^d\times A_1)=Q(A_1)$ so $\pi\in\Pi(P,Q)$.
\subsection{Proof of Theorem~\ref{thm:minimax_unbounded}}\label{appendix:minimax_unbounded}
For the proof, we consider the goodness-of-fit testing. An analogous argument proves the result for two-sample testing as well. 

\paragraph{Upper bound.} Choose $L=O(\log n)$ so that
\begin{align*}
    \frac{4c_1(6L)^pe^{-c_2L}}{1-3^pe^{-c_2L}}&\leq\rho_{n,\epsilon}^{\mathrm{gof}}(B(L),p,\alpha/2,\beta/2),\\
    nc_1e^{-c_2L}&\leq (\alpha\wedge\beta)/2.
\end{align*}
For every $P_{0L}\in \mathcal{P}(B(L))$, there exists an $\epsilon$-DP goodness-of-fit tester $\phi_{P_{0L}}$ on $B(L)$ achieving Type I and Type II errors at most $\alpha/2$ and $\beta/2$ respectively for the separation radius $2\rho_{n,\epsilon}^{\mathrm{gof}}(B(L),p,\alpha/2,\beta/2)$. Given $P_0$, we define the $\epsilon$-DP goodness-of-fit tester $\phi_{P_0}$ as follows:
\begin{align*}
    \phi_{P_0}(X_1,\ldots,X_n)=\phi_{P_{0L}}(\mathrm{clip}(X_1;L),\ldots,\mathrm{clip}(X_n;L)).
\end{align*}
Since clipping is applied separately to each observation and preserves adjacency, $\phi_{P_0}$ is $\epsilon$-DP. We show that $\phi_{P_0}$ can test
\begin{align*}
    H_0:P=P_0\textrm{ vs }H_1:W_p^p(P,P_0)\geq3\rho_{n,\epsilon}^{\mathrm{gof}}(B(L),p,\alpha/2,\beta/2),
\end{align*}
with Type I and Type II errors at most $\alpha$ and $\beta$, respectively.

Let $E$ be an event that no clipping occurs:
\begin{align*}
    E\coloneqq \{X_i\in B(L)\textrm{ for }i=1,\ldots,n\}.
\end{align*}
Conditional on $E$, $X_1,\ldots,X_n$ remain independent and each has distribution $P_L$. Also, $\mathbb{P}(E^c)\leq (\alpha\wedge\beta)/2$ by the choice of $L$.

Under $H_0$, $P_L=P_{0L}$, so $\phi_{P_0}$ achieves Type I error at most $\alpha/2$ given $E$. Thus, the Type I error of $\phi_{P_0}$ is
\begin{align*}
    \mathbb{P}(\phi_{P_0}(X_1,\ldots,X_n)=1|H_0)&= \mathbb{P}(\phi_{P_0}(X_1,\ldots,X_n)=1|H_0,E)\mathbb{P}(E)+\mathbb{P}(\phi_{P_0}(X_1,\ldots,X_n)=1,E^c|H_0)\\
    &\leq\alpha/2+\mathbb{P}(E^c)\\
    &\leq\alpha.
\end{align*}
Under the alternative hypothesis, Lemma~\ref{lem:trunc_Wp} implies
\begin{align*}
    W_p^p(P_L,P_{0L})&\geq P(B(L))P_0(B(L))W_p^p(P_L,P_{0L})\\
    &\geq W_p^p(P,P_{0})- \frac{4c_1(6L)^pe^{-c_2L}}{1-3^pe^{-c_2L}}\\
    &\geq 2\rho_{n,\epsilon}^{\mathrm{gof}}(B(L),p,\alpha/2,\beta/2).
\end{align*}
Thus, $\phi_{P_{0L}}$ achieves Type II error at most $\beta/2$ given $E$. The same decomposition shows that the unconditional Type II error is at most $\beta$. Therefore,
\begin{align*}
    \rho_{n,\epsilon}^{\mathrm{gof}}(\mathbb{R}^d,p,c_1,c_2,\alpha,\beta)\lesssim\rho_{n,\epsilon}^{\mathrm{gof}}(B(L),p,\alpha/2,\beta/2).
\end{align*}
Finally, since $B(L)\subset[-L,L]^d$, a goodness-of-fit tester on $[-L,L]^d$ can also be applied to distributions supported on $B(L)$. Hence,
\begin{align*}
\rho_{n,\epsilon}^{\mathrm{gof}}(B(L),p,\alpha/2,\beta/2)\leq \rho_{n,\epsilon}^{\mathrm{gof}}([-L,L]^d,p,\alpha/2,\beta/2).
\end{align*}
By scaling and translation,
\begin{align*}
\rho_{n,\epsilon}^{\mathrm{gof}}([-L,L]^d,p,\alpha/2,\beta/2)=(2L)^p\rho_{n,\epsilon}^{\mathrm{gof}}([0,1]^d,p,\alpha/2,\beta/2).
\end{align*}
Since $L=O(\log n)$, the desired upper bound follows.

\paragraph{Lower bound.} We show the lower bound for goodness-of-fit testing. To show the bound for exponentially decaying distributions, we use a slightly modified version of Lemma~\ref{lem:grid}. Denote $\boldsymbol{0}_d\coloneqq (0,\ldots,0)\in\mathbb{R}^d$.
\begin{lemma}\label{lem:grid_2}
Let $k$ be a positive integer, and let $x_1,\ldots,x_k$ satisfy $x_k=\boldsymbol{0}_d$ and
\begin{align*}
    d(x_i,x_{i^\prime})\geq h,\quad i\not=i^\prime.
\end{align*}
For $u\in\mathcal{P}([k-1])$, define
\begin{align*}
    S_hu\coloneqq \frac{1}{2}\delta_{\boldsymbol{0}_d}+\frac{1}{2}\sum_{i=1}^{k-1}u_i\delta_{x_i}.
\end{align*}
Then for all $u,v\in\mathcal{P}([k-1])$,
\begin{align*}
    W_p^p(S_hu,S_hv)\geq \frac{1}{4}h^p\lVert u-v\rVert_1.
\end{align*}
\end{lemma}
The proof is omitted, as it is almost identical to that of Lemma~\ref{lem:grid}.

Denote $r_0=c_2^{-1}\log(2c_1)$ and $B\coloneqq B(r_0)\backslash B(r_0/2)$. For a positive integer $j$, denote $h_j^B\coloneqq \mathrm{diam}(B)\gamma^j/6$. We choose an $h_j^B$-separated set $\{x_i^{(j)}\}_{i=1}^{N_j^B}$ of size $N_j^B\coloneqq \mathcal{N}(B,\mathrm{diam}(B)\gamma^j/3)$ in $B$, which exists by Lemma~\ref{lem:sep_cover}. That is,
\begin{align*}
    \lVert x_i^{(j)}-x_{i^\prime}^{(j)}\rVert_1\geq h_j^B
\end{align*}
holds for every $i\neq i^\prime$. Let $x_{N_j^B+1}^{(j)}=\mathbf{0}\in[0,1]^d$. Then
\begin{align*}
    \lVert x_i^{(j)}-x_{N_j^B+1}^{(j)}\rVert_1\geq \frac{r_0}{2}\geq \frac{\mathrm{diam}(B)}{4}>h_j^B.
\end{align*}
Thus, by Lemma~\ref{lem:grid_2}, for $u,v\in\mathcal{P}([N_j^B])$, the following holds.
\begin{align*}
    W_p^p(S_{h_j^B}u,S_{h_j^B}v)\geq\frac{1}{4}h_j^{Bp}\lVert u-v\rVert_1.
\end{align*}
We show that $S_{h_j^B}$ embeds $\mathcal{P}([N_j^B])$ into $\mathcal{P}(\mathbb{R}^d,c_1,c_2)$. For $u\in\mathcal{P}([N_j^B])$, suppose $r\leq r_0$. Then,
\begin{align*}
    S_{h_j^B}u(B(r)^c)\leq \frac{1}{2}=c_1e^{-c_2r_0}\leq c_1e^{-c_2r}.
\end{align*}
If $r\geq r_0$ then $S_{h_j^B}u(B(r_0)^c)=0$ since the support of $S_{h_j^B}u$ is contained in $B(r_0)$. Thus, $S_{h_j^B}u\in\mathcal{P}(\mathbb{R}^d,c_1,c_2)$.

Now, let $u_0$ be a uniform distribution on $[N_j^B]$. For $R>0$, suppose there exists an $\epsilon$-DP goodness-of-fit test $\phi$ with Type I and Type II errors at most $\alpha$ and $\beta$, respectively, for the following Wasserstein goodness-of-fit testing problem.
\begin{align*}
    H_0:P=S_{h_j^B}u_0\textrm{ vs }H_1:W_p^p(S_{h_j^B}u,S_{h_j^B}u_0)\geq R.
\end{align*}
Next, let $Z_1,\ldots,Z_n\overset{\textrm{i.i.d.}}{\sim}u$ for some $u\in\mathcal{P}([N_j^B])$ and consider the following $\epsilon$-DP goodness-of-fit testing problem on $[N_j^B]$.
\begin{align}\label{eq:reduced_l1_gof_test_unbounded}
    H_0:u=u_0\textrm{ vs }H_1:\lVert u-u_0\rVert_1\geq 4h_j^{B-p}R.
\end{align}
Next, generate $A_1,\ldots,A_n\overset{\textrm{i.i.d.}}{\sim}\mathrm{Bernoulli}(1/2)$ independently from $Z_1,\ldots,Z_n$ and define
\begin{align*}
Y_i\coloneqq \begin{cases}
x_{Z_i}^{(j)},&A_i=1,\\[6pt]
\boldsymbol{0}_d,&A_i=0.
\end{cases}
\end{align*}
Then $Y_1,\ldots,Y_n$ are i.i.d. from $S_{h_j^B}u$. Moreover, the randomized test $\phi(Y_1,\ldots,Y_n)$ is $\epsilon$-DP as a function of $Z_1,\ldots,Z_n$, since, conditional on $A_1,\ldots,A_n$, neighboring input datasets are mapped to neighboring or identical datasets in $\mathbb{R}^d$. By Lemma~\ref{lem:grid_2}, $\phi(Y_1,\ldots,Y_n)$ is a test with Type I and Type II errors at most $\alpha$ and $\beta$ for \eqref{eq:reduced_l1_gof_test_unbounded}. Using the fact that the minimax radius cannot exceed the induced separation radius $4h_j^{B-p}R$, we can show a lower bound for $R$. Since $B$ is a fixed $d$-dimensional annulus with nonempty interior, $N_j^B\asymp\gamma^{-jd}$, with constants depending only on $d,c_1$, and $c_2$. Therefore, applying the finite-domain $\ell_1$ lower bound and repeating the optimization over $j$ used in the proof of Theorem~\ref{thm:goodness_of_fit_minimax_lower} gives
\begin{align*}
\rho_{n,\epsilon}^{\mathrm{gof}}(\mathbb{R}^d,p,c_1,c_2,\alpha,\beta)\gtrsim\left( n^{-\left(\frac{1}{2}\wedge\frac{2p}{d}\right)}\vee (n\sqrt{\epsilon})^{-\left(1\wedge \frac{2p}{d}\right)}\vee(n^{3/4}\sqrt{\epsilon})^{-\left(1\wedge \frac{4p}{d}\right)}\vee(n\epsilon)^{-1}\right)\wedge1,
\end{align*}
up to a constant factor depending only on $d,p,c_1,c_2,\alpha$, and $\beta$.

\subsection{Proof of Lemma~\ref{lem:log_domination_by_g}}
Note that the following holds whenever $d=4p$ or $d=2p$.
\begin{align*}
    g_{\mathrm{upper}}^{\mathrm{gof}}(n)\geq\frac{c_{\mathrm{upper}}^{\mathrm{gof}}}{\sqrt{n}}\geq\frac{1}{\sqrt{n}}.
\end{align*}
Thus, $\log(1/g_{\mathrm{upper}}^{\mathrm{gof}}(n))\leq\frac{1}{2}\log n\leq \log(en)$ holds. This implies
\begin{align}
    \log(1/g_{\mathrm{upper}}^{\mathrm{gof}}(n))\log\log(1/g_{\mathrm{upper}}^{\mathrm{gof}}(n))&\leq \log n\log\log (en)=\ell_n\label{eq:1}\\
    \log^{3/2}(1/g_{\mathrm{upper}}^{\mathrm{gof}}(n))\log\log(1/g_{\mathrm{upper}}^{\mathrm{gof}}(n))&\leq\left(\frac{15}{e}\right)^{5/2}n^{1/12}\label{eq:2}.
\end{align}
The second inequality can be shown as follows.
\begin{align*}
\log^{3/2}(1/g_{\mathrm{upper}}^{\mathrm{gof}}(n))\log\log(1/g_{\mathrm{upper}}^{\mathrm{gof}}(n))&\leq\log^{5/2}(1/g_{\mathrm{upper}}^{\mathrm{gof}}(n))\\
&\leq\frac{1}{2^{5/2}}\log^{5/2} n\\
&\leq\left(\frac{15}{e}\right)^{5/2}n^{1/12},
\end{align*}
where the first inequality holds since $0\leq\log x\leq x$ if $x\geq 1$, the second inequality holds since $\log (1/g_{\mathrm{upper}}^{\mathrm{gof}}(n))\leq \frac{1}{2}\log n$, and the last inequality holds since $\sup_{x>0}x^{5/2}e^{-x/12}=30^{5/2}e^{-5/2}$.

\paragraph{First inequality.} Suppose first that $d=4p$. Since $c_{\mathrm{upper}}^{\mathrm{gof}}\geq\ell_2^{-1}$ and $\ell_n\geq\ell_2$,
\begin{align*}
g_{\mathrm{upper}}^{\mathrm{gof}}(n)\geq c_{\mathrm{upper}}^{\mathrm{gof}}\frac{\ell_n}{\sqrt{n}}
\end{align*}
holds. Also,
\begin{align*}
\frac{\log(1/g_{\mathrm{upper}}^{\mathrm{gof}}(n))\log\log(1/g_{\mathrm{upper}}^{\mathrm{gof}}(n))}{\sqrt{n}}&\leq\frac{\ell_n}{2\sqrt{n}}\leq\frac{g_{\mathrm{upper}}^{\mathrm{gof}}(n)}{2c_{\mathrm{upper}}^{\mathrm{gof}}}
\end{align*}
by \eqref{eq:1}.

\paragraph{Second inequality.} By the definition of $g_{\mathrm{upper}}^{\mathrm{gof}}(n)$,
\begin{align*}
\frac{g_{\mathrm{upper}}^{\mathrm{gof}}(n)}{c_{\mathrm{upper}}^{\mathrm{gof}}}\geq(n\sqrt{\epsilon})^{-1/2}\vee(n\epsilon)^{-1}\geq\left((n\sqrt{\epsilon})^{-1/2}\right)^{2/3}\left((n\epsilon)^{-1}\right)^{1/3}=\frac{1}{n^{2/3}\sqrt{\epsilon}}
\end{align*}
holds. Thus,
\begin{align*}
\frac{1}{n^{3/4}\sqrt{\epsilon}}\leq\frac{g_{\mathrm{upper}}^{\mathrm{gof}}(n)}{c_{\mathrm{upper}}^{\mathrm{gof}}n^{1/12}}.
\end{align*}
Then
\begin{align*}
    \frac{\log^{3/2}(1/g_{\mathrm{upper}}^{\mathrm{gof}}(n))\log\log(1/g_{\mathrm{upper}}^{\mathrm{gof}}(n)) }{n^{3/4}\sqrt{\epsilon}}\leq\left( \frac{g_{\mathrm{upper}}^{\mathrm{gof}}(n)}{c_{\mathrm{upper}}^{\mathrm{gof}}n^{1/12}}\right)\left(\frac{15}{e}\right)^{5/2}n^{1/12}\leq\left(\frac{15}{e}\right)^{5/2} \frac{g_{\mathrm{upper}}^{\mathrm{gof}}(n)}{c_{\mathrm{upper}}^{\mathrm{gof}}},
\end{align*}
by \eqref{eq:2}.

\paragraph{Third inequality.} Suppose $d=2p$. By the definition of $g_{\mathrm{upper}}^{\mathrm{gof}}(n)$,
\begin{align*}
\frac{1}{n^{3/4}\sqrt{\epsilon}}\leq\frac{g_{\mathrm{upper}}^{\mathrm{gof}}(n)}{c_{\mathrm{upper}}^{\mathrm{gof}}}
\end{align*}
holds. Thus,
\begin{align*}
\frac{1}{n\sqrt{\epsilon}}\leq\frac{g_{\mathrm{upper}}^{\mathrm{gof}}(n)}{c_{\mathrm{upper}}^{\mathrm{gof}}n^{1/4}}.
\end{align*}
Then
\begin{align*}
    \frac{ \log^{3/2}(1/g_{\mathrm{upper}}^{\mathrm{gof}}(n))\log\log(1/g_{\mathrm{upper}}^{\mathrm{gof}}(n))}{n\sqrt{\epsilon}}\leq\left(\frac{g_{\mathrm{upper}}^{\mathrm{gof}}(n)}{c_{\mathrm{upper}}^{\mathrm{gof}}n^{1/4}}\right)\left(\frac{15}{e}\right)^{5/2}n^{1/12}\leq\left(\frac{15}{e}\right)^{5/2} \frac{g_{\mathrm{upper}}^{\mathrm{gof}}(n)}{c_{\mathrm{upper}}^{\mathrm{gof}}},
\end{align*}
by \eqref{eq:2}. \qed

\subsection{Proof of Lemma~\ref{lem:ts_log_domination}}
By the definition of $g_{\mathrm{upper}}^{\mathrm{ts}}(n)$,
\begin{align*}
g_{\mathrm{upper}}^{\mathrm{ts}}(n)\geq c_{\mathrm{upper}}^{\mathrm{ts}}n^{-3/8}\qquad\textrm{if }d=4p,
\end{align*}
and
\begin{align*}
g_{\mathrm{upper}}^{\mathrm{ts}}(n)\geq c_{\mathrm{upper}}^{\mathrm{ts}}n^{-1/2}\qquad\textrm{if }d=2p.
\end{align*}
Since $c_{\mathrm{upper}}^{\mathrm{ts}}\geq1$, in either case,
\begin{align*}
\log(1/g_{\mathrm{upper}}^{\mathrm{ts}}(n))\leq\frac{1}{2}\log n.
\end{align*}
Moreover, since $g_{\mathrm{upper}}^{\mathrm{ts}}(n)<e^{-e}$, we have
$\log(1/g_{\mathrm{upper}}^{\mathrm{ts}}(n))>e$. Hence, by the same argument
used in the proof of Lemma~\ref{lem:log_domination_by_g},
\begin{align*}
\log^{3/2}(1/g_{\mathrm{upper}}^{\mathrm{ts}}(n))
\log\log(1/g_{\mathrm{upper}}^{\mathrm{ts}}(n))
&\leq
\log^{5/2}(1/g_{\mathrm{upper}}^{\mathrm{ts}}(n))\\
&\leq
\frac{1}{2^{5/2}}\log^{5/2}n\\
&\leq
\left(\frac{15}{e}\right)^{5/2}n^{1/12}.
\end{align*}
Also, since $\log(1/g_{\mathrm{upper}}^{\mathrm{ts}}(n))>1$,
\begin{align*}
\log(1/g_{\mathrm{upper}}^{\mathrm{ts}}(n))
\log\log(1/g_{\mathrm{upper}}^{\mathrm{ts}}(n))
\leq
\log^{3/2}(1/g_{\mathrm{upper}}^{\mathrm{ts}}(n))
\log\log(1/g_{\mathrm{upper}}^{\mathrm{ts}}(n)).
\end{align*}
Therefore,
\begin{align}
\log(1/g_{\mathrm{upper}}^{\mathrm{ts}}(n))
\log\log(1/g_{\mathrm{upper}}^{\mathrm{ts}}(n))
&\leq\left(\frac{15}{e}\right)^{5/2}n^{1/12},
\label{eq:ts_log_bound_1}\\
\log^{3/2}(1/g_{\mathrm{upper}}^{\mathrm{ts}}(n))
\log\log(1/g_{\mathrm{upper}}^{\mathrm{ts}}(n))
&\leq\left(\frac{15}{e}\right)^{5/2}n^{1/12}.
\label{eq:ts_log_bound_2}
\end{align}
Suppose first that $d=4p$. Since
\begin{align*}
\frac{1}{\sqrt{n}}\leq\frac{g_{\mathrm{upper}}^{\mathrm{ts}}(n)}{c_{\mathrm{upper}}^{\mathrm{ts}}n^{1/8}}\leq\frac{g_{\mathrm{upper}}^{\mathrm{ts}}(n)}{c_{\mathrm{upper}}^{\mathrm{ts}}n^{1/12}},
\end{align*}
the first inequality follows from \eqref{eq:ts_log_bound_1}. Also, by the definition of $g_{\mathrm{upper}}^{\mathrm{ts}}(n)$,
\begin{align*}
\frac{g_{\mathrm{upper}}^{\mathrm{ts}}(n)}{c_{\mathrm{upper}}^{\mathrm{ts}}}\geq(n\sqrt{\epsilon})^{-1/2}\vee(n\epsilon)^{-1}\geq\left((n\sqrt{\epsilon})^{-1/2}\right)^{2/3}\left((n\epsilon)^{-1}\right)^{1/3}=\frac{1}{n^{2/3}\sqrt{\epsilon}},
\end{align*}
and hence
\begin{align*}
\frac{1}{n^{3/4}\sqrt{\epsilon}}\leq\frac{g_{\mathrm{upper}}^{\mathrm{ts}}(n)}{c_{\mathrm{upper}}^{\mathrm{ts}}n^{1/12}}.
\end{align*}
The second inequality follows from \eqref{eq:ts_log_bound_2}.

Suppose next that $d=2p$. We have
\begin{align*}
\frac{1}{n^{3/4}}\leq\frac{g_{\mathrm{upper}}^{\mathrm{ts}}(n)}{c_{\mathrm{upper}}^{\mathrm{ts}}n^{1/4}}\leq\frac{g_{\mathrm{upper}}^{\mathrm{ts}}(n)}{c_{\mathrm{upper}}^{\mathrm{ts}}n^{1/12}},
\end{align*}
so the third inequality follows from \eqref{eq:ts_log_bound_1}. Finally, since
\begin{align*}
\frac{g_{\mathrm{upper}}^{\mathrm{ts}}(n)}{c_{\mathrm{upper}}^{\mathrm{ts}}}\geq\frac{1}{n^{3/4}\sqrt{\epsilon}},
\end{align*}
we have
\begin{align*}
\frac{1}{n\sqrt{\epsilon}}\leq\frac{g_{\mathrm{upper}}^{\mathrm{ts}}(n)}{c_{\mathrm{upper}}^{\mathrm{ts}}n^{1/4}}\leq\frac{g_{\mathrm{upper}}^{\mathrm{ts}}(n)}{c_{\mathrm{upper}}^{\mathrm{ts}}n^{1/12}},
\end{align*}
and the fourth inequality follows from \eqref{eq:ts_log_bound_2}. \qed

\section{Proof of Theorem~\ref{thm:unavoidable_log}}\label{appendix:unavoidable_log}
In this subsection, we characterize the minimax separation radius $\rho_n^{\mathrm{gof}}([0,1]^d,B,\alpha,\beta)$ for goodness-of-fit testing over $[0,1]^d$. Let $\mathbb{T}^d$ be a $d$-dimensional torus, a quotient space of $\mathbb{R}^d$ with the equivalence relation $x\sim y\iff x-y\in\mathbb{Z}^d$. We define Wasserstein-1 distance on $\mathbb{T}^d$ as follows.
\begin{definition}[Wasserstein-1 distance on $\mathbb{T}^d$]
    Let $P,Q\in\mathcal{P}(\mathbb{T}^d)$. Then
    \begin{align*}
        \widetilde{W}_1(P,Q)\coloneqq\inf_{\pi\in\Pi(P,Q)}\int_{\mathbb{T}^d\times\mathbb{T}^d}d_{\mathbb{T}^d}(x,y)d\pi(x,y),
    \end{align*}
    where the distance function $d_{\mathbb{T}^d}$ on $\mathbb{T}^d$ is defined as
    \begin{align*}
        d_{\mathbb{T}^d}(x,y)\coloneqq\inf_{m\in \mathbb{Z}^d}\lVert x-y-m\rVert_2.
    \end{align*}
\end{definition}
Before the proof, we show that the Wasserstein-1 distance on $[0,1]^d$ is equivalent to the Wasserstein distance on $\mathbb{T}^d$. Thus, Wasserstein-1 goodness-of-fit testing on $[0,1]^d$ is equivalent to Wasserstein-1 goodness-of-fit testing on $\mathbb{T}^d$. Define an embedding $i:[0,1]^d\rightarrow\mathbb{T}^d$ by
\begin{align*}
    i(x)=x/3.
\end{align*}
Then
\begin{align*}
    \lVert x-y\rVert_1&\leq \sqrt{d}\lVert x-y\rVert_2\\
    &\leq3\sqrt{d}\lVert x/3-y/3\rVert_2\\
    &=3\sqrt{d}\inf_{m\in\mathbb{Z}^d}\lVert x/3-y/3-m\rVert_2\\
    &=3\sqrt{d}d_{\mathbb{T}^d}(i(x),i(y)).
\end{align*}
Also,
\begin{align*}
    d_{\mathbb{T}^d}(i(x),i(y))&=\inf_{m\in\mathbb{Z}^d}\lVert x/3-y/3-m\rVert_2\\
    &\leq \lVert x/3-y/3\rVert_2\\
    &\leq \lVert x/3-y/3\rVert_1\\
    &=\lVert x-y\rVert_1/3.
\end{align*}
Thus, for $P,Q\in\mathcal{P}_B([0,1]^d)$,
\begin{align*}
   3\widetilde{W}_1(\widetilde{P},\widetilde{Q})\leq  W_1(P,Q)\leq 3\sqrt{d}\widetilde{W}_1(\widetilde{P},\widetilde{Q}),
\end{align*}
where $\widetilde{P}$ and $\widetilde{Q}$ are distributions of $i(P)$ and $i(Q)$. Thus, two Wasserstein distances are equivalent. Denote
\begin{align*}
    \mathcal{P}_B(\mathbb{T}^d)&\coloneqq \left\{P\in\mathcal{P}(\mathbb{T}^d):\frac{dP}{d\lambda_{\mathbb{T}^d}}\leq B\right\}\\
    \mathcal{P}_{1}^{\mathrm{gof}}(P_0,\mathbb{T}^d,B,\rho)&\coloneqq\mathcal{P}_B(\mathbb{T}^d)\cap \{P\in\mathcal{P}(\mathbb{T}^d):\widetilde{W}_1(P,P_0)\geq \rho\},
\end{align*}
where $\lambda_{\mathbb{T}^d}$ is the Lebesgue measure on $\mathbb{T}^d$. Define the minimax separation radius of Wasserstein-1 testing on $\mathbb{T}^d$ as follows.
\begin{definition}
    Let $\phi_{P_0}:(\mathbb{T}^d)^n\rightarrow\{0,1\}$ be a goodness-of-fit testing procedure for \eqref{eq:ht_gof}. Given Type I and Type II error levels $\alpha,\beta\in(0,1/3]$, density upper bound $B$, and sample size $n$, the separation radius of $\phi_{P_0}$ is defined as
\begin{align*}
    \rho_{n}^{\mathrm{gof}}(\phi_{P_0},\mathbb{T}^d,B,\alpha,\beta)\coloneqq \inf\left\{\rho>0:\mathbb{E}_{P_0}[\phi_{P_0}]\leq\alpha,\sup_{P\in\mathcal{P}_{1}^{\mathrm{gof}}(P_0,\mathbb{T}^d,B,\rho)}\mathbb{E}_P[1-\phi_{P_0}]\leq\beta\right\},
\end{align*}
The minimax separation radius for goodness-of-fit testing is then defined as
\begin{align*} \rho_{n}^{\mathrm{gof}}(\mathbb{T}^d,B,\alpha,\beta)\coloneqq \sup_{P_0\in\mathcal{P}_B(\mathbb{T}^d)}\inf_{\phi_{P_0}}\rho_{n}^{\mathrm{gof}}(\phi_{P_0},\mathbb{T}^d,B,\alpha,\beta).\end{align*}
\end{definition}
Note that $\widetilde{P}\in \mathcal{P}_{3^dB}(\mathbb{T}^d)$. Thus,
\begin{align*}
    \rho_n^{\mathrm{gof}}(\mathbb{T}^d,3^dB,\alpha,\beta)\asymp\rho_n^{\mathrm{gof}}([0,1]^d,B,\alpha,\beta).
\end{align*}
Therefore, it suffices to develop a goodness-of-fit test on $\mathbb{T}^d$ for distributions with bounded densities. 

On $\mathbb{T}^d$, Fourier-analytic tools can be used to derive upper and lower bounds on the Wasserstein-1 distance (Lemmas~\ref{lem:W1_ub} and~\ref{lem:W1_lb}), following \citet{Bobkov2021simple}. Intuitively, these bounds relate the Wasserstein-1 distance to negative-order Sobolev norms, under which perturbations across multiple resolutions can accumulate, giving rise to the $\log^{1/4} (en)$ factor in the case of $d=4$. Using these bounds, we first establish the upper bound on $\rho_n^{\mathrm{gof}}(\mathbb{T}^d,B,\alpha,\beta)$ by constructing a Wasserstein-1 test. For the matching minimax lower bound, we construct distributions obtained by perturbing the uniform distribution across multiple resolutions.

We define Sobolev spaces, homogeneous Sobolev spaces, and negative-order Sobolev spaces.
\begin{definition}[Sobolev spaces and their norms on $\mathbb{T}^d$ and $\mathbb{R}^d$]
Let $1<r<\infty$, $r'=r/(r-1)$, and
$\mathcal X\in\{\mathbb{T}^d,\mathbb{R}^d\}$.
We define the following Sobolev spaces.
\begin{itemize}
    \item For $s\in\{1,2\}$, define
    \begin{align*}
        W^{s,r}(\mathcal X)
        &\coloneqq  
        \left\{
            f\in L^r(\mathcal X):
            \sum_{\lVert\beta\rVert_1\leq s}
            \|D^\beta f\|_{L^r(\mathcal X)}<\infty
        \right\},
    \end{align*}
    where $D^\beta\coloneqq \partial_{x_1}^{\beta_1}\cdots\partial_{x_d}^{\beta_d}$ equipped with the norm
    \begin{align*}
        \|f\|_{W^{s,r}(\mathcal X)}
        &\coloneqq  
        \sum_{\lVert\beta\rVert_1\leq s}
        \|D^\beta f\|_{L^r(\mathcal X)}.
    \end{align*}

    \item We define the homogeneous Sobolev space by
    \begin{align*}
        \dot W^{1,r}(\mathcal X)
        &\coloneqq  
        \left\{
            [f]:
            f\in W_{\mathrm{loc}}^{1,r}(\mathcal X),
            \ \nabla f\in L^r(\mathcal X;\mathbb R^d)
        \right\},
    \end{align*}
    where $[f]\coloneqq  \{f+c:c\in\mathbb R\}$ denotes
    the equivalence class of $f$ modulo constants,
    and $W_{\mathrm{loc}}^{1,r}(\mathcal X)$ denotes
    the space of locally $W^{1,r}$ functions.
    The space is equipped with the norm
    \begin{align*}
        \|[f]\|_{\dot W^{1,r}(\mathcal X)}
        &\coloneqq  
        \|\nabla f\|_{L^r(\mathcal X)}\\
        &\coloneqq  
        \left(
            \int_{\mathcal X}
            \|\nabla f(x)\|_2^r\,dx
        \right)^{1/r}.
\end{align*}
    \item The corresponding negative-order Sobolev
    spaces are defined as the dual spaces
    \begin{align*}
        W^{-1,r}(\mathcal X)
        &\coloneqq  
        \left(W^{1,r'}(\mathcal X)\right)^*,\\
        \dot W^{-1,r}(\mathcal X)
        &\coloneqq  
        \left(\dot W^{1,r'}(\mathcal X)\right)^*.
    \end{align*}
    Their norms are defined by
    \begin{align*}
        \|f\|_{W^{-1,r}(\mathcal X)}
        &\coloneqq  
        \sup_{\substack{
            g\in W^{1,r'}(\mathcal X)\\
            \|g\|_{W^{1,r'}(\mathcal X)}\leq1
        }}
        |\langle f,g\rangle|,\\
        \|f\|_{\dot W^{-1,r}(\mathcal X)}
        &\coloneqq  
        \sup_{\substack{
            [g]\in\dot W^{1,r'}(\mathcal X)\\
            \|[g]\|_{\dot W^{1,r'}(\mathcal X)}\leq1
        }}
        |\langle f,[g]\rangle|.
    \end{align*}
\end{itemize}
\end{definition}
For $t>0$, let $\gamma_t$ be a probability measure supported on $\mathbb{T}^d$ with density
\begin{align*}
    p_t(x)\coloneqq\sum_{m\in\mathbb{Z}^d}e^{2\pi i\langle m,x\rangle-\lVert m\rVert_2^2t}.
\end{align*}
Also, for $P\in\mathcal{P}(\mathbb{T}^d)$, denote the Gaussian smoothed distribution $P_t$ of $P$ as follows: for every Borel subset $A$ of $\mathbb{T}^d$,
\begin{align*}
P_t(A)\coloneqq\int_{\mathbb{T}^d}\int_{\mathbb{T}^d}\mathbf{1}(x+y\in A)dP(x)d\gamma_t(y).
\end{align*}
Then Wasserstein-1 distances are bounded using negative-order Sobolev norms.
\begin{lemma}[Proposition 2 of \cite{Bobkov2021simple}]\label{lem:W1_ub}
    Given two probability distributions $P$ and $Q$ on $\mathbb{T}^d$ the following holds for any $t>0$.
    \begin{align*}
        \widetilde{W}_1(P,Q)\leq\lVert P_t-Q_t\rVert_{\dot{W}^{-1,2}(\mathbb{T}^d)}+2\sqrt{2dt}.
    \end{align*}
\end{lemma}
We note that the negative-order Sobolev norm can be represented as a weighted $\ell_2$ discrepancy between Fourier coefficients \citep{Bobkov2021simple}.
\begin{align}\label{eq:W1_ub}
\widetilde{W}_1(P,Q)\leq\left(\sum_{m\neq\mathbf{0}_d}\frac{e^{-2\lVert m\rVert_2^2t}}{\lVert m\rVert_2^2}\left|\mathbb{E}_P[e^{2\pi i\langle m,X\rangle}]-\mathbb{E}_Q[e^{2\pi i\langle m,X\rangle}]\right|^2\right)^{1/2}+2\sqrt{2dt}.
\end{align}
Also, Wasserstein-1 distances can be lower bounded as follows.
\begin{lemma}\label{lem:W1_lb}
    Let $P\neq Q\in\mathcal{P}(\mathbb{T}^d)$ be measures with densities $p,q\in L^\infty([0,1)^d)$. Then
\begin{align}\label{eq:W1_lb}
     \widetilde{W}_1(P,Q)\gtrsim \frac{\lVert P-Q\rVert_{\dot{W}^{-1,2}(\mathbb{T}^d)}^3}{\lVert P-Q\rVert_{\dot{W}^{-1,4}(\mathbb{T}^d)}^2},
\end{align}
up to a constant factor depending on $d$.
\end{lemma}
\subsection{Upper bound}
We construct a goodness-of-fit test based on the Fourier representation of the negative-order Sobolev norm in \eqref{eq:W1_ub}. The argument consists of three steps. First, after smoothing at scale \(t>0\), \eqref{eq:W1_ub} bounds the Wasserstein distance by a weighted \(\ell_2\) discrepancy between the Fourier coefficients of \(P\) and \(P_0\), up to a smoothing error of order \(\sqrt t\). Hence, if \(\widetilde W_1(P,P_0)\) is sufficiently large relative to \(\sqrt t\), the corresponding weighted Fourier discrepancy must also be large. We estimate the squared weighted discrepancy by a second-order U-statistic. Then the sum of squared weights contributes to the variance of the U-statistics and the separation radius. As the following lemma shows, the sum of squared weights grows logarithmically for $d=4$, yielding the $\log^{1/4}n $ factor.
\begin{lemma}\label{lem:weight_sum}
There exists a constant $c_{\mathrm{thres}}$ depending on $d$ such that, for $t\in (0,1)$,
    \begin{align*}
        c_{\mathrm{thres}}^{-1}\sum_{m\neq\mathbf{0}_d,m\in\mathbb{Z}^d}\frac{e^{-2\lVert m\rVert_2^2t}}{\lVert m\rVert_2^{4}}\leq \begin{cases}
          1,&\textrm{ if }d<4\\
           \log\frac{e}{t},&\textrm{ if }d=4\\
           t^{-(d-4)/2},&\textrm{ if }d>4.
        \end{cases}
    \end{align*}
\end{lemma}

We construct a test as follows. For a given null distribution $P_0\in\mathcal{P}_B(\mathbb{T}^d)$ and $X_1,\ldots,X_n\overset{\textrm{i.i.d.}}{\sim}P$ for some $P\in\mathcal{P}_B(\mathbb{T}^d)$, denote
\begin{align*}
\varphi_{0,m}&\coloneqq \mathbb{E}_{P_0}[e^{2\pi i\langle m,X\rangle}]\\  
w_m&\coloneqq\frac{e^{-2\lVert m\rVert_2^2t}}{\lVert m\rVert_2^2}\\
A_M&\coloneqq \{m\in\mathbb{Z}^d:m\neq\mathbf{0}_d,\lVert m\rVert_2\leq M\}\\
Z_{m,l}&\coloneqq e^{2\pi i\langle m,X_l\rangle}-\varphi_{0,m}.
\end{align*}
Given a threshold $\rho_{\mathrm{thres}}$, define the test as follows.
\begin{align*}
    \phi_{P_0}(X_1,\ldots,X_n)&\coloneqq \mathbf{1}\left(T_{P_0}(X_1,\ldots,X_n)\geq \rho_{\mathrm{thres}}\right)\\
    T_{P_0}(X_1,\ldots,X_n)&\coloneqq  \sum_{m\in A_M}\frac{w_m}{n(n-1)}\sum_{l\neq l^\prime} Z_{m,l}\overline{Z_{m,l^\prime}}.
\end{align*}
For convenience, denote
\begin{align*}
\mu_m&\coloneqq\mathbb{E}_P[Z_{m,1}]\\
    \varphi_m&\coloneqq \mathbb{E}_{P}[e^{2\pi i\langle m,X\rangle}].
\end{align*}
Then the following holds.
\begin{lemma}\label{lem:W1_minimax_ub}
Set $M,t$, and $\rho_{\mathrm{thres}}$ so that they satisfy
    \begin{align*}
        \rho_{\mathrm{thres}}&\geq\left(\sqrt{\frac{4B\sum_{m\neq\mathbf{0}_d}w_m^2 }{\alpha\wedge\beta}}\vee \frac{4B}{\alpha\wedge\beta}\right)\frac{1}{n}\\
        t&\leq\frac{\rho_{\mathrm{thres}}}{d}\wedge1\\
        M&\geq \frac{C_\textrm{remain}}{\rho_{\mathrm{thres}}}\vee \left(\frac{d-1}{t}+\sqrt{d}\right)
    \end{align*}
where $c_{\mathrm{thres}}$ and $C_\textrm{remain}$ are constants depending on $d$ in Lemma~\ref{lem:weight_sum} and Lemma~\ref{lem:remainder}, respectively. Then, for $n\geq2$, $\phi_{P_0}$ achieves Type I and Type II errors at most $\alpha$ and $\beta$ for the following Wasserstein goodness-of-fit testing problem:
\begin{align*}
    H_0:P=P_0\textrm{ vs }H_1:\widetilde{W}_1(P,P_0)\geq 4\sqrt{2\rho_{\mathrm{thres}}}.
\end{align*}
\end{lemma}
We use Lemma~\ref{lem:W1_minimax_ub} to establish the upper bound on
$\rho_n^{\mathrm{gof}}(\mathbb{T}^d,B,\alpha,\beta)$.

By Lemma~\ref{lem:weight_sum}, the squared
weights satisfy
\begin{align*}
    c_{\mathrm{thres}}^{-1}\sum_{m\neq\mathbf{0}_d}w_m^2
    \leq
    \begin{cases}
        1, & d<4,\\
        \log\frac{e}{t}, & d=4,\\
        t^{-(d-4)/2}, & d>4.
    \end{cases}
\end{align*}
We first show the bound holds for sufficiently large $n$. Suppose $d<4$. Set
\begin{align*}
    \rho_{\mathrm{thres}}&=\left(\sqrt{\frac{4Bc_{\mathrm{thres}}}{\alpha\wedge\beta}}
    \vee
    \frac{4B}{\alpha\wedge\beta}\right)\frac{1}{n}\\
    t&=\frac{\rho_{\mathrm{thres}}}{d}\\
    M&= \left\lceil
    \frac{C_{\textrm{remain}}}{\rho_{\mathrm{thres}}}
    \vee
    \left(
    \frac{d(d-1)}{\rho_{\mathrm{thres}}}+\sqrt d
    \right)
    \right\rceil.
\end{align*}
For all sufficiently large $n$, we have $\rho_{\mathrm{thres}}\leq d$, and hence $t=\rho_{\mathrm{thres}}/d\leq 1$. Therefore, Lemma~\ref{lem:W1_minimax_ub} gives a test with Type I and Type II errors at most $\alpha$ and $\beta$, respectively, whenever
\begin{align*}
    \widetilde W_1(P,P_0)
    \gtrsim n^{-1/2}.
\end{align*}
Thus,
\begin{align*}
    \rho_n^{\mathrm{gof}}(\mathbb T^d,B,\alpha,\beta)
    \lesssim n^{-1/2}
\end{align*}
for all sufficiently large $n$.

Next, suppose $d=4$. Set
\begin{align*}
     \rho_{\mathrm{thres}}&=\left(\sqrt{\frac{4Bc_{\mathrm{thres}}}{\alpha\wedge\beta}}\vee   \frac{4B}{\alpha\wedge\beta}\right)\frac{\sqrt{\log(en)}}{n}\\
    t&=\frac{\rho_{\mathrm{thres}}}{4}\\
    M&= \left\lceil
    \frac{C_{\textrm{remain}}}{\rho_{\mathrm{thres}}}
    \vee
    \left(
    \frac{12}{\rho_{\mathrm{thres}}}+2
    \right)
    \right\rceil.
\end{align*}
For all sufficiently large $n$, $\rho_{\mathrm{thres}}\leq4$. Also,
\begin{align*}
    \log\frac{e}{t}
    =
    \log\frac{4e}{\rho_{\mathrm{thres}}}
    =\log\left(\frac{4en}{\sqrt{\log(en)}}
    \right)-\log \left(\sqrt{\frac{4Bc_{\mathrm{thres}}}{\alpha\wedge\beta}}
    \vee
    \frac{4B}{\alpha\wedge\beta}\right)   \leq \log(en)\leq 2\log n,
\end{align*}
since $\frac{4B}{\alpha\wedge\beta}>\frac{4}{\alpha\wedge\beta}>4$. Therefore,
Lemma~\ref{lem:W1_minimax_ub} yields
\begin{align*}
    \rho_n^{\mathrm{gof}}(\mathbb T^4,B,\alpha,\beta)
    \lesssim
    \sqrt{\rho_{\mathrm{thres}}}
    \asymp
    \frac{\log^{1/4}n }{\sqrt n}
\end{align*}
for all sufficiently large $n$.

Finally, suppose $d>4$. Set
\begin{align*}
     \rho_{\mathrm{thres}}&=Kn^{-4/d}\\
    t&=\frac{\rho_{\mathrm{thres}}}{d}\\
    M&= \left\lceil
    \frac{C_{\textrm{remain}}}{\rho_{\mathrm{thres}}}
    \vee
    \left(
    \frac{d(d-1)}{\rho_{\mathrm{thres}}}+\sqrt{d}
    \right)
    \right\rceil,
\end{align*}
for some $K>0$. For all sufficiently large $n$, $\rho_{\mathrm{thres}}\leq d$. Also,
\begin{align*}
\frac{1}{n}\sqrt{\frac{4B\sum_{m\neq\mathbf{0}_d} w_m^2}{\alpha\wedge\beta}}&\leq\frac{1}{n}
\sqrt{\frac{4Bc_{\mathrm{thres}}t^{-(d-4)/2}}{\alpha\wedge\beta}}\\
&=\frac{1}{n}\sqrt{\frac{4d^{d/2-2}Bc_{\mathrm{thres}}\rho_{\mathrm{thres}}^{-(d-4)/2}}{\alpha\wedge\beta}}\\
&=\rho_{\mathrm{thres}}\sqrt{\frac{4d^{d/2-2}Bc_{\mathrm{thres}}}{\alpha\wedge\beta}}K^{-d/4}.
\end{align*}
Thus, choosing $K\geq \frac{4B}{\alpha\wedge\beta}\vee\left(\frac{4d^{d/2-2}Bc_{\mathrm{thres}}}{\alpha\wedge\beta}\right)^{2/d}$, implies
\begin{align*}
    \rho_{\mathrm{thres}}\geq \left(\sqrt{\frac{4B\sum_{m\neq\mathbf{0}_d}w_m^2}{\alpha\wedge\beta}}\vee\frac{4B}{\alpha\wedge\beta}\right)\frac{1}{n}.
\end{align*}
Applying Lemma~\ref{lem:W1_minimax_ub}, we obtain
\begin{align*}
    \rho_n^{\mathrm{gof}}(\mathbb T^d,B,\alpha,\beta)
    \lesssim
    \sqrt{\rho_{\mathrm{thres}}}
    \asymp n^{-2/d}
\end{align*}
for all sufficiently large $n$.

It remains to consider $e\leq n<n_0$. Define
\begin{align*}
    r_d(n)\coloneqq\begin{cases}
        n^{-1/2},&d<4\\
        n^{-1/2}\log^{1/4}n ,&d=4\\
        n^{-2/d},&d>4.
    \end{cases}
\end{align*}
Since $r_d(n)>0$ for every $n\geq e$,
\begin{align*}
\frac{\mathrm{diam}(\mathbb{T}^d)}{\min_{e\leq n^\prime<n_0}r_d(n^\prime)} r_d(n)\geq \mathrm{diam}(\mathbb{T}^d).
\end{align*}
Since $\sup_{P,Q\in\mathcal{P}(\mathbb{T}^d)}\widetilde{W}_1(P,Q)\leq \mathrm{diam}(\mathbb{T}^d)$, for $n\geq e$,
\begin{align*}
\frac{\mathrm{diam}(\mathbb{T}^d)}{\min_{e\leq n^\prime<n_0}r_d(n^\prime)}r_d(n)\geq \rho_n^{\mathrm{gof}}(\mathbb{T}^d,B,\alpha,\beta).
\end{align*}
Combining this with the preceding bound for $n\geq n_0$ yields
\begin{align*}
    \rho_n^{\mathrm{gof}}(\mathbb T^d,B,\alpha,\beta)
    \lesssim
    \begin{cases}
        n^{-1/2}, & d<4,\\
        n^{-1/2}\log^{1/4}n , & d=4,\\
        n^{-2/d}, & d>4
    \end{cases}
\end{align*}
for every $n\geq e$.
\subsection{Proof of Lemma~\ref{lem:W1_minimax_ub}}
We bound the Type I and Type II errors using Chebyshev's inequality. To this end, we use the following lemma on the expectation and variance of $T_{P_0}$.
\begin{lemma}\label{lem:T}
For $P\in\mathcal{P}_B(\mathbb{T}^d)$,
    \begin{align*}
    \mathbb{E}_P[T_{P_0}]&=\sum_{m\in A_M}w_m|\mu_m|^2\\    \mathrm{Var}_P(T_{P_0})&\leq\frac{2B}{n(n-1)}\sum_{m\in A_M}w_m^2+\frac{4B}{n}\sum_{m\in A_M}w_m^2|\mu_{m}|^2.
    \end{align*}
\end{lemma}
\paragraph{Type I error.} Assume $P=P_0$. By Chebyshev's inequality, the following holds.
\begin{align*}
\mathbb{P}\left(T_{P_0}\geq \mathbb{E}_{P_0}[T_{P_0}]+\sqrt{\frac{\mathrm{Var}_{P_0}\left(T_{P_0}\right)}{\alpha\wedge\beta}}\middle|H_0\right)\leq \alpha.
\end{align*}
$\mathbb{E}_{P_0}[T_{P_0}]=0$ and by Lemma~\ref{lem:T}, the following holds.
\begin{align*}
    \mathbb{E}_{P_0}[T_{P_0}]+\sqrt{\frac{\mathrm{Var}_{P_0}\left(T_{P_0}\right)}{\alpha\wedge\beta}}&\leq \sqrt{\frac{2B}{n(n-1)(\alpha\wedge\beta)}\sum_{m\in A_M}w_m^2}\\&\leq \frac{1}{n}    \sqrt{\frac{4B}{\alpha\wedge\beta}\sum_{m\in A_M}w_m^2}\\
    &\leq \frac{1}{n}\sqrt{\frac{4B\sum_{m\neq\mathbf{0}_d}w_m^2}{\alpha\wedge\beta}}
\end{align*}
since $n^2\leq 2n(n-1)$. Thus,
\begin{align*}
    \mathbb{E}_{P_0}[T_{P_0}]+\sqrt{\frac{\mathrm{Var}_{P_0}\left(T_{P_0}\right)}{\alpha\wedge\beta}}\leq \sqrt{\frac{4B\sum_{m\neq\mathbf{0}_d}w_m^2}{\alpha\wedge\beta}}\frac{1}{n}\leq \rho_{\mathrm{thres}}.
\end{align*}
Therefore, the Type I error is at most $\alpha$:
\begin{align*}
    \mathbb{P}\left(T_{P_0}\geq \rho_{\mathrm{thres}}\middle| P_0\right)&\leq\mathbb{P}\left(T_{P_0}\geq \mathbb{E}_{P_0}[T_{P_0}]+\sqrt{\frac{\mathrm{Var}_{P_0}\left(T_{P_0}\right)}{\alpha\wedge\beta}}\middle|H_0\right)\\
    &\leq \alpha.
\end{align*}
\paragraph{Type II error.}  Assume $\widetilde{W}_1(P,P_0)\geq 4\sqrt{2\rho_{\mathrm{thres}}}$. We first show that
\begin{align*}
    \sum_{m\in A_M}w_m|\mu_m|^2\geq 4\rho_{\mathrm{thres}}.
\end{align*}
By \eqref{eq:W1_ub}, the following holds.
\begin{align*}
    \left(\sum_{m\neq\mathbf{0}_d}w_m|\mu_m|^2\right)^{1/2}&\geq \widetilde{W}_1(P,P_0)-2\sqrt{2dt}\\
    &\geq 4\sqrt{2\rho_{\mathrm{thres}}}-2\sqrt{2dt}\\
    &\geq 2\sqrt{2\rho_{\mathrm{thres}}},
\end{align*}
where the last inequality holds since $t\leq \frac{\rho_{\mathrm{thres}}}{d}$. By Lemma~\ref{lem:remainder},
\begin{align*}
    \sum_{\lVert m\rVert_2>M}w_m|\mu_m|^2\leq 4\sum_{\lVert m\rVert_2>M}w_m\leq\frac{4C_{\textrm{remain}}}{M},
\end{align*}
where the first inequality holds since
\begin{align*}
    |\mu_m|=|\varphi_m-\varphi_{0,m}|\leq |\varphi_m|+|\varphi_{0,m}|\leq2.
\end{align*}
Thus,
\begin{align*}
    \sum_{m\in A_M}w_m|\mu_m|^2&\geq \sum_{m\neq\mathbf{0}_d}w_m|\mu_m|^2-\sum_{\lVert m\rVert_2>M}w_m|\mu_m|^2\\
    &\geq 8\rho_{\mathrm{thres}}-\frac{4C_{\textrm{remain}}}{M}\\
    &\geq 4\rho_{\mathrm{thres}}.
\end{align*}
Now we can bound the Type II error. By Chebyshev's inequality,
\begin{align*}
    \mathbb{P}\left(T_{P_0}\leq \mathbb{E}_P[T_{P_0}]-\sqrt{\frac{\mathrm{Var}_{P}\left(T_{P_0}\right)}{\alpha\wedge\beta}}\middle|H_1\right)\leq \beta.
\end{align*}
By Lemma~\ref{lem:T}, the variance term is bounded as follows. 
\begin{align*}
\sqrt{\frac{\mathrm{Var}_P(T_{P_0})}{\alpha\wedge\beta}}&\leq\sqrt{\frac{2B}{n(n-1)(\alpha\wedge\beta)}\sum_{m\in A_M}w_m^2+\frac{4B}{n(\alpha\wedge\beta)}\sum_{m\in A_M}w_m|\mu_m|^2}\\
&\leq \sqrt{\frac{2B}{n(n-1)(\alpha\wedge\beta)}\sum_{m\in A_M}w_m^2}+\sqrt{\frac{4B}{n(\alpha\wedge\beta)}\sum_{m\in A_M}w_m|\mu_m|^2}\\
&\leq \frac{1}{n}\sqrt{\frac{4B}{(\alpha\wedge\beta)}\sum_{m\neq\mathbf{0}_d}w_m^2}+\sqrt{\frac{4B}{n(\alpha\wedge\beta)}\sum_{m\in A_M}w_m|\mu_m|^2}.
\end{align*}
Note that
\begin{align*}
    \sum_{m\in A_M}w_m|\mu_m|^2\geq 4\rho_{\mathrm{thres}}\geq \frac{16B}{n(\alpha\wedge\beta)}.
\end{align*}
Thus, the second term can be bounded as follows:
\begin{align*}
    \sqrt{\frac{4B}{n(\alpha\wedge\beta)}\sum_{m\in A_M}w_m|\mu_m|^2}&\leq\sqrt{\left(\frac{1}{4}\sum_{m\in A_M}w_m|\mu_m|^2\right)\left(\sum_{m\in A_M}w_m|\mu_m|^2\right)}.\\
    &=\frac{1}{2}\sum_{m\in A_M}w_m|\mu_m|^2.
\end{align*}
Also,
\begin{align*}
\frac{1}{n}\sqrt{\frac{4B\sum_{m\neq\mathbf{0}_d}w_m^2}{(\alpha\wedge\beta)}}\leq  \rho_{\mathrm{thres}}.
\end{align*}
Since $\mathbb{E}_P[T_{P_0}]=\sum_{m\in A_M}w_m|\mu_m|^2$,
\begin{align*}
    \mathbb{E}_P[T_{P_0}]-\sqrt{\frac{\mathrm{Var}_{P}\left(T_{P_0}\right)}{\alpha\wedge\beta}}&= \sum_{m\in A_M}w_m|\mu_m|^2-\sqrt{\frac{\mathrm{Var}_{P}\left(T_{P_0}\right)}{\alpha\wedge\beta}}\\
    &\geq \frac{1}{2}\sum_{m\in A_M}w_m|\mu_m|^2-\rho_{\mathrm{thres}}\\
    &\geq 2\rho_{\mathrm{thres}}-\rho_{\mathrm{thres}}\\
    &=\rho_{\mathrm{thres}}.
\end{align*}
Therefore, the Type II error is at most $\beta$:
\begin{align*}
   \mathbb{P}\left(T_{P_0}\leq \rho_{\mathrm{thres}}\middle|H_1\right)&\leq  \mathbb{P}\left(T_{P_0}\leq \mathbb{E}_P[T_{P_0}]-\sqrt{\frac{\mathrm{Var}_{P}\left(T_{P_0}\right)}{\alpha\wedge\beta}}\middle|H_1\right)\\
   &\leq \beta.
\end{align*}
\subsection{Proof of Lemma~\ref{lem:T}}
We first derive the expectation of $T_{P_0}$.
\paragraph{Expectation of $T_{P_0}$.} If $l\neq l^\prime$ then
\begin{align*}
    \mathbb E[Z_{m,l}\overline{Z_{m,l'}}]=\mathbb E[Z_{m,l}]\mathbb E[\overline{Z_{m,l'}}]=|\mu_m|^2.
\end{align*}
Next, we derive the upper bound of the variance of $T_{P_0}$.
\paragraph{Variance of $T_{P_0}$.}
By Lemmas~\ref{lem:cov_emp_error} and~\ref{lem:weighted_sum_to_weight_sum}, the variance of $T_{P_0}$ can be represented as follows.
\begin{align*}
    \mathrm{Var}(T_{P_0})&=\frac{1}{n(n-1)}\sum_{m,m^\prime\in A_M}w_mw_{m^\prime}\left(|\varphi_{m+m^\prime}-\varphi_m\varphi_{m^\prime}|^2+|\varphi_{m-m^\prime}-\varphi_m\overline{\varphi_{m^\prime}}|^2\right)\\
&+\frac{2}{n}\sum_{m,m^\prime\in A_M}w_mw_{m^\prime}\mathrm{Re}\left[\left(\varphi_{m+m^\prime}-\varphi_m\varphi_{m^\prime}\right)\overline{\mu_m\mu_{m^\prime}}\right]\\
&+\frac{2}{n}\sum_{m,m^\prime\in A_M}w_mw_{m^\prime}\mathrm{Re}\left[\left(\varphi_{m-m^\prime}-\varphi_m\overline{\varphi_{m^\prime}}\right)\overline{\mu_m}\mu_{m^\prime}\right]\\
    &=\frac{2}{n(n-1)}\sum_{m,m^\prime\in A_M}w_mw_{m^\prime}|\varphi_{m-m^\prime}-\varphi_m\overline{\varphi_{m^\prime}}|^2\\
    &+\frac{4}{n}\sum_{m,m^\prime\in A_M}w_mw_{m^\prime}\mathrm{Re}\left[\left(\varphi_{m-m^\prime}-\varphi_m\overline{\varphi_{m^\prime}}\right)\overline{\mu_m}\mu_{m^\prime}\right]\\
    &\leq \frac{2B}{n(n-1)}\sum_{m\in A_M}w_m^2+\frac{4B}{n}\sum_{m\in A_M}w_m^2|\mu_{m}|^2,
\end{align*}
where the second equality holds since $m\in A_M\iff -m\in A_M$ and $w_m=w_{-m}$.
\subsection{Proof of Lemma~\ref{lem:cov_emp_error}}
\begin{lemma}\label{lem:cov_emp_error}
\begin{align*}
    \mathrm{Cov}_P\left(\sum_{l\neq l^\prime} Z_{m,l}\overline{Z_{m,l^\prime}},\sum_{j\neq j^\prime} Z_{m^\prime,j}\overline{Z_{m^\prime,j^\prime}}\right)&=n(n-1)\left(\left|\varphi_{m+m^\prime}-\varphi_m\varphi_{m^\prime}\right|^2+\left|\varphi_{m-m^\prime}-\varphi_m\overline{\varphi_{m^\prime}}\right|^2\right)\\
&+2n(n-1)^2\mathrm{Re}\left[\left(\varphi_{m+m^\prime}-\varphi_m\varphi_{m^\prime}\right)\overline{\mu_m\mu_{m^\prime}}\right]\\
&+2n(n-1)^2\mathrm{Re}\left[\left(\varphi_{m-m^\prime}-\varphi_m\overline{\varphi_{m^\prime}}\right)\overline{\mu_m}\mu_{m^\prime}\right].
\end{align*}
\end{lemma}
\begin{proof}
Define
\begin{align*}
Y_{m,l}\coloneqq Z_{m,l}-\mu_m=e^{2\pi i\langle m,X_l\rangle}-\varphi_m,
\end{align*}
so that $\mathbb{E}_P[Y_{m,l}]=0$. Also, let
\begin{align*}
A_m\coloneqq \sum_{l\neq l^\prime}Y_{m,l}\overline{Y_{m,l^\prime}},\qquad L_m\coloneqq (n-1)\sum_{l=1}^n\left(\overline{\mu_m}Y_{m,l}+\mu_m\overline{Y_{m,l}}\right).
\end{align*}
Then
\begin{align*}
\sum_{l\neq l^\prime}Z_{m,l}\overline{Z_{m,l^\prime}}=A_m+L_m+n(n-1)|\mu_m|^2.
\end{align*}
Then
\begin{align*}
    \mathrm{Cov}(Y_{m,l}\overline{Y_{m,l^\prime}},\overline{\mu_m}Y_{m,l}+\mu_m \overline{Y_{m,l}})&=\mathbb{E}[Y_{m,l}\overline{Y_{m,l^\prime}}(\overline{\mu_m}Y_{m,l}+\mu_m \overline{Y_{m,l}})]\\
    &=\mathbb{E}[Y_{m,l}(\overline{\mu_m}Y_{m,l}+\mu_m \overline{Y_{m,l}})]\mathbb{E}[\overline{Y_{m,l^\prime}}]\\
    &=0.
\end{align*}
Thus,
\begin{align*}
\mathrm{Cov}_P(A_m,L_{m^\prime})=\mathrm{Cov}_P(L_m,A_{m^\prime})=0.
\end{align*}
Hence
\begin{align*}
\mathrm{Cov}_P\left(\sum_{l\neq l^\prime}Z_{m,l}\overline{Z_{m,l^\prime}},\sum_{j\neq j^\prime}Z_{m^\prime,j}\overline{Z_{m^\prime,j^\prime}}\right)=\mathrm{Cov}_P(A_m,A_{m^\prime})+\mathrm{Cov}_P(L_m,L_{m^\prime}).
\end{align*}

We begin by computing $\mathrm{Cov}_P(A_m,A_{m^\prime})$. Since $\mathbb{E}_P[A_m]=0$, the only nonzero terms in $\mathbb{E}_P[A_mA_{m^\prime}]$ are those for which either $(j,j^\prime)=(l,l^\prime)$ or $(j,j^\prime)=(l^\prime,l)$. Therefore,
\begin{align*}
\mathrm{Cov}_P(A_m,A_{m^\prime})=n(n-1)\left(\left|\mathbb{E}_P[Y_{m,1}Y_{m^\prime,1}]\right|^2+\left|\mathbb{E}_P[Y_{m,1}\overline{Y_{m^\prime,1}}]\right|^2\right).
\end{align*}
Moreover,
\begin{align*}
\mathbb{E}_P[Y_{m,1}Y_{m^\prime,1}]&=\varphi_{m+m^\prime}-\varphi_m\varphi_{m^\prime},\\
\mathbb{E}_P[Y_{m,1}\overline{Y_{m^\prime,1}}]&=\varphi_{m-m^\prime}-\varphi_m\overline{\varphi_{m^\prime}}.
\end{align*}
Thus,
\begin{align*}
\mathrm{Cov}_P(A_m,A_{m^\prime})=n(n-1)\left(\left|\varphi_{m+m^\prime}-\varphi_m\varphi_{m^\prime}\right|^2+\left|\varphi_{m-m^\prime}-\varphi_m\overline{\varphi_{m^\prime}}\right|^2\right).
\end{align*}

Next, by independence across the sample index,
\begin{align*}
\mathrm{Cov}_P(L_m,L_{m^\prime})=n(n-1)^2\mathbb{E}_P\left[\left(\overline{\mu_m}Y_{m,1}+\mu_m\overline{Y_{m,1}}\right)\left(\overline{\mu_{m^\prime}}Y_{m^\prime,1}+\mu_{m^\prime}\overline{Y_{m^\prime,1}}\right)\right].
\end{align*}
Expanding the expectation and pairing conjugate terms gives
\begin{align*}
\mathrm{Cov}_P(L_m,L_{m^\prime})={}&2n(n-1)^2\mathrm{Re}\left[\left(\varphi_{m+m^\prime}-\varphi_m\varphi_{m^\prime}\right)\overline{\mu_m\mu_{m^\prime}}\right]\\
&+2n(n-1)^2\mathrm{Re}\left[\left(\varphi_{m-m^\prime}-\varphi_m\overline{\varphi_{m^\prime}}\right)\overline{\mu_m}\mu_{m^\prime}\right].
\end{align*}
Combining the two displays proves the stated identity.
\end{proof}
\subsection{Proof of Lemma~\ref{lem:weighted_sum_to_weight_sum}}
\begin{lemma}\label{lem:weighted_sum_to_weight_sum}
\begin{align*}
\sum_{m,m^\prime\in A_M}w_mw_{m^\prime}|\varphi_{m-m^\prime}-\varphi_m\overline{\varphi_{m^\prime}}|^2&\leq B\sum_{m\in A_M}w_m^2\\
\sum_{m,m^\prime\in A_M}w_mw_{m^\prime}\mathrm{Re}\left[\left(\varphi_{m-m^\prime}-\varphi_m\overline{\varphi_{m^\prime}}\right)\overline{\mu_m}\mu_{m^\prime}\right]&\leq B\sum_{m\in A_M}w_m^2|\mu_{m}|^2.
\end{align*}
\end{lemma}
\begin{proof}
For the proof, we use the fact that $\mathrm{Var}_P(f(X))\leq B\mathrm{Var}_{\mathrm{Unif}(\mathbb{T}^d)}(f(X))$. This can be shown as follows.
    \begin{align*}
        \mathrm{Var}_P(f(X))&=\inf_{a\in\mathbb{C}}\mathbb{E}_P[|f(X)-a|^2]\\
        &\leq\inf_{a\in\mathbb{C}}\mathbb{E}_{\mathrm{Unif}(\mathbb{T}^d)}[B|f(X)-a|^2]\\
        &=B\mathrm{Var}_{\mathrm{Unif}(\mathbb{T}^d)}(f(X)).
    \end{align*}
\paragraph{First inequality.} Let $\Sigma_P$ be the covariance matrix of $(e^{2\pi i\langle m,X\rangle})_{m\in A_M}$. We first show that $\Sigma_P\preceq BI$.  For $v\in\mathbb{C}^{A_M}$, let $f(X)=\sum_{m\in A_M}\overline{v_m} e^{2\pi i\langle m,X\rangle}$. Then
\begin{align*}
    v^\ast\Sigma_Pv&=\mathrm{Var}_P(f(X))\\
    &\leq B\mathrm{Var}_{\mathrm{Unif(\mathbb{T}^d)}}(f(X))\\
    &\leq B\sum_{m\in A_M}|v_m|^2\\
    &\leq B\lVert v\rVert_2^2.
\end{align*}
Thus, $\Sigma_P\preceq BI$.

Next, we bound the weighted sum. Let $W=\mathrm{diag}(w_m)_{m\in A_M}$. Since $[\Sigma_P]_{m,m^\prime}=\mathrm{Cov}(e^{2\pi i\langle m,X\rangle},e^{2\pi i\langle m^\prime,X\rangle})=\varphi_{m-m^\prime}-\varphi_m\overline{\varphi_{m^\prime}}$, the following holds.
\begin{align*}
    \sum_{m,m^\prime\in A_M}w_mw_{m^\prime}|\varphi_{m-m^\prime}-\varphi_m\overline{\varphi_{m^\prime}}|^2&=\lVert W^{1/2}\Sigma_P W^{1/2}\rVert_F^2\\
    &=\mathrm{tr}\left(W^{1/2}\Sigma_PW\Sigma_PW^{1/2}\right)\\
    &=\mathrm{tr}\left(\Sigma_P^{1/2}W\Sigma_PW\Sigma_P^{1/2}\right)\\
    &\leq B\mathrm{tr}\left(\Sigma_P^{1/2}W^2\Sigma_P^{1/2}\right)\\
    &=B\mathrm{tr}\left(W\Sigma_P W\right)\\
    &=B\sum_{m\in A_M}w_m^2(\Sigma_P)_{mm}\\
    &=B\sum_{m\in A_M}w_m^2(1-|\varphi_m|^2)\\
    &\leq B\sum_{m\in A_M}w_m^2,
\end{align*}
where we used the fact that $\mathrm{tr}(A)\leq\mathrm{tr}(B)$ if $A\preceq B$.
\paragraph{Second inequality.} This time, we set $f(X)=\sum_{m\in A_M}w_m\overline{\mu_m}e^{2\pi i\langle m,X\rangle}$. Then
\begin{align*}
    \sum_{m,m^\prime\in A_M}w_mw_{m^\prime}\mathrm{Re}\left[\left(\varphi_{m-m^\prime}-\varphi_m\overline{\varphi_{m^\prime}}\right)\overline{\mu_m}\mu_{m^\prime}\right]&=\mathrm{Var}_P\left(f(X)\right)\\
    &\leq B\mathrm{Var}_{\mathrm{Unif}(\mathbb{T}^d)}\left(f(X)\right)\\
    &=B\sum_{m\in A_M}w_m^2|\mu_{m}|^2.
\end{align*}
\end{proof}
\subsection{Proof of Lemma~\ref{lem:weight_sum}}
For $a\in \mathbb{Z}$, define $I_a\coloneqq [a-1/2,a]$ if $a\geq 0$ and $I_a\coloneqq  [a,a+1/2]$ if $a<0$. Let $Q_m\coloneqq I_{m_1}\times\cdots\times I_{m_d}$. Let $m\neq \mathbf{0}_d$. Then $x\in Q_m$ implies
\begin{align*}
    \lVert x\rVert_2^2\leq \lVert m\rVert_2^2+\frac{d}{4}\leq\frac{d+4}{4}\lVert m\rVert_2^2.
\end{align*}
Thus,
\begin{align*}
    \frac{e^{-2\lVert x\rVert_2^2t}}{\lVert x\rVert_2^{4}}&\geq \frac{e^{-2dt/4}e^{-2\lVert m\rVert_2^2t}}{\left(\frac{d+4}{4}\lVert m\rVert_2^2\right)^{2}}\geq\frac{e^{-2d/4}}{(d/4+1)^{2}}\frac{e^{-2\lVert m\rVert_2^2t}}{\lVert m\rVert_2^{4}}.
\end{align*}
Then
\begin{align*}
    \sum_{m\neq\mathbf{0}_d}\frac{e^{-2\lVert m\rVert_2^2t}}{\lVert m\rVert_2^{4}}&\leq2^de^{2d/4}(d/4+1)^2\sum_{m\neq\mathbf{0}_d}\int_{Q_m}\frac{e^{-2\lVert x\rVert_2^2t}}{\lVert x\rVert_2^{4}}dx \\
    &\leq2^de^{2d/4}(d/4+1)^2\int_{\lVert x\rVert_2\geq1/2}\frac{e^{-2\lVert x\rVert_2^2t}}{\lVert x\rVert_2^{4}}dx\\
    &=2^de^{2d/4}(d/4+1)^2\mathrm{Vol}(S^{d-1})\int_{1/2}^\infty\frac{e^{-2r^2t}}{r^{4}}r^{d-1}dr\\
    &=2^de^{2d/4}(d/4+1)^2\mathrm{Vol}(S^{d-1})\int_{t/2}^\infty e^{-u}\left(\frac{u}{2t}\right)^{(d-5)/2}\frac{du}{2\sqrt{2tu}}\\
    &=2^{d-1}(2t)^{-(d-4)/2}e^{2d/4}(d/4+1)^2\mathrm{Vol}(S^{d-1})\int_{t/2}^\infty e^{-u}u^{(d-4)/2-1}du.
\end{align*}
If $d> 4$ then
\begin{align*}
    \int_{t/2}^\infty e^{-u}u^{(d-4)/2-1}du\leq\int_{0}^\infty e^{-u}u^{(d-4)/2-1}du=\Gamma\left(\frac{d-4}{2}\right).
\end{align*}
If $d<4$ then
\begin{align*}
    \int_{t/2}^\infty e^{-u}u^{(d-4)/2-1}du&=\int_{t/2}^{1/2} e^{-u}u^{(d-4)/2-1}du+\int_{1/2}^\infty e^{-u}u^{(d-4)/2-1}du\\
    &\leq\int_{t/2}^{1/2} u^{(d-4)/2-1}du+(1/2)^{(d-4)/2}\int_{1/2}^\infty e^{-u}du\\
    &\leq \frac{(t/2)^{-(4-d)/2}}{(4-d)/2}+(1/2)^{(d-4)/2}e^{-1/2}\\
    &\lesssim t^{-(4-d)/2}.
\end{align*}
If $d=4$ then
\begin{align*}
\int_{t/2}^\infty e^{-u}u^{(d-4)/2-1}du&=\int_{t/2}^\infty e^{-u}u^{-1}du\\
&= \int_{t/2}^{1/2} e^{-u}u^{-1}du+\int_{1/2}^\infty e^{-u}u^{-1}du\\
&\leq \int_{t/2}^{1/2} u^{-1}du+\int_{1/2}^\infty e^{-u}du\\
&=\log\frac{1}{t}+e^{-1/2}.
\end{align*}
Thus,
\begin{align*}
 \int_{t/2}^\infty e^{-u}u^{(d-4)/2-1}du\lesssim\log\frac{e}{t}.
\end{align*}
\subsection{Proof of Lemma~\ref{lem:remainder}}
\begin{lemma}\label{lem:remainder}
For a given $t\in (0,1)$, there exists a constant $C_\textrm{remain}$ depending on $d$ such that if $M\geq \frac{d-1}{2t}+\sqrt{d}$ then 
    \begin{align*}
        \sum_{\lVert m\rVert_2>M}w_m\leq\frac{C_{\textrm{remain}}}{M}.
    \end{align*}
\end{lemma}
\begin{proof}
We first show that there exists a constant $C_{\textrm{remain}}>0$ such that if $r\geq \frac{d-1}{2t}$ then
\begin{align*}
r^{d-3}e^{-2r^2t}\leq \frac{1}{r^2}.
\end{align*}
We use the fact that $\log r\leq r-1$ if $r\geq 1$.
\begin{align*}
    \log \left(r^{d-1}e^{-2r^2t}\right)&=(d-1)\log r-2r^2t\\
    &\leq (d-1)(r-1)-2r^2t\\
    &\leq (d-1)r-2r^2t\\
    &\leq 0.
\end{align*}
For $a\in \mathbb{Z}$, define $I_a\coloneqq [a-1/2,a]$ if $a\geq 0$ and $I_a\coloneqq  [a,a+1/2]$ if $a<0$. Let $Q_m\coloneqq I_{m_1}\times\cdots\times I_{m_d}$. If $\lVert m\rVert_2>M$ then by the triangle inequality,
\begin{align*}
    Q_m\subset \{x\in\mathbb{R}^d:\lVert x-m\rVert_2\leq \sqrt{d}/2\}\subset \{x\in\mathbb{R}^d:\lVert x\rVert_2\geq M-\sqrt{d}/2\}.
\end{align*}
Thus,
\begin{align*}
\sum_{\lVert m\rVert_2>M}w_m&\leq2^d\sum_{\lVert m\rVert_2> M}\int_{Q_m}\frac{e^{-2\lVert x\rVert_2^2t}}{\lVert x\rVert_2^{2}}dx \\
    &\leq2^d\int_{\lVert x\rVert_2\geq M-\sqrt{d}/2}\frac{e^{-2\lVert x\rVert_2^2t}}{\lVert x\rVert_2^{2}}dx\\
    &=2^d\mathrm{Vol}(S^{d-1})\int_{M-\sqrt{d}/2}^\infty\frac{e^{-2r^2t}}{r^{2}}r^{d-1}dr\\
    &\leq2^d\mathrm{Vol}(S^{d-1})\int_{M-\sqrt{d}/2}^\infty\frac{1}{r^2}dr\\
    &\leq\frac{2^d\mathrm{Vol}(S^{d-1})}{M-\sqrt{d}/2}\\
    &\leq\frac{2^{d+1}\mathrm{Vol}(S^{d-1})}{M}
\end{align*}
where the second inequality holds since $M-\sqrt{d}/2\geq \frac{d-1}{2t}$.
\end{proof}

\subsection{Lower bound}
We derive the minimax lower bound by constructing a family of distributions supported on $\mathbb{T}^d$ using an orthonormal Daubechies wavelet basis. We first introduce multiresolution approximations of $L^2(\mathbb{R}^d)$ and then recall the construction and properties of Daubechies wavelets.
\begin{definition}[{\citealp{meyer1992wavelets}}]\label{def:MA}
An $r$-regular multiresolution approximation of $L^2(\mathbb{R}^d)$ is an increasing sequence $V_j$, $j \in \mathbb{Z}$, of closed linear subspaces of $L^2(\mathbb{R}^d)$ with the following properties:
\begin{enumerate}
\item $\bigcap_{j=-\infty}^{\infty} V_j = \{0\},\qquad \bigcup_{j=-\infty}^{\infty} V_j \text{ is dense in } L^2(\mathbb{R}^d)$
\item $f(x) \in V_j\iff f(2x) \in V_{j+1}, \qquad \text{for all } f \in L^2(\mathbb{R}^d),\ j \in \mathbb{Z}$
\item $f(x) \in V_0\iff f(x-k) \in V_0 \qquad\text{for all } f \in L^2(\mathbb{R}^d),\ k \in \mathbb{Z}^d$
\item There exists a function $g(x) \in V_0$ such that $\{g(x-k):k\in\mathbb{Z}^d\}$ is a Riesz basis of $V_0$, i.e., $\mathrm{span}\{g(x-k):k\in\mathbb{Z}^d\}$ is dense in $V_0$ and, for every real sequence $\{\alpha_k\}_{k\in\mathbb{Z}^d}$,
\begin{align*}
    \sum_{k}|\alpha_k|^2\asymp \left\lVert\sum_{k\in\mathbb{Z}^d}\alpha_kg(x-k)\right\rVert_{L^2(\mathbb{R}^d)}^2,
\end{align*}
up to constant factors.
\item For every positive integer $l$, and every multi-index $\beta=(\beta_1,\ldots,\beta_d)$ such that $\lVert\beta\rVert_1\leq r$, there exists a constant $C_l$ depending on $l$ such that
\begin{align*}
    |D^\beta g(x)|\leq C_l(1+\lVert x\rVert_2)^{-l}.
\end{align*}
\end{enumerate}
\end{definition}
A multiresolution approximation provides a dyadic decomposition of the function space by resolution level. Let $W_j$ denote an orthogonal complement of $V_{j}\subset V_{j+1}$. That is, $V_{j+1}=V_{j}\oplus W_j$. The space $W_j$ contains the components at scale $j$, whereas $V_j$ contains components at coarser scales. Then we can decompose a given function in $L^2(\mathbb{R}^d)$ by projecting onto $V_0$ and $W_j$s and decomposing the projected components by the basis of each projection space. The collection of these basis functions is called a wavelet basis. Thus, wavelets decompose a given function according to multiresolution levels. Among them, we use Daubechies wavelets to construct the family of probability densities.
\begin{definition}[{Daubechies wavelet on $\mathbb{R}^d$; \citealp{Daubechies1988Orthonormal,meyer1992wavelets}}]\label{def:Daubechies_wavelet}
For $r>0$, let $\psi^0,\psi^1:\mathbb{R}\to\mathbb{R}$ be functions satisfying the following.
\begin{enumerate}
\item There exists a universal constant $C>0$ such that $\psi^0$ and $\psi^1$ are supported on $[-Cr,Cr]$ and $\left|\frac{d^k}{dx^k}\psi^0\right|,\left|\frac{d^k}{dx^k}\psi^1\right|\leq C_N(1+|x|)^{-N}$ for every $N\geq1$ and $0\leq k\leq r$.
\item There exists a constant $C_r$ depending on $r$ such that $\lVert\psi^0\rVert_\infty,\lVert\psi^1\rVert_\infty\leq C_r$.
\item $\int x^k\psi^1(x)dx=0$ for $0\leq k\leq r$.
\end{enumerate}
The $r$-regular Daubechies wavelets on $\mathbb{R}^d$ are the orthonormal functions $\{\psi^\epsilon_{j,k}\}_{(\epsilon,j,k)\in\Lambda}$ with $\Lambda\coloneqq  \{(\epsilon,j,k):\epsilon\in\{0,1\}^d\backslash\{\mathbf{0}_d\},j\geq0,k\in\mathbb{Z}^d\}\cup\{(\mathbf{0}_d,0,k):k\in\mathbb{Z}^d\}$ defined as
\begin{align*}
\psi_{j,k}^\epsilon(x)=2^{jd/2}\prod_{i=1}^d\psi^{\epsilon_i}(2^jx_i-k_i).
\end{align*}
\end{definition}
It is known that there exists an $r$-regular multiresolution approximation $\{V_j,j\in\mathbb{Z}\}$ of $L^2(\mathbb{R}^d)$ such that $V_0=\mathrm{span}\{\psi^{\mathbf{0}_d}_{0,k}:k\in\mathbb{Z}^d\}$ and $W_j=\mathrm{span}\{\psi^{\epsilon}_{j,k}:\epsilon\in\{0,1\}^d\backslash\{\mathbf{0}_d\},k\in\mathbb{Z}^d\}$ \citep{Daubechies1988Orthonormal,meyer1992wavelets}.

Daubechies wavelets have two properties that are particularly useful for our construction. First, their compact support allows us to construct perturbations supported inside \([0,1)^d\), which can be identified with densities on \(\mathbb T^d\). Second, sufficiently regular wavelet bases provide square-function characterizations of negative-order Sobolev norms (see Proposition~\ref{prop:hom_L_ps}). This enables us to control simultaneously the \(\dot W^{-1,2}\) and \(\dot W^{-1,4}\) norms of a multiscale perturbation directly through its wavelet coefficients. Combined with Lemma~\ref{lem:W1_lb}, this yields a family of alternatives that are statistically indistinguishable from the null while remaining well separated in $\widetilde{W}_1$.

Denote
\begin{align*}
    S_1(J)&\coloneqq\sum_{j=1}^J2^{(d-4)(j+1)}\\
    S_{1/2}(J)&\coloneqq\sum_{j=1}^J2^{(d-2)j/2}.
\end{align*}
Based on Daubechies wavelets, we can construct a family of distributions satisfying the following.
\begin{lemma}\label{lem:hard_family}
    Let $P_0=\mathrm{Unif}(\mathbb{T}^d)$ and $p_0$ be its density. For a positive integer $n$ and $J\geq \lceil \log_2(8C+4)\rceil$ and $\alpha,\beta\in(0,1/3]$, suppose
    \begin{align*}
\frac{(8C+2)^dC_2^d\log^{1/4}(1+(1-\alpha-\beta)^2) S_{1/2}(J)S_1^{-1/4}(J)}{2\sqrt{n}}\leq (B-1)\wedge1.
    \end{align*}
   Then there exists an index set $\mathcal{V}$ and a family of distributions $\{P_v\}_{v\in\mathcal{V}}\subset\mathcal{P}_{B}(\mathbb{T}^d)$ with densities $\{p_v\}_{v\in\mathcal{V}}$ such that
    \begin{enumerate}
        \item $\lVert P_0^{\otimes n}-\overline{P}_n\rVert_{\mathrm{TV}}< 1-\alpha-\beta$ where $\overline{P}_n\coloneqq  \frac{1}{|\mathcal{V}|}\sum_{v\in\mathcal{V}}P_v^{\otimes n}$.
        \item $\lVert p_0-p_v\rVert_{\dot{W}^{-1,2}}\gtrsim \rho_{n,J}\gtrsim \lVert p_0-p_v\rVert_{\dot{W}^{-1,4}}$, up to a constant factor depending only on $d,B,\alpha$, and $\beta$, where $\rho_{n,J}$ is defined by
        \begin{align*}
            \rho_{n,J}=\frac{S_1^{1/4}(J)\log^{1/4} (1+(1-\alpha-\beta)^2)}{\sqrt{n}}.
        \end{align*}
    \end{enumerate}
\end{lemma}
Now we are ready to derive the lower bound of $\rho_n^{\mathrm{gof}}(\mathbb{T}^d,B,\alpha,\beta)$. Lemma~\ref{lem:hard_family}(1) implies that the distributions are statistically indistinguishable. That is, no test can achieve Type I and Type II errors at most $\alpha$ and $\beta$, respectively, for the testing problem
\begin{align}\label{eq:hard_testing}
    H_0:P=P_0\textrm{ vs }H_1:P=P_v\textrm{ for some }v\in\mathcal{V}.
\end{align}
Suppose, for the sake of contradiction, that such a test $\phi$ exists. Then
\begin{align*}
    P_0^{\otimes n}(\phi(X_1,\ldots,X_n)=0)&\geq1-\alpha.
\end{align*}
Moreover, for every $v \in \mathcal V$,
\begin{align*}
    P_v^{\otimes n}(\phi(X_1,\ldots,X_n)=0)&\leq \beta.
\end{align*}
Averaging the latter inequality over $v \in \mathcal V$ yields
\begin{align*}
   \overline{P}_n(\phi(X_1,\ldots,X_n)=0)&\leq \beta.
\end{align*}
Consequently,
\begin{align*}
    \lVert P_0^{\otimes n}-\overline{P}_n\rVert_{\mathrm{TV}}&=\sup_{A}P_0^{\otimes n}(A)-\overline{P}_n(A)\\
    &\geq 1-\alpha-\beta,
\end{align*}
contradicting Lemma~\ref{lem:hard_family}(1).

Next, Lemma~\ref{lem:hard_family}(2), combined with Lemma~\ref{lem:W1_lb}, implies that
\begin{align*}
    \widetilde{W}_1(P_v,P_0)\gtrsim \rho_{n,J},
\end{align*}
for every $v\in\mathcal{V}$. Now, we assume that there exists a test that achieves Type I and Type II errors at most $\alpha$ and $\beta$, respectively, for the testing problem
\begin{align*}
    H_0:P=P_0\textrm{ vs }H_1:\widetilde{W}_1(P,P_0)\geq\min_{v\in\mathcal{V}}\widetilde{W}_1(P_v,P_0).
\end{align*}
Then the test will also achieve Type I and Type II errors at most $\alpha$ and $\beta$ for \eqref{eq:hard_testing} as well, which is a contradiction. Thus,
\begin{align*}
    \rho_n^{\mathrm{gof}}(\mathbb{T}^d,B,\alpha,\beta)\geq \min_{v\in\mathcal{V}}\widetilde{W}_1(P_v,P_0)\gtrsim \rho_{n,J},
\end{align*}
if the condition of Lemma~\ref{lem:hard_family} holds. Therefore, it is enough to show that there exists $J$ such that
\begin{align*}
J&\geq\lceil \log_2(8C+4)\rceil\\
\frac{(8C+2)^dC_2^dc_{\alpha,\beta}S_{1/2}(J)S_1^{-1/4}(J)}{2\sqrt{n}}&\leq (B-1)\wedge1\\
\rho_{n,J}&\gtrsim \begin{cases}
        n^{-1/2},&d<4\\
        n^{-1/2}\log^{1/4}n ,&d=4\\
        n^{-2/d},&d>4.
    \end{cases}
\end{align*}
We first show the statement when $n$ is sufficiently large.

\paragraph{Case $d<4$.} Set $J=\lceil\log_2(8C+4)\rceil$. Then $S_{1/2}(J)$ and $S_1(J)$ are constants. Thus, the condition of Lemma~\ref{lem:hard_family}
\begin{align*}
\frac{(8C+2)^dC_2^dc_{\alpha,\beta}S_{1/2}(J)S_1^{-1/4}(J)}{2\sqrt{n}}\leq (B-1)\wedge1
\end{align*}
holds if $n$ is sufficiently large. Thus,
\begin{align*}
\rho_{n,J}\asymp n^{-1/2}.
\end{align*}
\paragraph{Case $d=4$.} In this case,
\begin{align*}
    S_1(J)&=J\\
    S_{1/2}(J)&=\sum_{j=1}^J2^j\leq 2^{J+1}.
\end{align*}
Set $J=\left\lfloor \frac12\log_2 (4n)\right\rfloor$. Then
\begin{align*}
    \frac{S_{1/2}(J)S_1^{-1/4}(J)}{\sqrt n}
    \leq
    \frac{2^{J+1}}{\sqrt n J^{1/4}}
    \leq 4J^{-1/4}.
\end{align*}
Thus,
\begin{align*}
    \frac{(8C+2)^dC_2^dc_{\alpha,\beta}S_{1/2}(J)S_1^{-1/4}(J)}{2\sqrt{n}}\leq \frac{2(8C+2)^dC_2^dc_{\alpha,\beta}}{\left\lfloor \frac12\log_2 (4n)\right\rfloor^{1/4}}\leq (B-1)\wedge1
\end{align*}
holds if $n$ is sufficiently large. Also,
\begin{align*}
J=\left\lfloor \frac12\log_2 (4n)\right\rfloor\geq\lceil \log_2(8C+4)\rceil
\end{align*}
holds if $n$ is sufficiently large. Since
$J\asymp\log n$, we obtain
\begin{align*}
    \rho_{n,J}\asymp\frac{J^{1/4}}{\sqrt n}
    \asymp
    \frac{\log_2^{1/4}(4n)}{\sqrt n}\asymp\frac{\log^{1/4}n }{\sqrt n}.
\end{align*}
\paragraph{Case $d>4$.} In this case
\begin{align*}
    S_1(J)&\asymp 2^{(d-4)J},\\
    S_{1/2}(J)&\asymp 2^{(d-2)J/2},
\end{align*}
and hence
\begin{align*}
    S_{1/2}(J)S_1^{-1/4}(J)
    \asymp 2^{dJ/4}.
\end{align*}
For $K>0$, set $J=\left\lfloor \frac{2}{d}\log_2 n\right\rfloor-K$. Then
\begin{align*}
    \frac{S_{1/2}(J)S_1^{-1/4}(J)}{\sqrt n}
    \asymp
    \frac{2^{dJ/4}}{\sqrt n}\asymp 2^{-Kd/4}.
\end{align*}
Thus, the conditions
\begin{align*}
J&\geq\lceil \log_2(8C+4)\rceil\\
\frac{(8C+2)^dC_2^dc_{\alpha,\beta}S_{1/2}(J)S_1^{-1/4}(J)}{2\sqrt{n}}&\leq (B-1)\wedge1,
\end{align*}
hold if $n$ and $K$ are sufficiently large. Plugging $J=\left\lfloor \frac{2}{d}\log_2 n\right\rfloor-K$, we obtain
\begin{align*}
    \rho_{n,J}\asymp\frac{S_1^{1/4}(J)}{\sqrt{n}}\asymp n^{-2/d}.
\end{align*}

Let $n_0$ be the positive integer such that the statement holds if $n\geq n_0$. It remains to consider the case when $n< n_0$. We use the fact that the minimax separation radius is non-increasing in the sample size: indeed, any test based on $n$ observations can be used with more observations by ignoring the additional samples. Hence, for $n<n_0$,
\begin{align*}
    \rho_n^{\mathrm{gof}}(\mathbb T^d,B,\alpha,\beta)
    \geq
    \rho_{n_0}^{\mathrm{gof}}(\mathbb T^d,B,\alpha,\beta).
\end{align*}
Define
\begin{align*}
    r_d(n)\coloneqq
    \begin{cases}
        n^{-1/2}, & d<4,\\
        n^{-1/2}\log^{1/4}n , & d=4,\\
        n^{-2/d}, & d>4.
    \end{cases}
\end{align*}
Since $r_d(n)>0$ for every $n\geq e$,
\begin{align*}
\rho_{n_0}^{\mathrm{gof}}(\mathbb T^d,B,\alpha,\beta)&\gtrsim r_d(n_0)\\
&\geq\frac{r_d(n_0)}{\max_{e\leq n^\prime<n_0}r_d(n^\prime)} r_d(n)\\
&\gtrsim r_d(n).
\end{align*}
This yields, for
all $n\geq e$,
\begin{align*}
    \rho_n^{\mathrm{gof}}(\mathbb T^d,B,\alpha,\beta)
    \gtrsim
    \begin{cases}
        n^{-1/2}, & d<4,\\
        n^{-1/2}\log^{1/4}n , & d=4,\\
        n^{-2/d}, & d>4.
    \end{cases}
\end{align*}
\subsection{Proof of Lemma~\ref{lem:hard_family}}
\subsubsection{Construction of the distributions $P_v$}
We use 2-regular Daubechies wavelets on $\mathbb{R}^d$. By Definition~\ref{def:Daubechies_wavelet}(1), $\psi^0$ and $\psi^1$ are supported on $[-2C,2C]$. For $j\geq0$, denote $K_j\coloneqq  \left\{k\in\mathbb{Z}^d:2^{-j}([-2C,2C]^d+k)\subset [0,1)^d\right\}$. Then $\psi^{\epsilon}_{j,k}$ is supported on $[0,1)^d$ if $k\in K_j$.

Denote $j_0\coloneqq  \lceil \log_2(8C+4)\rceil$. For $J\geq j_0$, denote $I\coloneqq  \left\{(j,k):j_0\leq j\leq J,k\in K_j\right\}$. Then for $v\in\{-1,+1\}^I$ and $\delta_1,\ldots,\delta_J>0$, define a perturbed density $p_v$ as follows.
\begin{align*}
    p_v(x)\coloneqq  p_0(x)+\sum_{j=j_0}^J\delta_j\sum_{k\in K_j}v_{j,k}\psi^{\mathbf{1}_d}_{j,k}(x).
\end{align*}
Note that $\int_{\mathbb{R}}\psi^1(x)dx=0$ by Definition~\ref{def:Daubechies_wavelet}(3). Thus,
\begin{align*}
    \int_{[0,1)^d}\psi^{\mathbf{1}_d}_{j,k}(x)dx&=\int_{\mathbb{R}^d}\psi^{\mathbf{1}_d}_{j,k}(x)dx\\
    &=2^{jd/2}\int_{\mathbb{R}^d}\psi^1(2^jx_1-k_1)\cdots\psi^1(2^jx_d-k_d)dx_1\cdots dx_d\\
    &=2^{jd/2}\left(\int_{\mathbb{R}}\psi^1(2^jx_1-k_1)dx_1\right)\cdots\left(\int_{\mathbb{R}}\psi^1(2^jx_d-k_d) dx_d\right)\\
    &=0.
\end{align*}
For $j=1,\ldots,J$, we set
\begin{align*}
    \delta_j=\frac{c_{\alpha,\beta}}{\sqrt{n}2^{j+1}}S_1^{-1/4}(J)
\end{align*}
where $c_{\alpha,\beta}=\log^{1/4} (1+(1-\alpha-\beta)^2)$.

Next, we show $p_v(x)\in [0,B]$. For $x\in[0,1)^d$, the following holds.
\begin{align*}
    |p_v(x)-1|&\leq\sum_{j=j_0}^J\delta_j\sum_{k\in K_j}\left|\psi^{\mathbf{1}_d}_{j,k}(x)\right|\\
    &\leq \sum_{j=j_0}^J\delta_j\left\lVert \psi^{\mathbf{1}_d}_{j,k}\right\rVert_{\infty}\left|\left\{k\in\mathbb{Z}^d:x\in\mathrm{supp}\left(\psi^{\mathbf{1}_d}_{j,k}\right)\right\}\right|.
\end{align*}
We use the following lemma to bound the cardinality of the set $\left\{k\in\mathbb{Z}^d:x\in \mathrm{supp}\left(\psi^{\mathbf{1}_d}_{j,k}\right)\right\}$.
\begin{lemma}\label{lem:overlap}
    $\mathrm{supp}\left(\psi^{\mathbf{1}_d}_{j,k}\right)\cap \mathrm{supp}\left(\psi^{\mathbf{1}_d}_{j,k^\prime}\right) \neq \emptyset$ only if $k-k^\prime\in [-4C,4C]^d$. Thus, for $j\in\mathbb{N}$ and $x\in [0,1)^d$, the number of overlaps $\left|\left\{k\in\mathbb{Z}^d:x\in \mathrm{supp}\left(\psi^{\mathbf{1}_d}_{j,k}\right)\right\}\right|\leq (8C+2)^d$.
\end{lemma}
Thus,
\begin{align*}
   |p_v(x)-1|&\leq (8C+2)^d\sum_{j=j_0}^J\delta_j\left\lVert \psi^{\mathbf{1}_d}_{j,k}\right\rVert_{\infty}\\
    &\leq(8C+2)^dC_2^d\sum_{j=j_0}^J2^{dj/2}\delta_j\\
    &=(8C+2)^dC_2^d\sum_{j=j_0}^J\frac{2^{dj/2}c_{\alpha,\beta}}{\sqrt{n}2^{j+1}}S_1^{-1/4}(J)\\
    &\leq\frac{(8C+2)^dC_2^dc_{\alpha,\beta}S_{1/2}(J)S_1^{-1/4}(J)}{2\sqrt{n}},
\end{align*}
where the second inequality holds by Definition~\ref{def:Daubechies_wavelet}(2) and the definition of $\psi^{\epsilon}_{j,k}$. By the condition,
\begin{align*}
\frac{(8C+2)^dC_2^dc_{\alpha,\beta}S_{1/2}(J)S_1^{-1/4}(J)}{2\sqrt{n}}\leq (B-1)\wedge1.
\end{align*}
Thus, $p_v$ is a probability density supported on $[0,1)^d$ upper bounded by $B$. Therefore, we can identify $p_v$ with a probability density in $\mathcal{P}_B(\mathbb{T}^d)$.

\subsubsection{Proof of Lemma~\ref{lem:hard_family}(1)}
We first bound the chi-square divergence between $P_0^{\otimes n}$ and $\overline{P}_n$. The chi-square divergence is defined as $\chi^2(P||Q)=\mathbb{E}_Q\left[\left(\frac{dP}{dQ}-1\right)^2\right]$.
\begin{align*}
    1+\chi^2(\overline{P}_n||P_0^{\otimes n})&=1+\int \left(\frac{d\overline{P}_n}{dP_0^{\otimes n}}-1\right)^2dP_0^{\otimes n}\\
    &=\int \left(\frac{d\overline{P}_n}{dP_0^{\otimes n}}\right)^2dP_0^{\otimes n}\\
    &=\mathbb{E}_0\left[\left(\frac{d\overline{P}_n}{dP_0^{\otimes n}}\right)^2\right]\\
    &=\mathbb{E}_{v}\left[\mathbb{E}_{v^\prime}\left[\mathbb{E}_0\left[\left(\frac{dP_v^{\otimes n}}{dP_0^{\otimes n}}\right)\left(\frac{dP_{v^\prime}^{\otimes n}}{dP_0^{\otimes n}}\right)\right]\right]\right].
\end{align*}
We use the fact that
\begin{align*}
    \mathbb{E}_0\left[\left(\frac{dP_v^{\otimes n}}{dP_0^{\otimes n}}\right)\left(\frac{dP_{v^\prime}^{\otimes n}}{dP_0^{\otimes n}}\right)\right]&=\mathbb{E}_0\left[\left(\frac{dP_v}{dP_0}\right)\left(\frac{dP_{v^\prime}}{dP_0}\right)\right]^n\\
    &=\left(1+\sum_{j=j_0}^J\delta_j^2\sum_{k\in K_j}v_{j,k}v_{j,k}^\prime\right)^n,
\end{align*}
since $\psi^\epsilon_{j,k}$ are orthogonal to each other and $\int \psi^1(x)dx=0$. Then
\begin{align*}
    \mathbb{E}_{v}\left[\mathbb{E}_{v^\prime}\left[\mathbb{E}_0\left[\left(\frac{dP_v^{\otimes n}}{dP_0^{\otimes n}}\right)\left(\frac{dP_{v^\prime}^{\otimes n}}{dP_0^{\otimes n}}\right)\right]\right]\right]&=\mathbb{E}_{v}\left[\mathbb{E}_{v^\prime}\left[\left(1+\sum_{j=j_0}^J\delta_j^2\sum_{k\in K_j}v_{j,k}v_{j,k}^\prime\right)^n\right]\right]\\
    &\leq \mathbb{E}_{v}\left[\mathbb{E}_{v^\prime}\left[\mathrm{exp}\left(n\left(\sum_{j=j_0}^J\delta_j^2\sum_{k\in K_j}v_{j,k}v_{j,k}^\prime\right)\right)\right]\right]\\
    &=\prod_{j=j_0}^J\prod_{k\in K_j}\mathbb{E}_v\left[\mathbb{E}_{v^\prime}\left[\mathrm{exp}\left(n\left(\delta_j^2v_{j,k}v_{j,k}^\prime\right)\right)\right]\right]\\
    &=\prod_{j=j_0}^J\prod_{k\in K_j}\mathbb{E}_v\left[\mathrm{cosh}\left(n\delta_j^2v_{j,k}\right)\right]\\
    &\leq \prod_{j=j_0}^J\prod_{k\in K_j}\mathrm{exp}\left(n^2\delta_j^4\right)\\
    &=\mathrm{exp}\left(\sum_{j=j_0}^J|K_j|n^2\delta_j^4\right)
\end{align*}
where the first inequality holds since $1+x\leq e^{x}$ and the second inequality holds since $\cosh x\leq e^{x^2}$. Also, the cardinality of $K_j$ is bounded by the following lemma.
\begin{lemma}\label{lem:size_Kj}
For $j\in\mathbb{N}$, if $j\geq \log_2(8C+4)$ then $2^{d(j-1)}\leq |K_j|\leq 2^{d(j+1)}$.
\end{lemma}
Plugging $\delta_j=\frac{c_{\alpha,\beta}}{\sqrt{n}2^{j+1}}S_1^{-1/4}(J)$ with $c_{\alpha,\beta}=\log^{1/4} (1+(1-\alpha-\beta)^2)$ gives
\begin{align*}
    \mathrm{exp}\left(\sum_{j=j_0}^J|K_j|n^2\delta_j^4\right)\leq \mathrm{exp}\left(c_{\alpha,\beta}^4\right)=1+(1-\alpha-\beta)^2.
\end{align*}
Therefore,
\begin{align*}
\lVert P_0^{\otimes n}-\overline{P}_n\rVert_{\mathrm{TV}}\leq\sqrt{\frac{1}{2}\chi^2(\overline{P}_n||P_0^{\otimes n})}< 1-\alpha-\beta.
\end{align*}
\subsubsection{Proof of Lemma~\ref{lem:hard_family}(2)}
\paragraph{Lower bound on the $\dot{W}^{-1,2}$ norm.} We use the following lemma to establish the lower bound.
\begin{lemma}\label{lem:Hs_to_L2}Let $f\in W^{1,2}(\mathbb{R}^d)$ and $V_j$, $j\in\mathbb{Z}$ be an $r(>1)$-regular multiresolution approximation of $L^2(\mathbb{R}^d)$. Let $\{\phi_k\}$ be an orthonormal basis of $V_0$ and for $j\geq0$, let $\{\psi_{j,k}\}$ be an orthonormal basis of $W_j$. Then up to a universal constant factor, the following holds.
\begin{align}\label{eq:sobolev_norm_decomp}
    \lVert f\rVert_{W^{1,2}(\mathbb{R}^d)}^2\asymp \sum_{k}|\langle f,\phi_k\rangle|^2+\sum_{j\geq 0}\sum_{k}4^{j}|\langle f,\psi_{j,k}\rangle|^2.
\end{align}
\end{lemma}
For a given $v\in\{-1,+1\}^I$, define a function $g_v$ as follows.
\begin{align*}
    g_v=\sum_{j=j_0}^J4^{-j}\delta_j\sum_{k\in K_j}v_{j,k}\psi^{\mathbf{1}_d}_{j,k}
\end{align*}
Since the support of $g_v$ is contained in $[0,1)^d$, we can identify $g_v$ with an element of $\dot{W}^{1,2}(\mathbb{T}^d)$. Then, by the definition of $\dot{W}^{-1,2}(\mathbb{T}^d)$, we obtain the following lower bound.
\begin{align*}
    \lVert p_v-p_0\rVert_{\dot{W}^{-1,2}(\mathbb{T}^d)}&=\sup_{g\in \dot{W}^{1,2}(\mathbb{T}^d),\lVert g\rVert_{\dot{W}^{1,2}(\mathbb{T}^d)}\leq 1}\langle p_v-p_0,g\rangle_{L^2(\mathbb{T}^d)}\\
    &\geq\frac{\langle p_v-p_0,g_v\rangle_{L^2(\mathbb{T}^d)}}{\lVert g_v\rVert_{\dot{W}^{1,2}(\mathbb{T}^d)}}\\
    &=\frac{\langle p_v-p_0,g_v\rangle_{L^2(\mathbb{R}^d)}}{\lVert g_v\rVert_{\dot{W}^{1,2}(\mathbb{R}^d)}}.
\end{align*}
Recall that
\begin{align*}
    p_v-p_0=\sum_{j=j_0}^J\delta_j\sum_{k\in K_j}v_{j,k}\psi_{j,k}^{\mathbf{1}_d}.
\end{align*}
Since the wavelets form an orthonormal basis, the numerator is given as follows.
\begin{align*}
    \langle p_v-p_0,g_v\rangle_{L^2(\mathbb{R}^d)}=\sum_{j=j_0}^J4^{-j}\delta_j^2|K_j|.
\end{align*}
Also, Lemma~\ref{lem:Hs_to_L2} gives
\begin{align*}
    \lVert g_v\rVert_{W^{1,2}(\mathbb{R}^d)}^2\asymp \sum_{j=j_0}^J4^{-j}\delta_j^2|K_j|.
\end{align*}
Since $\lVert g_v\rVert_{W^{1,2}(\mathbb{R}^d)}^2\geq \lVert g_v\rVert_{\dot{W}^{1,2}(\mathbb{R}^d)}^2$, we obtain the following.
\begin{align*}
    \frac{\langle p_v-p_0,g_v\rangle_{L^2(\mathbb{R}^d)}}{\lVert g_v\rVert_{\dot{W}^{1,2}(\mathbb{R}^d)}}&\gtrsim \sqrt{\sum_{j=j_0}^J4^{-j}\delta_j^2|K_j|}\\
    &\geq\sqrt{\sum_{j=j_0}^J4^{-j}\delta_j^22^{d(j-1)}}\\
    &\asymp \frac{c_{\alpha,\beta}S_1^{1/4}(J)}{\sqrt{n}}
\end{align*}
as $|K_j|\geq 2^{d(j-1)}$ by Lemma~\ref{lem:size_Kj} and $\delta_j\asymp\frac{c_{\alpha,\beta}}{\sqrt{n}2^{j+1}}S_1^{-1/4}(J)$.

\paragraph{Upper bound on the $\dot{W}^{-1,4}$ norm.} We use the following lemma.
\begin{lemma}\label{lem:Lp_norm_wavelet}
Let $\{\psi_{j,k}^\epsilon\}$ be the Daubechies wavelet of $L^\infty(\mathbb{R}^d)$, and let $f\in L^{\infty}(\mathbb{R}^d)$ have the wavelet expression
\begin{align*}
f=\sum_{j=0}^\infty\sum_{k\in K_j}\sum_{\epsilon\neq\mathbf{0}_d}\gamma_{j,k}^\epsilon\psi_{j,k}^\epsilon.
\end{align*}
Suppose that $\int_{[0,1)^d}f(x)dx=0$ and $\sum_{\epsilon\neq\mathbf{0}_d}|\langle f,\psi_{j,k}^\epsilon\rangle|^2\leq a_j^2$ for every $j\geq0$ and $k\in K_j$, for some nonnegative sequence $\{a_j\}_{j\geq0}$. Then
\begin{align*}
\lVert f\rVert_{\dot{W}^{-1,4}(\mathbb{T}^d)}^4\lesssim\left(\sum_{j=0}^\infty a_j^22^{jd}4^{-j}\right)^2.
\end{align*}
\end{lemma}
Since $\sum_{\epsilon\neq\mathbf{0}_d}|\langle p_v-p_0,\psi_{j,k}^\epsilon\rangle|^2=\delta_j^2$, we have
\begin{align*}
    \lVert p_v-p_0\rVert_{\dot{W}^{-1,4}(\mathbb{T}^d)}^4&\lesssim \left(\sum_{j=j_0}^J \delta_j^22^{dj}4^{-j}\right)^2 \\
    &\asymp \frac{c_{\alpha,\beta}^4S_1(J)}{n^2}.
\end{align*}
\subsection{Proof of Lemma~\ref{lem:W1_lb}}
We adapt the lower-bound argument from Section 6 of \cite{Bobkov2021simple}. Let $f\coloneqq p-q$, and for $m\in\mathbb{Z}^d$, define its Fourier coefficient by
\begin{align*}
    \widehat f(m)
    \coloneqq \int_{\mathbb{T}^d} f(x)e^{2\pi i\langle m,x\rangle}dx= \mathbb{E}_P[e^{2\pi i\langle m,X\rangle}]- \mathbb{E}_Q[e^{2\pi i\langle m,X\rangle}].
\end{align*}
Define $h$ by the following multiple Fourier series:
\begin{align*}
h(x)\coloneqq  \frac{1}{4\pi^2}\sum_{m\neq\mathbf{0}_d}\frac{\widehat f(m)}{\lVert m\rVert_2^2}e^{-2\pi i\langle m,x\rangle},
\end{align*}
in $W^{2,2}(\mathbb{T}^d)$. Since $f\in L^2(\mathbb{T}^d)$, the series converges in $W^{2,2}(\mathbb{T}^d)$, so $h$ is well-defined. Then
\begin{align*}
    \Delta h(x)&=-\sum_{m\neq\mathbf{0}_d}\widehat f(m)e^{-2\pi i\langle m,x\rangle}.
\end{align*}
Also, by the Fourier inversion,
\begin{align*}
    p-q&=\sum_{m}\widehat f(m)e^{-2\pi i\langle m,x\rangle}\\
    &=\sum_{m\neq\mathbf{0}_d}\widehat f(m)e^{-2\pi i\langle m,x\rangle}
\end{align*}
in $L^2(\mathbb{T}^d)$ since $\widehat{f}(\mathbf{0}_d)=\mathbb{E}_P[e^{2\pi i\langle \mathbf{0}_d,X\rangle}]-\mathbb{E}_Q[e^{2\pi i\langle \mathbf{0}_d,X\rangle}]=0$. Consequently, if $\lVert\nabla v\rVert_{L^2}<\infty$ then
\begin{align*}
\int_{\mathbb{T}^d}v(dP-dQ)=-\int_{\mathbb{T}^d}v\Delta h dx=\int_{\mathbb{T}^d}\langle \nabla v,\nabla  h\rangle dx
\end{align*}
holds. Thus, by the Kantorovich-Rubinstein duality, we have the following:
\begin{align*}
    \widetilde{W}_1(P,Q)&=\sup_{\lVert u\rVert_{\mathrm{Lip}}\leq 1}\int_{\mathbb{T}^d}u(dP-dQ)\\
    &=\sup_{\lVert u\rVert_{\mathrm{Lip}}\leq 1}\int_{\mathbb{T}^d}\langle\nabla u,\nabla h\rangle dx.
\end{align*}
By choosing an appropriate Lipschitz function, we can derive the bound. We choose the Lipschitz function guaranteed by the following lemma.
\begin{lemma}[{\citealp[Lemma~5.1]{ambrosio2019pde}}]\label{lem:Lusin_approx}
There exists $K_d>0$, depending only on $d$, such that for every $f\in W^{1,2}(\mathbb{T}^d)$ and $\alpha>0$, one can find a function $u_\alpha:\mathbb{T}^d\to\mathbb{R}$ that is $\alpha$-Lipschitz with respect to $d_{\mathbb{T}^d}$ and satisfies
\begin{align*}
\lambda(\{f\neq u_\alpha\})\leq\frac{K_d}{\alpha^2}\int_{\mathbb{T}^d}\|\nabla f(x)\|_2^2\,dx.
\end{align*}
\end{lemma}
Since $h\in W^{2,2}(\mathbb{T}^d)\subset W^{1,2}(\mathbb{T}^d)$, Lemma~\ref{lem:Lusin_approx} can be applied. Thus, there exists an $\alpha$-Lipschitz function $u_\alpha:\mathbb{T}^d\rightarrow\mathbb{R}$ and a constant $K$ depending on $d$ such that
\begin{align*}
    \lambda (\{h\neq u_\alpha\})\leq \frac{K}{\alpha^2}\int_{[0,1)^d}\lVert\nabla h\rVert_2^2dx.
\end{align*}
Since $h,u_\alpha\in W^{1,2}(\mathbb T^d)$, the locality property of weak derivatives \citep[see, e.g.,][Theorem 4.4]{evans2015measure} implies that $\nabla h=\nabla u_\alpha$ almost everywhere on ${h=u_\alpha}$. Then we have
\begin{align*}
    \alpha \widetilde{W}_1(P,Q)&\geq  \int_{\mathbb{T}^d} u_\alpha(dP-dQ)\\
    &=\int_{\mathbb{T}^d}\langle \nabla u_\alpha,\nabla h\rangle dx\\
    &=\int_{\mathbb{T}^d}\lVert \nabla h\rVert_2^2 dx-\int_{\{h\neq u_\alpha\}}\langle \nabla h ,\nabla h-\nabla u_\alpha\rangle dx.
\end{align*}
Also,
\begin{align*}
    &\left|\int_{\{h\neq u_\alpha\}}\langle \nabla h ,\nabla h-\nabla u_\alpha\rangle dx\right|\\
    &\leq \int_{\{h\neq u_\alpha\}}\langle \nabla h ,\nabla h\rangle dx+\int_{\{h\neq u_\alpha\}}|\langle \nabla h ,\nabla u_\alpha\rangle| dx\\
    &\leq \int_{\{h\neq u_\alpha\}}\lVert\nabla h\rVert_2^2dx+\alpha\int_{\{h\neq u_\alpha\}}\lVert\nabla h\rVert_2dx\\
    &\leq \lambda^{1/2}(\{h\neq u_\alpha\})\left(\int_{\{h\neq u_\alpha\}}\lVert\nabla h\rVert_2^4dx\right)^{1/2}+\alpha\lambda^{3/4}(\{h\neq u_\alpha\})\left(\int_{\{h\neq u_\alpha\}}\lVert\nabla h\rVert_2^4dx\right)^{1/4}.
\end{align*}
Now denote
\begin{align*}
    m_2&\coloneqq  \int_{[0,1)^d}\lVert\nabla h\rVert_2^2dx\\
    m_4&\coloneqq  \int_{[0,1)^d}\lVert\nabla h\rVert_2^4dx.
\end{align*}
Then
\begin{align*}
    \left|\int_{\{h\neq u_\alpha\}}\langle \nabla h ,\nabla h-\nabla u_\alpha\rangle dx\right|\leq K^{1/2}\alpha^{-1}m_2^{1/2}m_4^{1/2}+K^{3/4}\alpha^{-1/2}m_2^{3/4}m_4^{1/4}.
\end{align*}
Thus,
\begin{align*}
\alpha \widetilde{W}_1(P,Q)\geq m_2-K^{1/2}\alpha^{-1}m_2^{1/2}m_4^{1/2}-K^{3/4}\alpha^{-1/2}m_2^{3/4}m_4^{1/4}.
\end{align*}
Set $\alpha=M\sqrt{m_4/m_2}$ with $M=4\max\{K^{1/2},K^{3/2}\}$. Then
\begin{align*}
    \alpha \widetilde{W}_1(P,Q)&\geq  m_2-K^{1/2}\alpha^{-1}m_2^{1/2}m_4^{1/2}-K^{3/4}\alpha^{-1/2}m_2^{3/4}m_4^{1/4}\\
    &=m_2\left(1-\frac{K^{1/2}}{M}-\frac{K^{3/4}}{\sqrt{M}}\right)\\
    &\geq \frac{1}{4}m_2.
\end{align*}
Therefore,
\begin{align*}
   \widetilde{W}_1(P,Q)\geq \frac{m_2^{3/2}}{16\max\{K^{1/2},K^{3/2}\}m_4^{1/2}}.
\end{align*}
It remains to bound $m_2$ and $m_4$ in terms of the corresponding negative-order Sobolev norms. We first show $m_2=\lVert p-q\rVert_{\dot{W}^{-1,2}(\mathbb{T}^d)}^2$.
\begin{align*}
        \lVert p-q\rVert_{\dot{W}^{-1,2}(\mathbb{T}^d)}&=\sup\left\{\left|\int_{\mathbb{T}^d} (p-q)(x)g(x)dx\right|:\lVert g\rVert_{\dot{W}^{1,2}(\mathbb{T}^d)}\leq 1\right\}\\
        &\leq \sup\left\{\left|\int_{\mathbb{T}^d} \langle \nabla h,\nabla g\rangle dx\right|:\lVert g\rVert_{\dot{W}^{1,2}(\mathbb{T}^d)}\leq 1\right\}\\
        &\leq \left(\int_{\mathbb{T}^d}\lVert \nabla h\rVert_2^2dx\right)^{1/2}.
\end{align*}
Also,
\begin{align*}
        \lVert p-q\rVert_{\dot{W}^{-1,2}(\mathbb{T}^d)}&=\sup\left\{\left|\int_{\mathbb{T}^d} (p-q)(x)g(x)dx\right|:\lVert g\rVert_{\dot{W}^{1,2}(\mathbb{T}^d)}\leq 1\right\}\\
        &\geq \left|\int_{\mathbb{T}^d} (p-q)(x)\frac{h}{\lVert h\rVert_{\dot{W}^{1,2}(\mathbb{T}^d)}}dx\right|\\
        &=\left|\int_{\mathbb{T}^d} \left\langle\nabla h,\frac{\nabla h}{\lVert h\rVert_{\dot{W}^{1,2}(\mathbb{T}^d)}}\right\rangle dx\right|\\
        &=\lVert h\rVert_{\dot{W}^{1,2}(\mathbb{T}^d)}\\
        &=\left(\int_{\mathbb{T}^d}\lVert \nabla h\rVert_2^2dx\right)^{1/2}.
\end{align*}
Next, we show $m_4\lesssim\lVert p-q\rVert_{\dot{W}^{-1,4}(\mathbb{T}^d)}^4$ up to a constant factor depending on $d$ using the following lemma.
\begin{lemma}[{Torus case; \citealp[Lemma~1.7]{Iacobelli2024stability}}]\label{lem:moment}
Suppose that $p_1,p_2\in L^\infty(\mathbb{T}^d)$ are probability densities. For $i=1,2$, let $h_i$ satisfy $-\Delta h_i=p_i-1$. Then, for every $1<r<\infty$, there exists a constant $B_{d,r}>0$ depending only on $d$ and $r$ such that
\begin{align*}
\|\nabla h_1-\nabla h_2\|_{L^r(\mathbb{T}^d)}\leq B_{d,r}\|p_1-p_2\|_{\dot W^{-1,r}(\mathbb{T}^d)}
\end{align*}
holds.
\end{lemma}
Set
\begin{align*}
U_1(x)&\coloneqq  \frac{1}{4\pi^2}\sum_{m\neq\mathbf{0}_d}\frac{\mathbb E_P[e^{2\pi i\langle m,X\rangle}]}{\lVert m\rVert_2^2}e^{-2\pi i\langle m,x\rangle},\\
U_2(x)&\coloneqq  \frac{1}{4\pi^2}\sum_{m\neq\mathbf{0}_d}\frac{\mathbb E_Q[e^{2\pi i\langle m,X\rangle}]}{\lVert m\rVert_2^2}e^{-2\pi i\langle m,x\rangle}.
\end{align*}
Then
\begin{align*}
    \Delta U_1&=-\sum_{m\neq\mathbf{0}_d}\mathbb E_P[e^{2\pi i\langle m,X\rangle}]e^{-2\pi i\langle m,x\rangle}=1-p\\
    \Delta U_2&=-\sum_{m\neq\mathbf{0}_d}\mathbb E_Q[e^{2\pi i\langle m,X\rangle}]e^{-2\pi i\langle m,x\rangle}=1-q.
\end{align*}
Thus, by Lemma~\ref{lem:moment}, there exists a constant $c_d$ depending on $d$ such that
\begin{align*}
\left(\int_{\mathbb{T}^d}\lVert \nabla h(x)\rVert_2^4dx\right)^{1/4}\leq c_d\lVert p-q\rVert_{\dot{W}^{-1,4}(\mathbb{T}^d)}.
\end{align*}
Therefore,
\begin{align*}
     \widetilde{W}_1(P,Q)\gtrsim \frac{\lVert P-Q\rVert_{\dot{W}^{-1,2}(\mathbb{T}^d)}^3}{\lVert P-Q\rVert_{\dot{W}^{-1,4}(\mathbb{T}^d)}^2}.
\end{align*}
\subsection{Proof of Lemma~\ref{lem:overlap}}
\begin{proof}[Proof of Lemma~\ref{lem:overlap}]
The support of $\psi^1$ is contained in $[-2C,2C]$. Thus, $\mathrm{supp}\left(\psi^{\mathbf{1}_d}_{j,k}\right)$ is contained in $2^{-j}([-2C,2C]^d+k)$. Also, $\mathrm{supp}\left(\psi^{\mathbf{1}_d}_{j,k^\prime}\right)$ is contained in $2^{-j}([-2C,2C]^d+k^\prime)$. Thus, $\mathrm{supp}\left(\psi^{\mathbf{1}_d}_{j,k}\right)\cap \mathrm{supp}\left(\psi^{\mathbf{1}_d}_{j,k^\prime}\right)\neq\emptyset$ only if
\begin{align*}
    2^{-j}([-2C,2C]^d+k) \cap 2^{-j}([-2C,2C]^d+k^\prime)\neq\emptyset.
\end{align*}
This implies $k-k^\prime\in [-4C,4C]^d$. 
\end{proof}
\subsection{Proof of Lemma~\ref{lem:size_Kj}}
The cardinality of $K_j$ is bounded as follows.
    \begin{align*}
        |K_j|&=|\{k\in\mathbb{Z}^d:2^{-j}([-2C,2C]^d+k)\subset [0,1)^d\}|\\
        &=|\{k\in\mathbb{Z}^d:([-2C,2C]^d+k)\subset [0,2^j)^d\}|\\
        &\leq |\{k\in\mathbb{Z}^d:k\in [-2C,2^j+2C)^d\}|\\
        &\leq (2^j+4C+2)^d.
    \end{align*}
If $j\geq \log_2(8C+4)$ then $4C+2\leq 2^{j-1}$ so $(2^j+4C+2)^d\leq 2^{d(j+1)}$. Also,
\begin{align*}
    |K_j|&=|\{k\in\mathbb{Z}^d:([-2C,2C]^d+k)\subset [0,2^j)^d\}|\\
        &\geq|\{k\in\mathbb{Z}^d:k\in [2C,2^j-2C)^d\}|\\
        &\geq(2^j-4C-2)^d\\
        &\geq2^{d(j-1)}.
\end{align*}
\subsection{Proof of Lemma~\ref{lem:Hs_to_L2}}
We use the following proposition. Denote the orthogonal projection from $L^2(\mathbb{R}^d)$ onto $V_0$ and $W_j$ as $E_0$ and $D_j$, respectively.
\begin{proposition}[{\citealp[Theorem~2.8]{meyer1992wavelets}}]\label{prop:sobolev_norm_decomp}
Let $V_j$, $j\in\mathbb{Z}$ be an $r(>1)$-regular multiresolution approximation of $L^2(\mathbb{R}^d)$. If $f\in W^{1,2}(\mathbb{R}^d)$, the following holds up to a universal constant factor.
\begin{align*}
    \lVert f\rVert_{W^{1,2}(\mathbb{R}^d)}\asymp \lVert E_0f\rVert_{L^2(\mathbb{R}^d)}+\sqrt{\sum_{j=0}^\infty 4^{j}\lVert D_jf\rVert_{L^2(\mathbb{R}^d)}^2}.
\end{align*}
\end{proposition}
Now suppose $\{\phi_k\}$ is an orthonormal basis of $V_0$ and $\{\psi_{j,k}\}$ is an orthonormal basis of $W_j$ for $j\in\mathbb{Z}$. Then Proposition~\ref{prop:sobolev_norm_decomp} implies the following.
\begin{align*}
    \lVert f\rVert_{W^{1,2}(\mathbb{R}^d)}^2\asymp \sum_{k}|\langle f,\phi_k\rangle|^2+\sum_{j\geq 0}\sum_{k}4^{j}|\langle f,\psi_{j,k}\rangle|^2.
\end{align*}
\subsection{Proof of Lemma~\ref{lem:Lp_norm_wavelet}}
Setting $s=-1$ in Theorem 3 in Chapter 6.2 of \cite{meyer1992wavelets} gives the following proposition.
\begin{proposition}\label{prop:hom_L_ps}
Let $\{V_j\}_{j\in\mathbb{Z}}$ be an $r(>1)$-regular multiresolution approximation of $L^2(\mathbb{R}^d)$, and let $\{\psi_{j,k}^\epsilon\}_{j\in\mathbb{Z},\,k\in\mathbb{Z}^d,\,\epsilon\in\{0,1\}^d\backslash\{\mathbf{0}_d\}}$ be the corresponding Daubechies wavelet. For $j\in\mathbb{Z}$ and $k\in\mathbb{Z}^d$, define $Q(j,k)\coloneqq  2^{-j}([0,1)^d+k)$. Then, for every $1<q<\infty$ and $f\in L^2(\mathbb{R}^d)$, the following holds.
\begin{align*}
\lVert f\rVert_{\dot{W}^{-1,q}(\mathbb{R}^d)}\asymp \left\lVert \left(\sum_{j\in\mathbb{Z}}\sum_{k\in\mathbb{Z}^d}\left(\sum_{\epsilon\neq\mathbf{0}_d}|\langle f,\psi_{j,k}^\epsilon\rangle|^2\right)2^{jd}4^{-j}\mathbf{1}_{Q(j,k)}\right)^{1/2}\right\rVert_{L^q(\mathbb{R}^d)}.
\end{align*}
\end{proposition}
While the original theorem is stated for the the homogeneous potential-space norm, $\dot{L}^{q,s}(\mathbb{R}^d)$ norm, it is known that the norm is equivalent to $\dot{W}^{s,q}(\mathbb{R}^d)$ norm when $s$ is an integer \citep[e.g., Theorem 5.2.3.1][]{Triebel1983Theory}.

For $\psi_{j,k}^\epsilon$, we use the 2-regular Daubechies wavelets. Proposition~\ref{prop:hom_L_ps} can be applied to $f$ since
\begin{align*}
    \int_{\mathbb{R}^d}|f(x)|^2dx=\int_{[0,1)^d}|f(x)|^2dx\leq \lVert f\rVert_{L^{\infty}(\mathbb{R}^d)}^2<\infty,
\end{align*}
so $f\in L^2(\mathbb{R}^d)$. Next, note that, for each fixed $j$, the cubes $\{Q(j,k):k\in\mathbb{Z}^d\}$ are pairwise disjoint. Since $k\in K_j$ implies $[-2C,2C]^d+k\subset [0,2^j)^d$, we have $Q(j,k)=2^{-j}([0,1)^d+k)\subset [0,1)^d$. Hence,
\begin{align*}
\sum_{k\in K_j}\left(\sum_{\epsilon\neq\mathbf{0}_d}|\langle f,\psi_{j,k}^\epsilon\rangle|^2\right)2^{jd}4^{-j}\mathbf{1}_{Q(j,k)}(x)\leq a_j^22^{jd}4^{-j}\mathbf{1}_{[0,1)^d}(x).
\end{align*}
Applying the preceding theorem with $s=-1$ and $q=4$ yields
\begin{align*}
\lVert f\rVert_{\dot{W}^{-1,4}(\mathbb{R}^d)}^4\lesssim\int_{\mathbb{R}^d}\left(\sum_{j=0}^\infty a_j^22^{jd}4^{-j}\mathbf{1}_{[0,1)^d}(x)\right)^2dx=\left(\sum_{j=0}^\infty a_j^22^{jd}4^{-j}\right)^2.
\end{align*}
Next, we show that $\lVert f\rVert_{\dot{W}^{-1,4}(\mathbb{T}^d)}\lesssim \lVert f\rVert_{\dot{W}^{-1,4}(\mathbb{R}^d)}$. Recall that
\begin{align*}
\lVert f\rVert_{W^{-1,r}(\mathbb{T}^d)}&\coloneqq  \sup_{\substack{           [g]\in\dot W^{1,r'}(\mathbb{T}^d)\\            \|[g]\|_{\dot W^{1,r'}(\mathbb{T}^d)}\leq1}}
|\langle f,[g]\rangle|.
\end{align*}
For $g\in \dot{W}^{1,4/3}(\mathbb{T}^d)$ such that $\lVert g\rVert_{\dot{W}^{1,4/3}(\mathbb{T}^d)}=1$, denote $\overline{g}\coloneqq  \int_{\mathbb{T}^d} gdx$ and let $g^\prime$ be a periodic extension of $g$ on $\mathbb{R}^d$. Let $\chi:\mathbb{R}^d\rightarrow\mathbb{R}$ be a bump function of bounded support such that $\chi\equiv1$ on $[0,1]^d$. Then
\begin{align*}
|\langle f,[g]\rangle|&=\int_{\mathbb{T}^d}fg\\&=\int_{\mathbb{T}^d}f(g-\overline{g})\\
    &=\int_{\mathbb{R}^d}f\chi(g^\prime-\overline{g}),
\end{align*}
where the first equality holds since $\int_{\mathbb{T}^d} f=0$. Note that $\lVert \chi(g^\prime-\overline{g})\rVert_{\dot{W}^{1,4/3}(\mathbb{R}^d)}<\infty$. This can be shown as follows.
\begin{align*}
    \lVert \chi(g^\prime-\overline{g})\rVert_{\dot{W}^{1,4/3}(\mathbb{R}^d)}&=\lVert \nabla(\chi(g^\prime-\overline{g}))\rVert_{L^{4/3}(\mathbb{R}^d)}\\
    &=\lVert(g^\prime-\overline{g})\nabla\chi+\chi\nabla g^\prime\rVert_{L^{4/3}(\mathbb{R}^d)}.
\end{align*}
Then
\begin{align*}
    \lVert(g^\prime-\overline{g})\nabla\chi\rVert_{L^{4/3}(\mathbb{R}^d)}&\lesssim\lVert g-\overline{g}\rVert_{L^{4/3}(\mathbb{T}^d)}\\
    &\lesssim\lVert\nabla g\rVert_{L^{4/3}(\mathbb{T}^d)}\\
    &\lesssim1.
\end{align*}
The first inequality holds by the periodicity of $g^\prime$ and the boundedness of the support of $\chi$. The second inequality follows from the Poincar\'e inequality \citep[see, e.g.,][Theorem 5.8.1.1]{evans2010partial}. Similarly,
\begin{align*}
    \lVert\chi\nabla g^\prime\rVert_{L^{4/3}(\mathbb{R}^d)}&\lesssim\lVert\nabla g\rVert_{L^{4/3}(\mathbb{T}^d)}\\
    &\lesssim1
\end{align*}
holds. Thus,
\begin{align*}
    \lVert f\rVert_{\dot{W}^{-1,4}(\mathbb{R}^d)}\gtrsim \int_{\mathbb{R}^d}f\chi(g^\prime-\overline{g})=\int_{\mathbb{T}^d}fg.
\end{align*}
Taking the supremum over $\{g\in\dot{W}^{1,4/3}(\mathbb{T}^d):\lVert g\rVert_{\dot{W}^{1,4/3}(\mathbb{T}^d)}=1\}$, we obtain
\begin{align*}
    \lVert f\rVert_{\dot{W}^{-1,4}(\mathbb{R}^d)}\gtrsim\lVert f\rVert_{\dot{W}^{-1,4}(\mathbb{T}^d)}.
\end{align*}

\end{document}